\documentclass[12pt,a4paper]{amsart} 

\usepackage[utf8]{inputenc}
\usepackage{verbatim}
\usepackage{lmodern}
\usepackage{textcomp}
\usepackage{amssymb,amscd,amsthm,amsrefs,amsxtra,mathtools}
\usepackage{amsmath}
\usepackage{dsfont,latexsym}
\usepackage{mathrsfs}
\usepackage{fancyhdr,enumerate}
\usepackage{bm}
\usepackage{enumitem}

\usepackage[margin=1.9cm]{geometry}

\usepackage{soul}

\usepackage{color}
\definecolor{rltred}{rgb}{0.75,0,0}
\definecolor{rltgreen}{rgb}{0,0.5,0}
\definecolor{rltblue}{rgb}{0,0,0.75}
\usepackage[
colorlinks=true,
urlcolor=rltblue,       
filecolor=rltgreen,     
citecolor=rltgreen,     
linkcolor=rltred,       
]{hyperref} 

\newtheorem{theorem}{Theorem}[section]
\newtheorem{prop}[theorem]{Proposition}
\newtheorem{corol}[theorem]{Corollary}
\newtheorem{lemma}[theorem]{Lemma}
\theoremstyle{definition}
\newtheorem{defn}[theorem]{Definition}
\newtheorem{rem}[theorem]{Remark}

\numberwithin{equation}{section}
\allowdisplaybreaks[4]

\def\XXint#1#2#3{{\setbox0=\hbox{$#1{#2#3}{\int}$}
		\vcenter{\hbox{$#2#3$}}\kern-.5\wd0}}

\newcommand{\Rn}{\mathbb{R}^n}
\newcommand{\R}{\mathbb{R}}
\newcommand{\N}{\mathbb{N}}

\newcommand{\snr}[1]{\lvert #1\rvert}

\newcommand{\Ec}{\mathcal{E}_K}
\newcommand{\Hc}{\mathcal{H}_K}
\newcommand{\Pc}{\mathcal{P}}

\renewcommand{\Gamma}{\varGamma}

\renewcommand{\epsilon}{\varepsilon}

\begin{document}

\title[Optimal decay]{Optimal decay of heteroclinic solutions \\
of the fractional Allen-Cahn equation \\ with a degenerate potential}

\author[F.~De~Pas]{Francesco De Pas} \address{Francesco De Pas\\  Department of Mathematics and Statistics\\
	University of Western Australia,\\
	35 Stirling Highway, WA 6009 Crawley, Australia}
	\email{\href{mailto:francesco.depas@uwa.edu.au}{francesco.depas@uwa.edu.au}}
	
\author[S.~Dipierro]{Serena Dipierro}  \address{Serena Dipierro\\  Department of Mathematics and Statistics\\
	University of Western Australia,\\
	35 Stirling Highway, WA 6009 Crawley, Australia}
	\email{\href{mailto:serena.dipierro@uwa.edu.au}{serena.dipierro@uwa.edu.au}}

\author[E.~Valdinoci]{Enrico Valdinoci}  \address{Enrico Valdinoci\\Department of Mathematics and Statistics\\
	University of Western Australia,\\
	35 Stirling Highway, WA 6009 Crawley, Australia}
\email{\href{mailto:enrico.valdinoci@uwa.edu.au}{enrico.valdinoci@uwa.edu.au}}

	

\thanks{{\it Aknowledgements.} 
SD and EV are members of the Australian Mathematical Society.
FDP and EV are supported by the Australian Laureate Fellowship FL190100081 ``Minimal
surfaces, free boundaries and partial differential equations''.
SD is supported by the Australian Future Fellowship
FT230100333 ``New perspectives on nonlocal equations''.}

\begin{abstract}
We refine the asymptotic estimates for minimizers of a class of nonlocal energy functionals of the form
\[
\frac{1}{4} \iint_{\R^{2n} \setminus (\R^n \setminus \Omega)^2} \snr{u(x) - u(y)}^2 K(x - y) \,dx\,dy + \int_\Omega W(u(x)) \,dx,
\]
as originally studied in~\cite{DPDV},
and we prove the optimality of our improved bounds.

Here, $W$ denotes a possibly \emph{degenerate} oscillatory double-well potential, satisfying a polynomial control on its second derivative near the wells. The kernel~$K$ belongs to a broad class of measurable functions and is modeled on the one of the fractional Laplacian.
\end{abstract}

\maketitle

\setcounter{equation}{0}\setcounter{theorem}{0}

\setcounter{tocdepth}{1} 
\begin{center}
	\begin{minipage}{11cm}
\footnotesize
		\tableofcontents
	\end{minipage}
\end{center}

\section{Introduction}
\subsection{Problem setting}
In this paper we deal with  the decay estimates for minimizers of an energy functional related to phase transition phenomena
with long-range particle interactions. Specifically, we are interested in functionals of the form
\begin{equation}\label{main_fun}
  \Ec(u;\Omega) := \Hc(u,\Omega)+\Pc(u,\Omega),
\end{equation}
where the nonlocal interaction term~$\Hc$ and the potential term~$\Pc$ are given, respectively, by
\begin{equation*}
\Hc(u,\Omega):=  \frac{1}{4}\iint_{\R^{2n}\setminus (\Rn \setminus \Omega)^2}\snr{u(x)-u(y)}^2 {K}(x-y) \, dx\, dy
\end{equation*}
and
\begin{equation}\label{pot_term}
\Pc(u,\Omega) := \int_\Omega W(u(x)) \,dx.
\end{equation}

Here, $K$ is a positive kernel modeled on that of the fractional Laplacian, while~$W$ is a double well potential, with wells at~$\pm1$. Also, differently from the classical literature on this topic,
the derivatives of~$W$, up to any integer order, are allowed to vanish at~$\pm1$.

More precisely, ${K}: \Rn \to [0,+\infty]$ is a measurable function satisfying
\begin{align}
K(x)= K(-x) \qquad \text{for a.~\!e.}~x\in \R^n \label{krn_symm}\tag{K1}
\end{align} and
\begin{align}
\frac{\lambda}{|x|^{n+2s}}\leq K(x) \leq \frac{\Lambda}{|x|^{n+2s}} \qquad \text{for a.~\!e.}~x \in \R^n, \label{newbound}\tag{K2} \end{align}
for some~$s\in(0,1)$ and~$0 <\lambda \leq \Lambda$.

We will also assume that~$K$ satisfies the following 
''slow oscillation'' assumption:
\begin{equation}
\limsup_{j \to +\infty}\left(\sup_{x \in \R^n \setminus \{0\}} \frac{K(\sigma_j x)}{K(x)} -1\right) \frac{1}{(1-\sigma_j)^{1-\epsilon}}\in [-\infty, 0], \quad\text{for any}~\sigma_j \nearrow 1~\text{and}~\epsilon \in (0,1) \ \label{nuovissima}\tag{K3}.
\end{equation}
We refer the reader to~\cite[Appendix A]{DPDV} for examples of kernels satisfying these hypotheses.

Regarding the potential~$W: \R \to [0,+\infty)$, we assume that\footnote{As customary, for every~$k \in \N$ and~$\theta \in (0,1]$ we use the norm notation
\begin{equation*}
[u]_{C^{\theta}(\Omega)}:= \sup_{\substack{x,y \in \Omega\\ x \neq y} } \frac{|u(x)-u(y)|}{|x-y|^{\theta}}
\end{equation*}
and \begin{equation*}
\Vert u \Vert_{C^{k,\theta}(\Omega)}:= \sum_{i=0}^{k}\Vert u^{(i)}\Vert_{L^{\infty}(\Omega)} + [u^{(k)}]_{C^{\theta}(\Omega)}.
\end{equation*}
A function is said to belong to~$C^{k,\theta}(\Omega)$ if~$\Vert u \Vert_{C^{k,\theta}(\Omega)}<+\infty$.

Also, when writing~$C^\vartheta(\Omega)$, we suppose that, if~$\vartheta>1$, the notation~$u\in C^\vartheta(\Omega)$ means~$u\in C^{k,\theta}(\Omega)$ with~$k\in\mathbb{N}$, $\theta\in(0,1]$ and~$\vartheta=k+\theta$.} 
   \begin{align}
   &W\in C_{\rm loc}^{2,\vartheta}(\R) \ \mbox{for some~$\vartheta>0$, with} \  W(\pm1)=W\rq{}(\pm1)=0 \label{pot_reg} \tag{W1}
   \\
   &{\mbox{and }}\ W(t)>0 \quad \mbox{for all } t \in (-1,1).\label{pot_zero}\tag{W2}
   \end{align}
Also, we ask that there exist~$c_2 \geq c_1 >0$, $c_4\ge c_3>0$,
$\mu \in (0,1)$, $\alpha \geq \beta \geq 2$ and~$\gamma \geq \delta \geq 2$ such that
   \begin{equation}\label{pot_deg}\tag{W3}
      \begin{cases}
       c_1 (1+t)^{\alpha-2} \leq W\rq{}\rq{}(t) \leq c_2 (1+t)^{\beta-2}& \text{for } t \in (-1,-1+\mu] \\ \mbox{and} \\
        c_3 (1-t)^{\gamma-2} \leq W\rq{}\rq{}(t) \leq c_4 (1-t)^{\delta-2} & \text{for } t  \in [1-\mu,1). 
   \end{cases}   
\end{equation}

We stress that condition~\eqref{pot_deg} is very general and, for instance, allows~$W$ to be~\textit{degenerate}\footnote{Throughout this work, given a double well potential~$V$ with wells at~$a$ and~$b$, we call it~\textit{non-degenerate} if~$V''(a)>0$ and~$V''(b)>0$. If this condition is not satisfied, we call it~\textit{degenerate}.} and also to present an oscillatory behavior near the wells.

\medskip

Specifically, in this work we address the open question left by the authors in~\cite[Remark 1.8]{DPDV}, concerning the improvement of decay estimates for minimizers in the presence of an oscillatory potential. In~\cite{DPDV}, the authors prove, in the same setting as ours, the existence of a particular class of minimizers for the energy functional~\eqref{main_fun}. In particular, in Remark~1.7 they establish the optimality of their decay estimates in the case of a (possibly degenerate) potential satisfying~\eqref{pot_deg} with~$\alpha = \beta$ and~$\gamma = \delta$, namely when~$W$ coincides, up to multiplicative constants, with the polynomials~$(1+t)^{\alpha}$ and~$(1-t)^{\gamma}$ near the wells at~$-1$ and~$1$, respectively.

In the present work, by contrast, we improve the lower bounds in case of an oscillatory potential~$W$, that is, when either~$\alpha \neq \beta$ or~$\gamma \neq \delta$, or both. In particular, we refine the estimates obtained in~\cite{DPDV} for both the minimizers and their derivatives and prove their optimality in the case of the fractional Laplacian operator.

\medskip

Functionals as in~\eqref{main_fun} constitute a non-scaled Ginzburg--Landau-type energy, in which the kinetic term~$\Hc$ is given by some nonlocal integrals, in place of the classical Dirichlet energy. Models of this form have attracted a great deal of attention, due to their capability to capture long-range interactions between particles. They naturally arise, for instance, when dealing with phase transition phenomena involving nonlocal tension effects (see e.g.~\cite{CozziValdNONLINEARITY, MR4581189, SV12, SV14}), or in the study of the Peierls--Nabarro model for crystal dislocation (see e.g.~\cite{BV16, DPV15, DFV14, MR4531940, GM12, MR3338445, MR3511786, MR3703556}). 

Moreover, precise asymptotic bounds for these models have a number important consequences,
both in terms of mathematical development of the theory and in view of concrete applications. In particular:
\begin{itemize}
\item These estimates
shed light on the qualitative and quantitative behavior of solutions, since one-dimensional
transition layers often constitute general models also for more complicated
solutions (and, in turn, they can provide general estimates by sliding methods and comparison principles),
\item They help isolate the influence of different terms in the energy functional, 
showcasing their impact on the specific features of the transition layers,
thereby offering a clearer understanding of the model,
\item They provide an accurate level of precision, which is especially valuable in physical models, where one typically starts from empirical observations: in this spirit,
when models are obtained via phenomenological considerations rather than via
first principles, the exact knowledge of specific features
of heteroclinic connections
is decisive to validate the model and allows one to reconstruct the specific values
of the parameters of the underlying potential
through the profiles of the observed layers
(that is, the observed data regarding heteroclinic profiles combined
with rigorous asymptotic estimates allows one
to identify reliably the potential, which would be otherwise known only at a qualitative level),
\item Decay estimates for transition layers
play a crucial role in constructing explicit barriers and auxiliary solutions, which in turn can be used to further refine the bounds themselves: a concrete example of this strategy appears in the context of crystal dislocation models, where explicit profiles are used to sharpen the decay estimates of transition layers (see e.g.~\cite[Section~6]{DPV15}),
\item Sharp asymptotics on hetetoclinic connections are particularly valuable when dealing with numerical implementations of the model, since they provide accurate initial guesses for iterative schemes and improve the stability and convergence of numerical methods,
\item Obtaining optimal bounds on the decay of solutions also
becomes essential when extending the analysis to more complex or generalized framework: in such cases, sharp estimates may represent a natural starting point for conjecturing the behavior of the solutions to the extended problem.
\end{itemize}

\subsection{Main results}\label{S2}
 
In order to state the main results of this paper, we introduce some additional notation.   

For any~$s \in (0,1)$ and for any kernel~$K$ satisfying~\eqref{newbound}, we define the operator
\begin{equation}\label{main_op}
L_K u(x) := \mathrm{PV}_x \int_{\R^n}(u(y)-u(x)) K(x-y) \, dy,
\end{equation}   
where PV stands for the Cauchy Principal Value. This operator plays a key role in the study of the energy~$\mathcal{E}_K$ in~\eqref{main_fun}, since the corresponding Euler--Lagrange equation takes the form
\begin{equation}\label{all_ca_frl}
L_K u = W'(u),
\end{equation}   
which is often regarded as a nonlocal analogue of the classical (elliptic) Allen--Cahn equation (formally recovered when~$K(x)=|x|^{-n-2s}$ and~$s\nearrow1$).

\begin{rem}
{\rm
When the symmetry condition~\eqref{krn_symm} holds, the operator~$L_K$ can also be represented in a nonsingular form as
\begin{equation*}
L_K u(x) = \frac{1}{2} \int_{\R^n} \delta u(x,z) K(z) \, dz,
\end{equation*}
where~$\delta u(x,z)$ denotes the second-order increment
\begin{equation*}
\delta u(x,z) := u(x+z) + u(x-z) - 2u(x).
\end{equation*}
In the special case~$K(x) = |x|^{-n-2s}$, the operator~$L_K$ reduces to the fractional Laplacian
\begin{equation}\label{bvtg}
L_s u(x) = \frac{1}{2} \int_{\R^n} \frac{\delta u(x,z)}{|z|^{n+2s}} \, dz,
\end{equation}
where we exploited the notation~$L_s$ used in~\cite{DPDV, DFV14, DPV15}.
}
\end{rem}

We now recall the notion of minimizers relevant to our setting.

\begin{defn}\label{defini}
Let~$\Omega \subset \R^n$ be a bounded domain. A measurable function~$u : \R^n \to \R$ is a \textnormal{local minimizer} of~$\Ec$ in~$\Omega$ if~$\Ec(u;\Omega)<+\infty$ and
\[ \Ec(u;\Omega) \leq \Ec(u+\phi;\Omega) \quad \text{for all } \phi \in C^\infty_0(\Omega). \]
Moreover, $u$ is a {\rm class~A minimizer} of~$\Ec$ if it is a local minimizer in every bounded domain~$\Omega \subset \R^n$.
\end{defn}

We also define the class of admissible one-dimensional class~A minimizers as
\begin{equation}\label{defmathcalixs00}
\mathcal{X} := \left\{ f \in L^1_{\mathrm{loc}}(\R) \; \text{such that} \; \lim_{x \to \pm \infty} f(x) = \pm 1 \right\}.
\end{equation}

In~\cite[Theorem~1.5]{DPDV}, the authors establish the following result concerning asymptotic estimates for minimizers of~\eqref{main_fun}, which we recall here below for convenience.

\begin{theorem}[Theorem 1.5 in~\cite{DPDV}]\label{vogpian}
Let~$n=1$. Assume that~\eqref{krn_symm}, \eqref{newbound}, \eqref{nuovissima}, \eqref{pot_reg}, \eqref{pot_zero}, and~\eqref{pot_deg} hold true.  Suppose also that
\begin{equation}\label{ricdifar}
\max \left\{ (\alpha-2)(\alpha-\beta),  (\gamma-2)(\gamma-\delta) \right\} <1.
\end{equation}
Then, within the class~$\mathcal{X}$, there exists a unique (up to translations) nontrivial class~A minimizer~$\bar{u}$ of~$\Ec$.

Moreover, $\bar{u}$ is strictly increasing, belongs to~$C^{1+2s+\theta}(\R)$ for some~$\theta \in (0,1)$, and it is the only increasing solution to
\[
L_K u = W'(u) \quad \text{in } \R.
\]

In addition, there exist constants~$C_1$, $C_2>0$ and~$R>0$ such that
\begin{align}&
\label{asymp_decay}\begin{dcases}1+\bar{u}(x) \leq C_2 |x|^{-\frac{2s}{\alpha-1}} \quad &\text{for } x \leq -R, \\
1 - \bar{u}(x) \leq C_2|x|^{-\frac{2s}{\gamma-1}} \quad &\text{for } x \geq R,\end{dcases} \\
&\begin{dcases}\bar{u}'(x) \geq C_1 |x|^{-\left(1+\frac{2s(\alpha-\beta+1)}{\alpha-1}\right)} \quad &\text{for } x \leq -R, \label{398r7gree}\\
\bar{u}'(x) \geq C_1 |x|^{-\left(1+\frac{2s(\gamma-\delta+1)}{\gamma-1}\right)} \quad &\text{for } x \geq R,\end{dcases} \\
&\begin{dcases}
1+\bar{u}(x) \geq C_1 |x|^{-\frac{2s(\alpha-\beta+1)}{\alpha-1}} \quad &\text{for } x \leq -R, \label{asymp1_vecchio}\\
1 - \bar{u}(x) \geq C_1|x|^{-\frac{2s(\gamma-\delta+1)}{\gamma-1}} \quad &\text{for } x \geq R, \end{dcases} \\
&\begin{dcases}\bar{u}'(x) \leq C_2|x|^{-\left(1+\frac{2s(1-(\alpha-2)(\alpha-\beta))}{\alpha-1}\right)} \quad &\text{for } x \leq -R, \label{asymp2_vecchio}\\
\bar{u}'(x) \leq C_2|x|^{-\left(1+\frac{2s(1-(\gamma-2)(\gamma-\delta))}{\gamma-1}\right)} \quad &\text{for } x \geq R.\end{dcases}
\end{align}
\end{theorem}

We now present the two main results of this paper. The first one provides sharper lower bounds for the asymptotic behavior of the minimizer of~\eqref{main_fun} compared to those established in~\eqref{asymp1_vecchio} and~\eqref{asymp2_vecchio}.

\begin{theorem}\label{main_thm}
Let~$n=1$. Assume that~\eqref{krn_symm}, \eqref{newbound}, \eqref{nuovissima}, \eqref{pot_reg}, \eqref{pot_zero} and~\eqref{pot_deg} hold true. Suppose also that
\begin{equation}\label{riar}
\max \{ \alpha-\beta, \gamma-\delta \} <1.
\end{equation}
Let~$\bar{u}$ be a nontrivial class~A minimizer of~$\Ec$.

Then, there exist~$C_1$, $C_2 > 0$ and~$R > 0$ such that
\begin{align}&\begin{dcases}
1+\bar{u}(x) \geq C_1 |x|^{-\frac{2s}{\beta-1}} \quad &\text{for } x \leq -R, \label{asymp_decay_lowbound}\\
1 - \bar{u}(x) \geq C_1 |x|^{-\frac{2s}{\delta-1}} \quad &\text{for } x \geq R, \end{dcases} \\
&\begin{dcases}
\bar{u}'(x) \leq C_2|x|^{-\left(1+\frac{2s(\beta-\alpha+1)}{\beta-1}\right)} \quad &\text{for } x \leq -R, \label{eq:asymp-derivata}\\
\bar{u}'(x) \leq C_2|x|^{-\left(1+\frac{2s(\delta-\gamma+1)}{\delta-1}\right)} \quad &\text{for } x \geq R. \end{dcases}
\end{align}
\end{theorem}

We point out that
Theorem~\ref{main_thm} sharpens the bounds from Theorem~\ref{vogpian}. Indeed, the condition~\eqref{riar} is weaker than~\eqref{ricdifar} thanks to the inequalities
\[ (\alpha-2)(\alpha-\beta) \geq (\alpha-\beta)^2 \qquad \text{and} \qquad (\gamma-2)(\gamma-\delta) \geq (\gamma-\delta)^2. \]

Moreover, we have that
$$ \frac{2s}{\beta-1}\le \frac{2s(\alpha-\beta+1)}{\alpha-1},$$
with equality if and only if either~$\alpha=\beta$ or~$\beta=2$, and
$$ \frac{2s}{\delta-1}\le\frac{ 2s(\gamma-\delta+1)}{\gamma-1},$$
with equality if and only if either~$\gamma=\delta$ or~$\delta=2$.
These inequalities say that the new decay bounds in~\eqref{asymp_decay_lowbound}
improve those in~\eqref{asymp1_vecchio}.

Also, thanks to Lemma~\ref{lemmazzone2}, we obtain that the new decay bounds in~\eqref{eq:asymp-derivata} improve those in~\eqref{asymp2_vecchio}.

We now state the second main result of the paper, which shows that the asymptotic estimates derived for the minimizer~$\bar{u}$ are in fact optimal. Specifically, we prove that, within the class of kernels and potentials considered in our setting, these estimates cannot be improved.

\begin{theorem}\label{optimal}
Let~$n=1$ and~$s\in(0,1)$. Let~$K$ be the fractional Laplacian kernel, namely
\[ K(x) := |x|^{-1-2s}.\]

Then, there exist a potential~$W$ that satisfies~\eqref{pot_reg}, \eqref{pot_zero} and~\eqref{pot_deg}, with
\begin{equation}\label{edkis}
\max \{ \alpha-\beta, \gamma-\delta \} <1\qquad{\mbox{and}}\qquad \min\{\beta,\,\delta\} \geq 5,
\end{equation}
and a function~$\widetilde{u}\in{\mathcal{X}}$ such that~$L_s \widetilde{u} = W'(\widetilde{u})$ in~$\R$.

Moreover, there exist diverging sequences of positive
real numbers~$x_k$, $y_k$ and~$z_k$ such that
\begin{equation}\label{7843tgfbewb}
\begin{split}
&1 - \widetilde{u}(x_k)= C_2x_k^{-\frac{2s}{\gamma-1}},\qquad 1 - \widetilde{u}(y_k)= C_1 y_k^{-\frac{2s}{\delta-1}},\\
&1+\widetilde{u}(-x_k) = C_2 |x_k|^{-\frac{2s}{\alpha-1}}, \qquad 1+\widetilde{u}(-y_k)=C_1 |y_k|^{-\frac{2s}{\beta-1}},\\ 
&\widetilde{u}'(z_k)=C_1 z_k^{-\left(1+\frac{2s(\gamma-\delta+1)}{\gamma-1}\right)},\qquad
\widetilde{u}'(-z_k)=C_1 |z_k|^{-\left(1+\frac{2s(\alpha-\beta+1)}{\alpha-1}\right)} .
\end{split}\end{equation}
\end{theorem}

We stress that equation~\eqref{7843tgfbewb} shows that
the estimates in~\eqref{asymp_decay}, \eqref{398r7gree} and~\eqref{asymp_decay_lowbound} are indeed optimal.

\begin{rem}\label{23bve8670g7}
In~\cite[Remark~1.7]{DPDV}, the authors establish the optimality of the estimates~\eqref{asymp_decay}, \eqref{398r7gree}, \eqref{asymp_decay_lowbound} and~\eqref{eq:asymp-derivata} in the
case~$\alpha=\beta$ and~$\gamma=\delta$. Hence, in this work we only focus on the case of an oscillatory potential, namely~$\alpha \neq \beta$ and~$\gamma \neq \delta$,
in order to establish the optimality of~\eqref{asymp_decay}, \eqref{398r7gree} and~\eqref{asymp_decay_lowbound}.
It would be interesting to check whether
the bounds in~\eqref{eq:asymp-derivata} remain optimal even when~$\alpha\neq\beta$ and~$\gamma\neq\delta$.
\end{rem}

The results presented here are stated for rather general interaction kernels as in~\eqref{krn_symm}, \eqref{newbound}
and~\eqref{nuovissima}, however they are new even for
the model case in which~$K(x)=|x|^{-n-2s}$,
corresponding to the fractional Laplace operator.
\medskip

Let us briefly outline the main ingredients of our approach. In~\cite{DPDV}, the upper decay bounds~\eqref{asymp_decay} were obtained through a barrier, inspired by the one originally proposed in~\cite[Lemma~3.1]{SV14}. The lower bounds in~\eqref{asymp1_vecchio} were instead deduced from~\eqref{398r7gree} via the Fundamental Theorem of Calculus, in the
absence of an explicit lower barrier.

In contrast, our improvements in~\eqref{asymp_decay_lowbound} and~\eqref{eq:asymp-derivata} are derived via a novel barrier construction (see Proposition~\ref{xcs54340m}). Importantly, this barrier provides pointwise—not merely asymptotic—control, which is crucial in our estimates.

To prove the optimality results in Theorem~\ref{optimal}, we construct a specific stepwise function~$\widetilde{u}$ and define a potential~$V$ suitably built on the function~$\widetilde{u}$, so that~${L_s\widetilde{u}=V'(\widetilde{u})}$. 
Verifying that such a potential fulfills the required structural assumptions presents certain challenges. First, to show that~$V(-1) = V(1) = 0$ and~$V > 0$ in~$(-1,1)$ we rely on the results in~\cite{CS14}. There, the authors focus on layers solutions to semilinear problems involving the fractional Laplacian. This allows them to exploit an extension argument and obtain useful Hamiltonian estimates\footnote{
We stress that in~\cite{CP} the authors extend this construction to a broader class of kernels, including those compatible with our setting. Remarkably in~\cite{CP}, despite the absence of an extension framework, the authors are still able to identify an underlying Hamiltonian structure.} for their problem.

Next, an important step in our construction is to check that~$\widetilde{u}$ is the solution of a suitable~ODE that guarantees its
monotonicity (this is also not obvious, since, given the fractional
nature of the problem, the role played by ODEs is typically way less transparent than in classical cases), see Appendix~\ref{parode}.


\subsection{Organization of the paper} 
The rest of the paper is organized as follows. Section~\ref{ehssiressi} addresses regularity results for the operator~$L_K$.
Section~\ref{barrlk} provides the construction of some barriers for the operator~$L_K$ that play a pivotal
role in the proofs of Theorems~\ref{main_thm} and~\ref{optimal}.

In Section~\ref{allnzehs}, we prove Theorem~\ref{main_thm}. Section~\ref{3sstzac} is devoted to the construction and analysis of the function~$\widetilde{u}$, which is then employed in Section~\ref{proptimal} to prove Theorem~\ref{optimal}.

The paper contains some appendices, structured as follows. Appendix~\ref{somboun} collects useful inequalities concerning auxiliary quantities introduced in Section~\ref{notas}.

In Appendix~\ref{parode}, we solve an ODE that plays a central role in the analysis of Section~\ref{3sstzac}
and Appendix~\ref{ukprop} discusses the main properties of a specific function introduced in Section~\ref{notas}.

Finally Appendix~\ref{resudiffn} gathers auxiliary results used throughout the article. 

\section{Some regularity results for~$L_K$}\label{ehssiressi}

We present here some regularity results for the operator~$L_K$, as defined in~\eqref{main_op}. These results serve as a counterpart to those proved by Silvestre in~\cite[Section~2]{S07} for the fractional Laplacian and will be used in Section~\ref{proptimal}.

Nevertheless, a clear and unified treatment of such estimates in the case of general kernels appears to be missing from the literature. For this reason, we provide here the details.

\begin{prop}\label{3o93uh44}
Let~$n=1$, let~$s\in (0,1)$ and let~$K$ satisfy~\eqref{krn_symm} and~\eqref{newbound}. Then,
\begin{enumerate}[label=(\roman*)]
\item If~$2s<1$, $\alpha \in (2s,1]$ and~$u \in C^{1,\alpha}(\R)$, then~$(L_K u)'(x)= L_K u'(x)$ for any~$x \in \R$.\\
\item If $2s\geq 1$, $\alpha \in (2s-1,1]$ and~$u \in C^{2,\alpha}(\R)$, then~$(L_K u)'(x)= L_K u'(x)$ for any~$x \in \R$.
\end{enumerate}
\end{prop}

\begin{proof}
Let~$x \in \R$ and~$|h|\leq 1$. We define the quantities
\begin{eqnarray*}
I_1(h)&:=& PV \int_{-1}^1 \frac{u(x+h+z)-u(x+h)- u(x+z)+u(x)}{h}\,K(z) \, dz\\
{\mbox{and }}\quad
I_2(h)&:=& \int_{\R\setminus (-1,1)} \frac{u(x+h+z)-u(x+h)- u(x+z)+u(x)}{h}\,K(z) \, dz,
\end{eqnarray*}
where~$PV$ denotes the Cauchy principal value.

We firstly focus on the proof of~\textit{(i)} and we claim that
\begin{equation}\label{ndehee}
\lim_{h \to 0 } I_1(h) = PV\int_{-1}^1 \big(u'(x+z)-u'(x)\big)K(z) \, dz.
\end{equation}
To show this, we use the Fundamental Theorem of Calculus to see that
\begin{eqnarray*}
& & u(x+h+z) - u(x+z) = h \int_0^1 u'(x+z+\tau h)\, d\tau\\
\mbox{and} & &u(x+h) - u(x) = h \int_0^1 u'(x+\tau h)\, d\tau.
\end{eqnarray*}
On this account, since~$u \in C^{1,\alpha}(\R)$,
\begin{equation*}
\begin{split}
|u(x+h+z)-u(x+h)- u(x+z)+u(x)|&\leq |h| \int_0^1 |u'(x+z+\tau h)-u'(x+\tau h)| \, d\tau\\
&\leq [u']_{C^{0,\alpha}(\R)} |z|^{\alpha}|h|.
\end{split}
\end{equation*}
As a result, recalling~\eqref{newbound}, we obtain that
$$ \left|\frac{u(x+h+z)-u(x+h)- u(x+z)+u(x)}{h}\,K(z)\right|\le \Lambda  [u']_{C^{0,\alpha}(\R)}|z|^{\alpha-1-2s}.$$
Since the function
\begin{equation*}
z\mapsto \Lambda [u']_{C^{0,\alpha}(\R)} |z|^{\alpha-1-2s} \in L^1(-1,1),
\end{equation*}
we obtain~\eqref{ndehee} by using the Dominated Convergence Theorem.

We also have that
\begin{equation}\label{jevrtf4}
\lim_{h \to 0} I_2(h) = \int_{\R\setminus (-1,1)} (u'(x+z)-u'(x)) K(z) \, dz.
\end{equation}
Indeed, we observe that 
\begin{equation*}
\big|u(x+h+z)-u(x+z)-u(x+h)+u(x)\big| \leq 2 |h| \Vert u'\Vert_{L^{\infty}(\R)}.
\end{equation*}
Therefore, using also~\eqref{newbound},
$$  \left|\frac{u(x+h+z)-u(x+h)- u(x+z)+u(x)}{h}\,K(z)\right|\le 2\Lambda \Vert u'\Vert_{L^{\infty}(\R)}|z|^{-1-2s}.$$
Since the function
\begin{equation*}
z \mapsto  2 \Lambda \Vert u'\Vert_{L^{\infty}(\R)} |z|^{-1-2s} \in L^1(\R\setminus(-1,1)),
\end{equation*}
we can apply the Dominated Convergence Theorem and obtain~\eqref{jevrtf4}.

Now, the limits in~\eqref{ndehee} and~\eqref{jevrtf4} together yield~\textit{(i)}.

To establish~\textit{(ii)}, we prove that~\eqref{ndehee} and~\eqref{jevrtf4} holds true in this case as well.
For this, we observe that, since~$u \in C^{2,\alpha}(\R)$, for some~$\alpha \in (2s-1,1]$,
\begin{eqnarray*}
&&u(x+z+h)-u(x+h)-u'(x+h)z = z \int_0^1 \big(u'(x+h+\tau z) - u'(x+h)\big) \, d\tau 
\\{\mbox{and }}
&&u(x+z)-u(x)-u'(x)z = z \int_0^1 \big(u'(x+\tau z) - u'(x)\big) \, d\tau.
\end{eqnarray*}
As a consequence,
\begin{equation*}
\begin{split}&
\big|\big(u(x+z+h)-u(x+h)-u'(x+h)z\big)-\big(u(x+z)-u(x)-u'(x)z\big)\big| \\
&\qquad\leq |z| \left| \int_0^1 \big(u'(x+h+\tau z) -u'(x+\tau z) \big)-\big(u'(x+h)-u'(x)\big) \, d\tau  \right|\\
&\qquad\leq |z|| h|\int_0^1\int_0^1 \big|u''(x+\sigma h+\tau z)- u''(x+\sigma h)\big| \, d\sigma\,d\tau \\
&\qquad\leq \frac{[u'']_{C^{0,\alpha}(\R)} }{\alpha+1}|z|^{1+\alpha}|h|.
\end{split}
\end{equation*}
Therefore, by~\eqref{newbound},
\begin{eqnarray*}&&
\left|\frac{\big(u(x+z+h)-u(x+h)-u'(x+h)z\big)-\big(u(x+z)-u(x)-u'(x)z\big)}{h}\,K(z)\right|\\&&\qquad\le
\frac{\Lambda [u'']_{C^{0,\alpha}(\R)} }{\alpha+1}|z|^{\alpha-2s}.
\end{eqnarray*}
Since~$\alpha-2s>-1$, the function
\begin{equation*}
z\mapsto  \frac{\Lambda [u'']_{C^{0,\alpha}(\R)} }{\alpha+1}|z|^{\alpha-2s}\in L^1(-1,1), 
\end{equation*} and
so the Dominated Convergence Theorem yields~\eqref{ndehee}.

The proof of~\eqref{jevrtf4} is the same as in the setting~\textit{(i)}.

Thus, \eqref{ndehee} and~\eqref{jevrtf4}, together with the assumption~\eqref{krn_symm},
establish~\textit{(ii)}.
\end{proof}

\begin{prop}\label{29ueh38}
Let~$n=1$, let~$s\in (0,1)$ and let~$K$ satisfy~\eqref{krn_symm} and~\eqref{newbound}.
Then, there exists a positive constant~$C$ such that
\begin{enumerate}[label=(\roman*)]
\item If~$2s<1$, $\alpha \in (2s,1]$ and~$u \in C^{0,\alpha}(\R)$, then~$L_K u \in C^{0,\alpha-2s}(\R)$ and
\begin{equation*}
[L_Ku]_{C^{\alpha-2s}(\R)} \leq C \Lambda [u]_{C^{0,\alpha}(\R)}.
\end{equation*}
\item If $2s\geq1$, $\alpha \in (2s-1,1]$ and~$u \in C^{1,\alpha}(\R)$, then~$L_K u \in C^{0,\alpha+1-2s}(\R)$ and
\begin{equation*}
[L_Ku]_{C^{0,\alpha+1-2s}(\R)} \leq C \Lambda [u']_{C^{0,\alpha}(\R)}.
\end{equation*}
\end{enumerate}
\end{prop}

\begin{proof}
Throughout the proof, we denote by~$C$ any positive constant, possibly varying from line to line.

Let~$x_1$, $x_2 \in \R$ and set~$r:=|x_1-x_2|$,
\begin{eqnarray*}
I_1&:=& PV \int_{-r}^r \big(u(x_1+z)- u(x_1)-u(x_2+z)+u(x_2)\big)K(z)\, dz\\
{\mbox{and }}\quad I_2&:=&\int_{\R\setminus(-r,r)} \big(u(x_1+z)- u(x_1)-u(x_2+z)+u(x_2)\big)K(z)\, dz,
\end{eqnarray*}
where PV denotes the Cauchy principal value.

We now focus on~\textit{(i)} and check that
\begin{equation}\label{eq87tr}
|I_1|+|I_2|\leq C \Lambda [u]_{C^{0,\alpha}(\R)} |x_1-x_2|^{\alpha-2s}.
\end{equation}
To show this, we observe that
\begin{eqnarray*}
&&|u(x_i+z)-u(x_i)|\leq [u]_{C^{0,\alpha}(\R)}|z|^{\alpha}\qquad {\mbox{for~$i=1,2$,}}\\{\mbox{and }}&&
|u(x_1)-u(x_2)| \leq [u]_{C^{0,\alpha}(\R)} |x_1-x_2|^{\alpha}.\end{eqnarray*} Then, recalling~\eqref{newbound}, we gather that
\begin{equation*}
\begin{split}
|I_1|+|I_2|&\leq 2 \Lambda [u]_{C^{0,\alpha}(\R)}\left( \int_{-r}^r |z|^{\alpha-1-2s}\, dz + |x_1-x_2|^{\alpha} \int_{\R\setminus(-r,r)} |z|^{-1-2s}\, dz  \right) \\
&\leq  C \Lambda [u]_{C^{0,\alpha}(\R)}\left( r^{\alpha-2s}
+ |x_1-x_2|^{\alpha}r^{-2s}\right)
\\&  \leq C \Lambda [u]_{C^{0,\alpha}(\R)} |x_1-x_2|^{\alpha-2s},
\end{split}
\end{equation*}
thus obtaining~\eqref{eq87tr} and, in turn, the estimate on~$[L_Ku]_{C^{0,\alpha-2s}(\R)}$. 

Moreover, an~$L^\infty$-bound for~$L_Ku$ can be obtained by noticing that, for any~$x \in \R$,
\begin{equation*}
\left| PV\int_{-1}^1 \big(u(x+z)-u(x)\big)K(z)\, dz\right| \leq \Lambda [u]_{C^{0,\alpha}(\R)} \int_{-1}^1 |z|^{\alpha-1-2s} \, dz\le  C \Lambda [u]_{C^{0,\alpha}(\R)}
\end{equation*}
and
\begin{equation*}
\int_{\R \setminus (-1,1)} \big(u(x+z)-u(x)\big)K(z)\, dz \leq 2 \Lambda \Vert u \Vert_{L^{\infty}(\R)} \int_{\R \setminus (-1,1)} |z|^{-1-2s} \, dz
\le C\Lambda\Vert u \Vert_{L^{\infty}(\R)}.
\end{equation*}
Gathering these pieces of information, we thereby conclude that~$L_K u \in C^{0,\alpha-2s}(\R)$, thus completing the proof of~\textit{(i)}.

We now focus on the proof of~\textit{(ii)}. The regularity of~$u$ implies that, for~$i=1,2$,
\begin{equation}\label{erx214wzw}
u(x_i+z)-u(x_i) = u'(x_i) z + z \int_0^1 \big(u'(x_i +\tau z)-u'(x_i)\big) \, d\tau.
\end{equation}
From this, \eqref{krn_symm} and~\eqref{newbound}, we deduce that
\begin{equation}\label{nrnr4}
|I_1| \leq C \Lambda [u']_{C^{0,\alpha}(\R)} \int_{-r}^r |z|^{\alpha-2s} \, dz \leq C \Lambda [u']_{C^{0,\alpha}(\R)} |x_1-x_2|^{\alpha+1-2s}.
\end{equation}

Also, \eqref{erx214wzw} yields that
\begin{equation*}
\begin{split}
\big|u(x_1+z)-u(x_1) -u(x_2+z)+u(x_2) \big|& \leq |z| \,|u'(x_1)-u'(x_2)| + C [u']_{C^{0,\alpha}(\R)} |z|^{1+\alpha} \\
&\leq  C [u']_{C^{0,\alpha}(\R)} \left( |z|\,|x_1-x_2|^{\alpha} +|z|^{1+\alpha}\right).
\end{split}
\end{equation*}
This and~\eqref{newbound} lead to
\begin{equation}\label{beuwy32}
|I_2| \leq C \Lambda [u']_{C^{0,\alpha}(\R)} |x_1-x_2|^{\alpha+1-2s}.
\end{equation}
The estimate on~$[L_Ku]_{C^{0,\alpha+1-2s}(\R)}$ thus follows from~\eqref{nrnr4} and~\eqref{beuwy32}.

To prove that~$L_Ku$ is uniformly bounded in~$\R$, we recall~\eqref{krn_symm} and~\eqref{newbound} and we
use~\eqref{erx214wzw}
and the Fundamental Theorem of Calculus to obtain that, for any~$x\in \R$,
\begin{equation*}
\begin{split}&
\left|\int_{-1}^1 \big(u(x+z)-u(x)-u'(x)z\big) K(z) \, dz\right| \leq \Lambda \int_{-1}^1 \int_0^1 |z|^{-2s} |u'(x+\tau z)-u'(x)|\, d\tau \, dz\\
&\qquad \leq C \Lambda [u']_{C^{0,\alpha}(\R)} \int_{-1}^1 |z|^{\alpha-2s} \, dz
\leq C \Lambda [u']_{C^{0,\alpha}(\R)}
\end{split}
\end{equation*}
and
\begin{equation*}
\left|\int_{\R \setminus (-1,1)} \big(u(x+z)-u(x)\big)K(z)\, dz\right|
\leq 2 \Lambda \Vert u \Vert_{L^{\infty}(\R)} \int_{\R \setminus (-1,1)} |z|^{-1-2s} \, d\le 2 \Lambda \Vert u \Vert_{L^{\infty}(\R)} .
\end{equation*}
Thus, we have that~$L_K u \in C^{0,\alpha+1-2s}(\R)$
and the proof is thereby complete.
\end{proof}

As a consequence of Propositions~\ref{3o93uh44} and~\ref{29ueh38}, we obtain the following:

\begin{corol}\label{ngotonl}
Let~$n=1$.
Let~$k \in \N$, $s\in (0,1)$ and let~$K$ satisfy~\eqref{krn_symm} and~\eqref{newbound}. Then,
\begin{enumerate}[label=(\roman*)]
\item If~$2s<1$, $\alpha \in (2s,1]$ and~$u \in C^{k,\alpha}(\R)$, then~$L_K u \in C^{k,\alpha-2s}(\R)$.
\item If $2s\geq1$, $k\geq 1$, $\alpha \in (2s-1,1]$ and~$u \in C^{k,\alpha}(\R)$, then~$L_K u \in C^{k-1,\alpha+1-2s}(\R)$.
\end{enumerate}
\end{corol}

\begin{rem}
We stress that Propositions~\ref{3o93uh44} and~\ref{29ueh38} and Corollary~\ref{ngotonl} only exploit the upper bound in~\eqref{newbound}.

Also, we point out that the assumption~\eqref{krn_symm} is only used when~$2s\geq1$.
\end{rem}

\section{Some useful barriers for~$L_K$}\label{barrlk}

In this section we collect some barrier constructions for the operator~$L_K$, defined in~\eqref{main_op}, in
dimension~$n = 1$. They will play a pivotal role in the proofs of Theorems~\ref{main_thm} and~\ref{optimal}.

\begin{prop}\label{xcs54340m}
Let~$s \in (0,1)$ and let~$K$ satisfy~\eqref{krn_symm} and~\eqref{newbound}. Let also~$\bar{x}\geq 1$, $\alpha$, $A>0$, $B$, $D \leq 1-\alpha \bar{x}^{-A}$ and
\begin{equation*}
\phi_{\bar{x}}(x):= \begin{cases}
B &\mbox{if } x\leq 0,\\
D &\mbox{if } x \in (0,\bar{x}),\\
1 - \alpha x^{-A} &\mbox{if } x \geq \bar{x}.
\end{cases}
\end{equation*}

Then, for any~$ x \geq 2\bar{x}$,
\begin{equation}\label{8fe0re34}
x^{2s}L_K\phi_{\bar{x}}(x) \leq -\frac{\lambda(1-B)}{2s}+ \alpha C x^{-A}.
\end{equation}
for some~$C>0$ depending on~$s$, $\lambda$ and~$\Lambda$.
\end{prop}

\begin{proof}
Let~$x \geq 2\bar{x}$. Thanks to~\eqref{newbound},
and applying the change of variable~$y:=x\theta$, we compute
\begin{equation}\label{b94fcrt0g8}
\begin{split}&
x^{2s} \int_{-\infty}^0 \left( \phi_{\bar{x}}(y)-\phi_{\bar{x}}(x)\right) K(x-y) \, dy 
=x^{2s} \int_{-\infty}^0 \left( B-1 + \alpha x^{-A}\right) K(x-y) \, dy \\&\qquad
\leq -\lambda x^{2s} \left(1 - B-\alpha x^{-A}\right)  \int_{-\infty}^0\frac{dy}{\left(x-y\right)^{1+2s}}\\
&\quad= - \lambda \left(1 - B-\alpha x^{-A}\right) \int_{-\infty}^0 \frac{d\theta}{\left(1-\theta\right)^{1+2s}}  =- \frac{\lambda}{2s} \left(1 - B\right) + \frac{\alpha \lambda}{2s} x^{-A}.
\end{split}
\end{equation}

Also, recalling the positivity of the kernel~$K$ and the fact that~$D \leq 1-\alpha\bar{x}^{-A}$, we find that
\begin{equation}\label{0nt83rdc}\begin{split}&
\int_{0}^{\frac{x}2} \left( \phi_{\bar{x}}(y)-\phi_{\bar{x}}(x)\right) K(x-y) \, dy 
\\&\qquad= \int_{0}^{\bar{x}} \big(D-(1-\alpha x^{-A})\big) K(x-y) \, dy
+\alpha\int_{\bar{x}}^{\frac{x}2} \big( x^{-A}-y^{-A}\big) K(x-y) \, dy
\leq 0.
\end{split}
\end{equation}

Now, we check that there exists a positive constant~$C$ depending on~$s$ and~$\Lambda$ such that
\begin{equation}\label{29fgf49f}
x^{2s}\left| PV\int_{\frac{x}2}^{\frac{3x}2}  \left( \phi_{\bar{x}}(y)-\phi_{\bar{x}}(x)\right) K(x-y) \, dy \right|\leq \alpha C x^{-A}.
\end{equation}
In the following computations, we will omit the principal value notation for the sake of readability.
In the aim of showing~\eqref{29fgf49f}, we use the change of variable~$\theta:= 1+z$, the symmetry of~$K$ and a Taylor expansion of~$(1+z)^{-A}$ around the origin to see that
\begin{equation*}
\begin{split}&
\int_{\frac12}^{\frac32}\big(1-\theta^{-A}\big)K (x(1-\theta ))\, d\theta = \int_{-\frac12}^{\frac12} \big( 1- ( 1+z )^{-A}\big) K(xz) \, dz \\
&\qquad= \int_{-\frac12}^{\frac12} \big( Az + O(|z|^2) \big) K(xz)\, dz =  \int_{-\frac12}^{\frac12} O(|z|^2) K(xz)\, dz.
\end{split}
\end{equation*}
As a consequence, changing variable~$y:=x\theta$ and recalling~\eqref{newbound}, 
\begin{equation*}
\begin{split}
&\left|\int_{\frac{x}2}^{\frac{3x}2}  \left( \phi_{\bar{x}}(y)-\phi_{\bar{x}}(x)\right) K(x-y) \, dy\right| 
=\alpha\left|\int_{\frac{x}2}^{\frac{3x}2}  \left( x^{-A}-y^{-A}\right) K(x-y) \, dy\right|
\\&\qquad
=  \alpha x^{-A+1} \left|\int_{\frac12}^{\frac32}\big(1-\theta^{-A}\big)K (x(1-\theta ))\, d\theta\right|
\leq \alpha C x^{-A+1} \int_{-\frac12}^{\frac12} |z|^2 K(xz)\, dz \\
&\qquad\leq \alpha C x^{-A-2s}  \int_{-\frac12}^{\frac12} |z|^{1-2s} \, dz \leq \alpha C x^{-A-2s}.
\end{split}
\end{equation*}
This completes the proof of~\eqref{29fgf49f}.

In addition, exploiting the change of variable~$y:=x\theta$ and applying~\eqref{newbound} one more time we gather
\begin{equation}\label{ne974f4}
\begin{split}&
x^{2s}\int_{\frac{3x}2} ^{+\infty}  \left( \phi_{\bar{x}}(y)-\phi_{\bar{x}}(x)\right) K(x-y) \, dy
=\alpha x^{2s}
\int_{\frac{3x}2} ^{+\infty}  \left(x^{-A}-y^{-A}\right) K(x-y) \, dy\\&\qquad
=\alpha x^{2s-A+1}
\int_{\frac{3}2} ^{+\infty}  \big(1-\theta^{-A}\big) K(x(1-\theta)) \, d\theta
\leq \alpha \Lambda x^{-A}  \int_{\frac32}^{+\infty}\frac{1-\theta^{-A}}{\left( \theta-1\right)^{1+2s}} \, d\theta\\&\qquad=\alpha C x^{-A}  .
\end{split}
\end{equation}
Merging together~\eqref{b94fcrt0g8}, \eqref{0nt83rdc}, \eqref{29fgf49f} and~\eqref{ne974f4}, we obtain the desired result.
\end{proof}

The following statements generalize~\cite[Propositions~5.1 and~5.2]{DPDV} to the situation in which there is an inequality in~\eqref{bvnciw3759843ygtuedwoiu} below rather than equality (also we need to require less regularity for the function~$\phi$).

\begin{prop}\label{tonto} let~$s\in(0,1)$ and let~$K$ satisfy~\eqref{krn_symm} and~\eqref{newbound}.
Let~$\bar{C}$, $\kappa\in (0, +\infty)$ and~$\sigma$, $\tau\in(1,+\infty)$. 

Also, let~$\phi \in C^{1,1}(\R)$ be a nonnegative function such that
\begin{equation} \label{bvnciw3759843ygtuedwoiu}
\phi(x)\leq \begin{cases}
\bar{C}|x|^{-\sigma} &\mbox{if } x<-\kappa, \\
\bar{C}|x|^{-\tau} &\mbox{if } x>\kappa
\end{cases}
\end{equation}
and
\begin{equation}\label{addipotesi76}
{\mbox{$\phi(x)\ge\gamma$ for all~$x\in[-\kappa,\kappa]$, for some~$\gamma\in(0,+\infty)$.}}
\end{equation}

In addition, suppose that
\begin{equation}\label{7f9c7w3n}
\lim_{x\to +\infty} x^3 \Vert \phi''\Vert_{L^{\infty}(\frac{x}2,\frac{3x}2)}= 0 \qquad\mbox{and}\qquad \lim_{x\to-\infty} x^3 \Vert \phi''\Vert_{L^{\infty}(\frac{3x}2,\frac{x}2)}= 0.
\end{equation}

Then,
\begin{equation}\label{yuuy_secs}
\lim_{|x | \to  +\infty} |x|^{1+2s} {L}_K\phi(x) \leq \Lambda \left( \frac{\bar{C}\kappa^{1-\sigma}}{\sigma-1}+ \int_{-\kappa}^{\kappa} \phi(y) \, dy+\frac{\bar{C}\kappa^{1-\tau}}{\tau-1}\right),
\end{equation}
where~$\Lambda$ is the quantity appearing in~\eqref{newbound}.
\end{prop}

\begin{proof}
We will only prove~\eqref{yuuy_secs} for~$x \to +\infty$, being the limit as~$x \to -\infty$ analogous. 

Let~$ x \geq 2\kappa$ and define
\begin{equation}\label{52BIS}
\begin{split}
&A_K(x):=  \int_{-\infty}^{-\kappa}\left( \phi(y)-\phi(x)\right) K(x-y)\, dy, \\ &B_K(x):= \int_{-\kappa}^{\kappa} \left( \phi(y) -\phi(x)\right) K(x-y) \, dy\\
{\mbox{and }}\qquad
&D_K(x):= PV_x\int_{\kappa}^{+\infty}\left( \phi(y)-\phi(x)\right) K(x-y)\, dy.
\end{split}
\end{equation}
In the following computations we will omit the principal value notation, for the sake of readability.

We claim that
\begin{equation}\label{x_secs}
\lim_{x\to+\infty} x^{1+2s}A_K(x) \leq \frac{\bar{C}\Lambda\kappa^{1-\sigma}}{\sigma-1}.
\end{equation}
To check this, we exploit~\eqref{newbound} and the change of variable~$y=x\theta$ and obtain that
\begin{equation*}
\begin{split}
(\bar{C}\Lambda)^{-1} x^{1+2s}A_K(x) &\leq \int_{-\infty}^{-\kappa} \left(|y|^{-\sigma}+ x^{-\tau}\right) \left( 1- \frac{y}{x}\right)^{-1-2s} \, dy \\
&=  x^{1-\sigma} \int_{-\infty}^{-\frac{\kappa}{x}} \frac{ \left| \theta\right|^{-\sigma}}{(1-\theta)^{1+2s}} \, d\theta + x^{1-\tau} \int_{-\infty}^{-\frac{\kappa}{x}} \frac{d\theta}{(1-\theta)^{1+2s}} .
\end{split}
\end{equation*}

We notice that
\begin{equation*}
\begin{split}
\int_{-\infty}^{-\frac{\kappa}{x}} \frac{ \left| \theta\right|^{-\sigma}}{(1-\theta)^{1+2s}} \, d\theta  & =  \int_{-\infty}^{-\frac12} \frac{ \left| \theta\right|^{-\sigma}}{(1-\theta)^{1+2s}} \, d\theta + \int_{-\frac12}^{-\frac{\kappa}{x}} \frac{ \left| \theta\right|^{-\sigma}}{(1-\theta)^{1+2s}} \, d\theta \\
&=  \int_{-\infty}^{-\frac12}\frac{ \left| \theta\right|^{-\sigma}}{(1-\theta)^{1+2s}} \, d\theta + \int_{-\frac12}^{-\frac{\kappa}{x}} \left| \theta\right|^{-\sigma} \left(1+O(\theta)\right) \, d\theta \\
&= \int_{-\infty}^{-\frac12}\frac{ \left| \theta\right|^{-\sigma}}{(1-\theta)^{1+2s}} \, d\theta +  \frac1{\sigma-1}\left( \left(\frac{x}{\kappa}\right)^{\sigma-1}-2^{\sigma-1} \right) + O(x^{\sigma-2}).
\end{split}
\end{equation*}
As a result,
\begin{equation*}
\begin{split}
(\bar{C}\Lambda)^{-1} x^{1+2s}A_K(x) \leq \frac{\kappa^{1-\sigma}}{\sigma-1} + o(1)
\end{split}
\end{equation*}
as~$x \to +\infty$, which gives the desired claim in~\eqref{x_secs}.

Now, we want to show that
\begin{equation}\label{fcpmd}
\lim_{x \to +\infty} x^{1+2s} B_K(x) \leq \Lambda \int_{-\kappa}^{\kappa} \phi(y) \, dy.
\end{equation}
For this, we take~$\gamma$ as in~\eqref{addipotesi76}
and we suppose that~$x>(\bar{C}/\gamma)^{1/\tau}$.
In this way, we have that~$\phi(x)\leq \bar{C} x^{-\tau}<\gamma\le\phi(y)$
for all~$y\in(-\kappa,\kappa)$,
thanks to~\eqref{addipotesi76}.
Therefore, by~\eqref{newbound},
\begin{eqnarray*}
x^{1+2s}B_K(x)&=& x^{1+2s}
\int_{-\kappa}^{\kappa} \left( \phi(y) -\phi(x)\right) K(x-y) \, dy\\
&\le& \Lambda \int_{-\kappa}^{\kappa} \left( \phi(y)-\phi(x) \right) \left| 1 -\frac{y}{x} \right|^{-(1+2s)}\, dy.
\end{eqnarray*}
We can now take the limit as~$x \to +\infty$ and use the Dominated Convergence Theorem to deduce~\eqref{fcpmd}. 

Furthermore, we check that
\begin{equation}\label{fcpmt}
\lim_{x\to+\infty} x^{1+2s} D_K(x) \leq \frac{\bar{C}\Lambda\kappa^{1-\tau}}{\tau-1}.
\end{equation}
To prove this, we set
\begin{equation}\label{53BIS}
\begin{split}
& D_{K,I}(x):=\int_{\kappa}^{\frac{x}2} \left( \phi(y)-\phi(x)\right) K(x-y)\, dy, \\
& D_{K,II}(x):= \int_{\frac{x}2}^{\frac{3x}2}  \left( \phi(y)-\phi(x)\right) K(x-y)\, dy \\ \mbox{and}\qquad
&D_{K,III}(x):=  \int_{\frac{3x}2}^{+\infty} \left( \phi(y)-\phi(x)\right) K(x-y)\, dy.
\end{split}
\end{equation}

Now, we rely on~\eqref{newbound} and the change of variable~$y:=x\theta$ and compute
\begin{equation*}
\begin{split}
(\bar{C}\Lambda)^{-1}x^{1+2s} D_{K,I}(x)  &\leq \int_{\kappa}^{\frac{x}2}\left(y^{-\tau}+x^{-\tau}\right) \left(1- \frac{y}{x} \right)^{-(1+2s)} \, dy \\
&= x^{1-\tau} \int_{\frac{\kappa}{x}}^{\frac12} \frac{\theta^{-\tau}}{\left( 1- \theta \right)^{1+2s}} \, d\theta + x^{1-\tau}\int_{\frac{\kappa}{x}}^{\frac12} \frac{d\theta}{ (1- \theta)^{1+2s}} .
\end{split}
\end{equation*}
Also, we observe that
\begin{equation*}
 \int_{\frac{\kappa}{x}}^{\frac12} \frac{\theta^{-\tau}}{\left( 1- \theta \right)^{1+2s}} \, d\theta =  \int_{\frac{\kappa}{x}}^{\frac12}\theta^{-\tau}\left(1+O(\theta)\right) \, d\theta = \frac1{\tau-1} \left(\left(\frac{x}{\kappa}\right)^{\tau-1}-2^{\tau-1}\right) + O(x^{\tau-2})
\end{equation*}
and we obtain that
\begin{equation}\label{nr9cdf}
 \lim_{x \to +\infty}x^{1+2s} D_{K,I}(x)  \leq \frac{\bar{C}\Lambda\kappa^{1-\tau}}{\tau-1} .
\end{equation}


Moreover, we recall that~$\phi' $ is Lipschitz continuous and that~$\phi'' \in L^{\infty}(\R)$.
As a consequence, applying twice the Fundamental Theorem of Calculus, we obtain that, for any~$\theta \in (1/2,3/2)$,
\begin{equation}\label{vcbnxe7598760t9ewh9oigtu4}
\begin{split}&
\phi(x\theta)- \phi(x)=  \int_{x}^{x\theta} \phi'(\tau)\,d\tau = \phi'(x)x(\theta-1)+ \int_{x}^{x\theta} \left( \phi'(\tau)- \phi'(x)\right)\,d\tau\\
&\qquad=  \phi'(x)x(\theta-1) +\int_{x}^{x\theta}\int_x^\tau \phi''(\eta)\, d\eta\, d\tau.
\end{split}
\end{equation}
Thus, changing variable~$y:=x\theta$, the symmetry of~$K$ in~\eqref{krn_symm} and the bound on~$K$ in~\eqref{newbound} lead to
\begin{equation*}
\begin{split}
x^{1+2s} |D_{K,II}(x)|& = x^{2+2s}\left| \int_{\frac12}^{\frac32}\left(\phi(x\theta)-\phi(x)\right)K(x(1-\theta)) \, d\theta\right| \\
&= x^{2+2s}\left| \int_{\frac12}^{\frac32}\left(\int_{x}^{x\theta}\int_x^\tau \phi''(\eta)\, d\eta \,d\tau\right)K(x(1-\theta)) \, d\theta\right| \\
&\leq x^{4+2s}\Vert \phi''\Vert_{L^{\infty}(\frac{x}2,\frac{3x}2)}   \int_{\frac12}^{\frac32} (\theta-1)^2  K(x(1-\theta)) \, d\theta\\
&\leq \Lambda x^3 \Vert \phi''\Vert_{L^{\infty}(\frac{x}2,\frac{3x}2)}  \int_{\frac12}^{\frac32}\left|1-\theta\right|^{1-2s} \, d\theta.
\end{split}
\end{equation*}
Hence, exploiting ~\eqref{7f9c7w3n} we conclude that 
\begin{equation}\label{PN0}
\lim_{x\to+\infty}x^{1+2s} |D_{K,II}(x)|=0.
\end{equation}



In addition, by~\eqref{newbound}, 
\begin{equation*}
\begin{split}
x^{1+2s} |D_{K,III}(x)| & \leq  \Lambda \int_{\frac{3x}2}^{+\infty}\left(\phi(y)+\phi(x)\right) \left(\frac{y}{x}-1\right)^{-(1+2s)}\, dy\\
&\leq 2\bar{C}\Lambda x^{-\tau} \int_{\frac{3x}2}^{+\infty}\left(\frac{y}{x}-1\right)^{-(1+2s)}\, dy \\
&= 2\bar{C}\Lambda x^{1-\tau} \int_{\frac32}^{+\infty}\frac{d\theta}{(\theta-1)^{1+2s}}.
\end{split}
\end{equation*}
As a consequence
\begin{equation*}
\lim_{x\to+\infty} x^{1+2s} D_{K,III}(x) =0.
\end{equation*}
Since~$D_K (x) = D_{K,I}(x)+D_{K,II}(x)+D_{K,III}(x)$, from this and~\eqref{nr9cdf} and~\eqref{PN0} we infer~\eqref{fcpmt}.

Now, being~$ L_K\phi(x) = A_{K}(x)+B_{K}(x)+D_{K}(x)$, combining together~\eqref{x_secs} ~\eqref{fcpmd} and~\eqref{fcpmt} we obtain the desired result.
\end{proof}

\begin{prop}\label{tontobis}
Let~$s\in(0,1)$ and let~$K$ satisfy~\eqref{krn_symm} and~\eqref{newbound}. 
Let~$\bar{C}$, $\kappa\in (0, +\infty)$ and~$\sigma$, $\tau\in(1,+\infty)$. 

Also, let~$\phi \in C^{1,1}(\R)$ be such that
\begin{equation*}
\phi(x)\geq  \begin{cases}
\bar{C} |x|^{-\sigma} &\mbox{if } x<-\kappa, \\
\bar{C} |x|^{-\tau} &\mbox{if } x>\kappa
\end{cases}
\end{equation*}
and
\begin{equation}\label{addipotesi76bis}
{\mbox{$\phi(x)\ge\gamma$ for all~$x\in[-\kappa,\kappa]$, for some~$\gamma\in(0,+\infty)$.}}
\end{equation}

In addition, suppose that
\begin{eqnarray}
\label{aggiuntaforse}
&&\lim_{x\to\pm\infty} |x|\phi(x)=0,\\
&&
\label{93n7w7cf}
\lim_{x\to +\infty} x^3 \Vert \phi''\Vert_{L^{\infty}(\frac{x}2,\frac{3x}2)}= 0 \qquad\mbox{and}\qquad \lim_{x\to-\infty} x^3 \Vert \phi''\Vert_{L^{\infty}(\frac{3x}2,\frac{x}2)}= 0.
\end{eqnarray}

Then,
\begin{equation}\label{yuuy_secsbis}
\lim_{|x | \to  +\infty} x^{1+2s} {L}_K\phi(x) \geq \lambda \left( \frac{\bar{C}\kappa^{1-\sigma}}{\sigma-1}+ \int_{-\kappa}^{\kappa} \phi(y) \, dy+\frac{\bar{C}\kappa^{1-\tau}}{\tau-1}\right),
\end{equation}
where~$\lambda$ is the  quantity appearing in~\eqref{newbound}. 
\end{prop}

\begin{proof}
We will only prove~\eqref{yuuy_secsbis} for~$x \to +\infty$, being
the limit as~$x\to -\infty$ analogous.
Also, for any~$x>2\kappa$, we take~$A_K(x)$, $B_K(x)$ and~$D_K(x)$ as in~\eqref{52BIS}.

We claim that
\begin{equation}\label{o399c63v0c6}
\lim_{x\to+\infty} x^{1+2s}A_K(x) \geq \frac{\bar{C}\lambda\kappa^{1-\sigma}}{\sigma-1}.
\end{equation}
To check this, we use~\eqref{newbound}
and exploit the change if variable~$y:=x\theta$ to find that
\begin{equation}\label{bvcneiwry389trguefbvadk09876}
\begin{split}
x^{1+2s}A_K(x) &= x^{1+2s} \int_{-\infty}^{-\kappa} \phi(y) K(x-y) \, dy - x^{1+2s}  \phi(x) \int_{-\infty}^{-\kappa} K(x-y) \, dy\\
&\geq \bar{C} \lambda  \int_{-\infty}^{-\kappa} |y|^{-\sigma} \left(1-\frac{y}{x} \right)^{-(1+2s)}\, dy
 -\Lambda \phi(x) \int_{-\infty}^{-\kappa} \left(1-\frac{y}{x}\right)^{-1-2s} \, dy\\
&= \bar{C} \lambda   x^{1-\sigma}\int_{-\infty}^{-\frac{\kappa}{x}}
\frac{|\theta|^{-\sigma}}{(1-\theta)^{1+2s}} \, d\theta
- \Lambda x\phi(x) \int_{-\infty}^{-\frac{\kappa}{x}}
\frac{d\theta}{( 1-\theta)^{1+2s}}\\
&= \bar{C}  \lambda  x^{1-\sigma}\int_{-\frac12}^{-\frac{\kappa}{x}}|\theta|^{-\sigma}(1+O(\theta))\, d\theta + O( x^{1-\sigma})-C \Lambda x\phi(x).
\end{split}
\end{equation}
Also, we notice that
\begin{equation*}
\begin{split}
\int_{-\frac12}^{-\frac{\kappa}{x}}|\theta|^{-\sigma}(1+O(\theta))\, d\theta  & = \frac1{\sigma-1}\left( \left(\frac{x}{\kappa}\right)^{\sigma-1}- 2^{\sigma-1} \right) + O(x^{\sigma-2}).
\end{split}
\end{equation*}
Plugging this information into~\eqref{bvcneiwry389trguefbvadk09876},
and recalling~\eqref{aggiuntaforse}, we obtain~\eqref{o399c63v0c6}.

Now, we show that
\begin{equation}\label{fcpmdbis}
\lim_{x \to + \infty} x^{1+2s} B_K(x) \geq  \lambda \int_{-\kappa}^{\kappa} \phi(y)\, dy.
\end{equation}
For this, we employ~\eqref{aggiuntaforse} to see that for all~$\epsilon>0$
there exists~$M>0$ such that if~$x\geq M$ then~$x\phi(x)\le {\epsilon}$.
Also,we take~$\gamma$ as in~\eqref{addipotesi76bis}
and we suppose that~$x>\max\{M, \epsilon/\gamma\}$.
In this way, we have that, for all~$y\in(-\kappa,\kappa)$,
$$ \phi(x)\leq\frac{\epsilon}x <\gamma\le\phi(y),$$
thanks to~\eqref{addipotesi76bis}.

Therefore, in light of the lower bound in~\eqref{newbound}, we conclude that, if~$x$ is sufficiently large, 
\begin{eqnarray*}
x^{1+2s}B_K(x)&=& x^{1+2s}
\int_{-\kappa}^{\kappa} \big( \phi(y) -\phi(x)\big) K(x-y) \, dy\\
&\geq& \lambda \int_{-\kappa}^{\kappa} \left( \phi(y)-\phi(x) \right) \left| 1 -\frac{y}{x} \right|^{-(1+2s)}\, dy.
\end{eqnarray*}
We can now take the limit as~$x\to+\infty$ and use~\eqref{aggiuntaforse} and the
Dominated Convergence Theorem to deduce~\eqref{fcpmdbis}.

Also, we claim that
\begin{equation}\label{fcpmtbis}
\lim_{x \to + \infty} x ^{1+2s} D_K(x) \geq \frac{\bar{C}\lambda\kappa^{1-\tau}}{\tau-1}.
\end{equation}
In order to show this, in the notation of~\eqref{53BIS}, we write~$D_K(x) = D_{K,I}(x)+D_{K,II}(x)+D_{K,III}(x)$.

We exploit~\eqref{newbound} to estimate
\begin{equation*}
\begin{split}
x^{1+2s} D_{K,I}(x) &=  x^{1+2s} \int_{\kappa}^{\frac{x}2} \phi(y) K(x-y)\, dy - x^{1+2s}\phi(x)  \int_{\kappa}^{\frac{x}2} K(x-y)\, dy \\
&\geq  \bar{C}\lambda \int_{ \kappa}^{\frac{x}2}y^{-\tau}\left(1-\frac{y}{x}\right)^{-(1+2s)} \, dy - \Lambda\phi(x)\int_{\kappa}^{\frac{x}2} \left(1-\frac{y}{x}\right)^{-(1+2s)} \, dy\\
&= \bar{C} \lambda x^{1-\tau}  \int_{ \frac{\kappa}{x}}^{\frac12}\theta^{-\tau}(1+O(\theta)) \, d\theta - \Lambda x\phi(x)
\int_{\frac{\kappa}{x}}^{\frac{1}2} \frac{d\theta}{(1-\theta)^{1+2s}}\\
&\geq \frac{\bar{C}\lambda x^{1-\tau}}{\tau-1}  \left( \left(\frac{x}{\kappa}\right)^{\tau-1} - 2^{\tau-1}\right) + O(x^{-1})-Cx\phi(x).
\end{split}
\end{equation*}
Hence, recalling also~\eqref{aggiuntaforse},
\begin{equation}\label{mogo}
\lim_{x \to +\infty} x^{1+2s} D_{K,I}(x)\geq \frac{\bar{C}\lambda\kappa^{1-\tau}}{\tau-1}.
\end{equation}

Furthermore, we recall that~$\phi' $ is Lipschitz continuous and that~$\phi'' \in L^{\infty}(\R)$. 
Thus, changing variable~$y:=x\theta$ and using~\eqref{vcbnxe7598760t9ewh9oigtu4}, the symmetry of~$K$ in~\eqref{krn_symm} and the bound on~$K$
in~\eqref{newbound} yield that
\begin{equation*}
\begin{split}
x^{1+2s} |D_{K,II}(x)|& = x^{2+2s}\left| \int_{\frac12}^{\frac32}\left(\phi(x\theta)-\phi(x)\right)K(x(1-\theta)) \, d\theta\right| \\
&= x^{2+2s}\left| \int_{\frac12}^{\frac32}\left(\int_{x}^{x\theta}\int_x^\tau \phi''(\eta)\, d\eta\, d\tau\right)K(x(1-\theta)) \, d\theta\right| \\
&\leq x^{4+2s}\Vert \phi''\Vert_{L^{\infty}(\frac{x}2,\frac{3x}2)}   \int_{\frac12}^{\frac32} (\theta-1)^2  K(x(1-\theta)) \, d\theta\\
&\leq \Lambda x^3 \Vert \phi''\Vert_{L^{\infty}(\frac{x}2,\frac{3x}2)}  \int_{\frac12}^{\frac32}\left|1-\theta\right|^{1-2s} \, d\theta.
\end{split}
\end{equation*}
Exploiting~\eqref{93n7w7cf} we threby obtain that 
\begin{equation}\label{NODi}
\lim_{x\to+\infty}x^{1+2s} |D_{K,II}(x)|=0.
\end{equation}


In addition, from the upper bound in~\eqref{newbound} we obtain
\begin{equation*}
x^{1+2s} |D_{K,III}(x)|  \leq  \Lambda \int_{\frac{3x}2}^{+\infty}(\phi(y)+\phi(x)) \left(\frac{y}{x}-1\right)^{-(1+2s)}\, dy.\end{equation*}
Moreover, in light of~\eqref{aggiuntaforse} we have that for all~$\epsilon\in(0,1)$ there exists~$M>0$ such that if~$x\geq M$ then~$x\phi(x)\le\epsilon$.
Thus, if~$x\geq M$ and~$y\geq \frac{3x}2$, we see that
\begin{eqnarray*}
\phi(y)+\phi(x)\le \epsilon\left(\frac1y+\frac1x\right)\leq \epsilon
\left(\frac2{3x}+\frac1x\right)\leq\frac{C\epsilon}{x}.
\end{eqnarray*}
Consequently, if~$x$ is sufficiently large,
\begin{eqnarray*}
x^{1+2s} |D_{K,III}(x)|  \leq  \frac{C\epsilon}{x} \int_{\frac{3x}2}^{+\infty} \left(\frac{y}{x}-1\right)^{-(1+2s)}\, dy.
\end{eqnarray*}
The change of variable~$y:=x\theta$ now gives that
\begin{eqnarray*}
\lim_{x\to+\infty}x^{1+2s} |D_{K,III}(x)|  \leq  C\epsilon \int_{\frac{3}2}^{+\infty} \frac{d\theta}{(\theta-1)^{1+2s}}\leq C\epsilon.
\end{eqnarray*}
Since~$\epsilon$ is arbitrary, we thereby conclude that
$$ \lim_{x\to+\infty}x^{1+2s} |D_{K,III}(x)|=0.$$

This observation and~\eqref{NODi} yield that
\begin{equation*}
 \lim_{x\to+\infty} x^{1+2s} D_K(x) =  \lim_{x\to+\infty} x^{1+2s} D_{K,I}(x),
\end{equation*}
which, together with~\eqref{mogo}, lead to~\eqref{fcpmtbis}.

Finally, we recall that~$L_K\phi(x) = A_K(x)+B_K(x)+D_K(x)$, therefore~\eqref{o399c63v0c6}, \eqref{fcpmdbis} and~\eqref{fcpmtbis} together give the desired result.
\end{proof}    

\section{Proof of Theorem~\ref{main_thm}}\label{allnzehs}
In this section, we prove Theorem~\ref{main_thm}. Specifically, Section~\ref{weipr} establishes the estimates in~\eqref{asymp_decay_lowbound}, while Section~\ref{weilse} the ones in~\eqref{eq:asymp-derivata}.

We begin with a preliminary result on the potential~$W$.
   \begin{lemma}[Lemma~4.1 in~\cite{DPDV}]\label{thc_deg}
Let~$W:\R \to\R$ be a function satisfying~\eqref{pot_reg} and~\eqref{pot_deg}.

Then, for any~$r$, $t \in [-1,-1+\mu]$ with~$r\leq t$, 
\begin{equation}\label{5968hht}
\begin{split}
&\displaystyle\frac{c_1}{\alpha(\alpha-1)}\left((1+t)^{\alpha}-(1+r)^{\alpha}\right) \leq W(t)-W(r) \leq \frac{c_2}{\beta(\beta-1)}\left((1+t)^{\beta}-(1+r)^{\beta}\right) \\ &\mbox{and}\\
&\displaystyle\frac{c_1}{\alpha-1} \left( (1+t)^{\alpha-1} - (1+r)^{\alpha-1}\right) \leq W\rq{}(t) - W\rq{}(r) \leq \frac{c_2}{\beta-1} \left((1+t)^{\beta-1}-(1+r)^{\beta-1}\right).
\end{split}
\end{equation}
Moreover, for any~$r$, $t \in [1-\mu,1]$ with~$r\leq t$,
\begin{equation}\label{eoi3j59}
\begin{split}
&  \displaystyle\frac{c_3}{\gamma(\gamma-1)} \left((1-t)^{\gamma}-(1-r)^{\gamma}\right)\leq W(t) - W(r) \leq \displaystyle\frac{c_4}{\delta(\delta-1)} \left( (1-t)^{\delta} - (1-r)^{\delta}\right) \\& \mbox{and} \\
&\displaystyle\frac{c_3}{\gamma-1} \left( (1-r)^{\gamma-1} - (1-t)^{\gamma-1}\right) \leq W\rq{}(t) - W\rq{}(r) \leq \frac{c_4}{\delta-1} \left((1-r)^{\delta-1}-(1-t)^{\delta-1}\right).
\end{split}
\end{equation}
\end{lemma}

\subsection{Proof of~\eqref{asymp_decay_lowbound}} \label{weipr}
Let~$\bar{u} \in  \mathcal{X}$ be the class~$A$ minimizer for the energy in~\eqref{main_fun}. Its existence, uniqueness and regularity properties are guaranteed by Theorem~\ref{vogpian}. Also, without loss of generality, we can assume that~$\bar{u}(0)=0$.

We know by Lemma~\ref{thc_deg} that there exist~$\widetilde{C}>0$ and~$\bar{r} \in (0,1)$ such that, for any~$r \in[\bar{r},1]$,
\[ W\rq{}(r) = W\rq{}(r)- W\rq{}(1)\geq - \frac{\widetilde{C}}{\delta-1} (1-r)^{\delta-1}.  \]

Now we let~$\bar{x}\in\R$ be such that~$\bar{u}(\bar{x})=\bar{r}$. By possibly taking~$\bar{r}$ closer to~$1$, we also suppose that~$\bar{x}\geq 2$
and that
\begin{equation}\label{bvcnxmoe8239562345678}
\bar u(2\bar x)>\frac12 \end{equation}
(notice that this is possible since~$\bar{u} \in \mathcal{X}$ and~$\bar{u}$ is continuous).

Then, recalling that~$\bar{u}$ satifies~\eqref{all_ca_frl} and that~$\bar{u}$ is increasing in~$\R$, we have that, for any~$x \geq \bar{x}$,
\begin{equation}\label{5487507080f}
L_K \bar{u}(x) = W\rq{}(\bar{u}(x)) \geq -\frac{\widetilde{C}}{\delta-1} (1-\bar{u}(x))^{\delta-1}.
\end{equation}

Now, we set
\begin{equation}\label{bvcniewory34859y234}
\bar{\alpha} := \min \left\{(1-\bar{u}(2\bar{x})) \bar{x}^{\frac{2s}{\delta-1}} , \,\frac{\lambda}{8C} , \, \left( \frac{\lambda}{8\widetilde{C}}\right)^{\frac1{\delta-1}} \right\},\end{equation}
where~$\lambda$ is the quantity in~\eqref{newbound} while ~$C$ is the constant in~\eqref{8fe0re34}, depending on~$s$, $\lambda$ and~$\Lambda$.

Moreover, we consider the function
\begin{equation}\label{gidefi76543}\phi(x):=
\begin{cases}
\displaystyle\frac12 &\mbox{if } x \leq 0 ,\\
\bar{u}(\bar{x}) &\mbox{if } x \in (0,\bar{x}),\\
1-\bar{\alpha} x^{-{\frac{2s}{\delta-1}}} &\mbox{if } x \geq \bar{x}.
\end{cases}
\end{equation}
We stress that, by definition of~$\bar{\alpha}$, we have that
\begin{equation}\label{njudgior98768765}
\phi(\bar{x})=1-\bar{\alpha} \bar{x}^{-{\frac{2s}{\delta-1}}}\geq 1-(1-\bar{u}(2\bar{x})) \bar{x}^{\frac{2s}{\delta-1}} \bar{x}^{-{\frac{2s}{\delta-1}}}
= \bar{u}(2\bar{x}).\end{equation}

We claim that
\begin{equation}\label{ioebvvcbv332}
\bar{u}(x)<\phi(x)  \quad \mbox{for any } x \in (-\infty,2\bar{x}).
\end{equation}
Indeed, if~$ x \leq 0$, from the monotonicity of~$\bar{u}$ and the fact that~$\bar{u}(0)=0$, we obtain that
$$\bar{u}(x)\le \bar{u}(0)=0<\frac12=\phi(x).$$
If~$x\in(0,\bar{x})$, then~$ \bar{u}(x)<\bar{u}(\bar{x})=\phi(x)$, thanks to the monotonicity of~$\bar{u}$.

If instead~$x\in[\bar{x},2\bar{x})$, then, by the monotonicity of~$\bar{u}$ and~\eqref{njudgior98768765},
\begin{eqnarray*}
\bar{u}(x)<\bar{u}(2\bar{x})\leq \phi(\bar{x})=1-\bar{\alpha} \bar{x}^{-{\frac{2s}{\delta-1}}}\le1-\bar{\alpha} x^{-{\frac{2s}{\delta-1}}}=\phi(x).
\end{eqnarray*}
{F}rom these observations, we deduce~\eqref{ioebvvcbv332}.

Also, by the definition of~$\phi$, it holds that
\begin{equation}\label{vggvcplkki95466}
x^{-2s} = ( 1 - \phi(x) )^{\delta-1} \bar{\alpha}^{-(\delta-1)} \quad\mbox{for any } x\geq \bar{x}.
\end{equation}

Now we would like to exploit Proposition~\ref{xcs54340m} with~$\phi_{\bar{x}}:=\phi$ as in~\eqref{gidefi76543}.
This is possible with the choices~$B:=1/2$, $D:=\bar{u}(\bar{x})$, $\alpha:=\bar{\alpha}$ and~$A:=\frac{2s}{\delta-1}$. 
Indeed, by the definition of~$\bar\alpha$ in~\eqref{bvcniewory34859y234}
and~\eqref{bvcnxmoe8239562345678},
$$ \bar\alpha\le (1-\bar{u}(2\bar{x})) \bar{x}^{\frac{2s}{\delta-1}}\le
\frac12 \bar{x}^{\frac{2s}{\delta-1}}.$$
This implies that
$$ \frac12< 1-\bar\alpha \bar{x}^{-\frac{2s}{\delta-1}}.$$
Using the monotonicity of~$\bar u$, we also get that
$$\bar{u}(\bar x)<\bar u(2\bar{x})\leq 1-\bar\alpha \bar{x}^{-\frac{2s}{\delta-1}}.
$$
Thus, we are in the position of
employing Proposition~\ref{xcs54340m}, obtaining that, for any~$ x \geq 2\bar{x}\geq 4$,
\begin{equation}\label{4653467xpp}
x^{2s} L_K\phi(x) \leq - \frac{\lambda}{4s} +\bar{\alpha}Cx^{-\frac{2s}{\delta-1}} 
\leq  - \frac{\lambda}{4s} +\bar{\alpha}C 
\leq - \frac{\lambda}{4s} + \frac{\lambda}{8} \leq - \frac{\lambda}{8s}.
\end{equation}

We claim that
\begin{equation}\label{04jvx8uf}
\bar{u}(x) \leq \phi(x) \quad\mbox{for any } x \in \R.
\end{equation}
In order to show~\eqref{04jvx8uf}, we define, for any~$b \in [0,+\infty)$, the function~$w_b:= \phi +b -\bar{u}$. 
We point out that~$w_b\ge \phi +1 -\bar{u}>0$ in~$\R$ for any~$b \in[ 1,+\infty)$. 

Now, if~$w_b>0$ in~$\R$ for any~$b \in [0,+\infty)$, then the claim in~\eqref{04jvx8uf} plainly follows
by taking~$b=0$. Hence, from now on, we suppose that there exists~$b_0\in(0,1)$ such that~$w_b>0$ for any~$b \in (b_0,+\infty)$ and~$w_{b_0}(z)=0$ for some~$z \in \R$. Furthermore, relying on~\eqref{ioebvvcbv332}, we conclude that any~$z$ such that~${w_{b_0}(z)=0}$ belongs to the set~$[2\bar{x},+\infty)$.

Also, recalling the definition of~$\bar\alpha$ in~\eqref{bvcniewory34859y234} and the fact that~$\delta\geq 2$, we see that
\begin{equation*}
\bar{\alpha}^{\delta-1} \leq \frac{\lambda}{8\widetilde{C}} \leq \frac{\lambda(\delta-1)}{8 \widetilde{C}s}.
\end{equation*}
{F}rom this, \eqref{5487507080f}, \eqref{vggvcplkki95466} and~\eqref{4653467xpp}, we obtain that
\begin{equation}\label{9cerdcgf657}
\begin{split}
L_K w_{b_0} (z)& =L_K \phi (z) - L_K \bar{u}(z) \leq- \frac{\lambda z^{-2s}}{8s} + \frac{\widetilde{C}}{\delta-1} (1-\bar{u}(z))^{\delta-1}\\
&= - \frac{\lambda \bar{\alpha}^{-(\delta-1)}}{8s}( 1 - \phi(z))^{\delta-1}  + \frac{\widetilde{C}}{\delta-1} (1-\bar{u}(z))^{\delta-1}\\
&\leq  - \frac{\widetilde{C}}{\delta-1} \big(
(1 - \phi(z))^{\delta-1} - (1 - \bar{u}(z))^{\delta-1} \big)\\
&= - \frac{\widetilde{C}}{\delta-1}   \big( (1 - \phi(z))^{\delta-1} -(1 - \phi(z) - b_0)^{\delta-1} \big)\\
&<0.
\end{split}
\end{equation}

On the other hand, the fact that~$w_{b_0}(z)=0$ implies that
\begin{equation*}
L_K w_{b_0}(z) = \int_{\R} (w_{b_0}(y) - w_{b_0}(z)) K(z-y)\, dy =   \int_{\R} w_{b_0}(y)K(z-y)\, dy \geq 0.
\end{equation*}
This is in contradiction with~\eqref{9cerdcgf657} and
therefore the proof of~\eqref{04jvx8uf} is complete.

As a consequence of~\eqref{04jvx8uf}, there exist~$C_1$, $R >0$ such that, for any~$x\geq R$,
\[ \bar{u}(x) \leq 1 - C_1 x^{-\frac{2s}{\delta-1}}\]
and this establishes the second estimate in~\eqref{asymp_decay_lowbound}.

We now show the first estimate in~\eqref{asymp_decay_lowbound}. To this aim, we define the function~$\bar{v}(x) := -\bar{u}(-x)$ and observe that~$\bar{v}$ inherits the regularity properties of~$\bar{u}$. Moreover, $\bar{v}$ is strictly increasing and belongs to~$\mathcal{X}$.

In addition, exploiting~\eqref{all_ca_frl}, we obtain that, for any~$x \in \R$,
\begin{equation*}
L_K \bar{v}(x) = - L_K\bar{u}(-x) = - W'(\bar{u}(-x)).
\end{equation*}
Therefore, by~\eqref{5968hht}, there exist~$\widetilde{C}$, $\bar{x}>0$ such that, for any~$x \geq \bar{x}$,
\begin{equation*}
L_K\bar{v}(x) = W'(-1) - W'(\bar{u}(-x)) \geq - \frac{\widetilde{C}}{\beta-1}(1 +\bar{u}(-x))^{\beta-1}=-
\frac{\widetilde{C}}{\beta-1}(1 -v(x))^{\beta-1}.
\end{equation*}

Hence, we are in the position of exploiting the first part of this proof that took care of the second estimate in~\eqref{asymp_decay_lowbound}, applied now to~$\bar{v}$, with the only caveat that Proposition~\ref{xcs54340m} must be used here with~$A:= \frac{2s}{\beta-1}$ and
\begin{equation*}
\alpha:= \min \left\{ ( 1 - \bar{v}(2\bar{x})) \bar{x}^{\frac{2s}{\beta-1}}, \frac{\lambda}{8C}, \left( \frac{\lambda}{8\widetilde{C}}\right)^{\frac1{\beta-1}} \right\},
\end{equation*}
where~$C$ is the constant appearing in~\eqref{8fe0re34}, depending on~$s$, $\lambda$ and~$\Lambda$.

In this way, we obtain that
\begin{equation*}
1-\bar{v}(x) \geq C_1 x^{-\frac{2s}{\beta-1}} \quad\mbox{if } x\geq R.
\end{equation*}
Namely,
        $$
          1+\bar{u}(-x) \geq C_1 x^{-\frac{2s}{\beta-1}} \quad\mbox{if } x\geq R,$$
which gives the first estimate in~\eqref{asymp_decay_lowbound}, as desired.

\subsection{Proof of~\eqref{eq:asymp-derivata}}\label{weilse}
Let~$\bar{u} \in  \mathcal{X}$ be the class~$A$ minimizer for the energy in~\eqref{main_fun}. Its existence, uniqueness and regularity properties are guaranteed by Theorem~\ref{vogpian}.

The regularity of~$\bar{u}$ allows us to differentiate~\eqref{all_ca_frl} and find that, for any~$x\in\R$,
$$ {L}_K \bar{u}\rq{}(x) = W\rq{}\rq{}(\bar{u}(x))  \bar{u}\rq{} (x).$$
Therefore, from~\eqref{pot_deg}, \eqref{asymp_decay_lowbound} and the strict monotonicity of~$\bar{u}$
we obtain that there exists~$\bar{x}>0$ such that
\begin{equation}\label{90377djoi}
\begin{split}
L_K \bar{u}\rq{}(x) &\geq
\begin{cases}\widetilde{C}(1+\bar{u}(x))^{\alpha-2}\bar{u}\rq{}(x) &\mbox{if } x\leq -\bar{x}, \\
\widetilde{C} (1-\bar{u}(x))^{\gamma-2}\bar{u}\rq{}(x) &\mbox{if } x\geq \bar{x}
\end{cases}\\
&\geq \begin{cases}
\widetilde{C} |x |^{-\frac{2s(\alpha-2)}{\beta-1}}\bar{u}\rq{}(x) &\mbox{if } x\leq -\bar{x},\\
\widetilde{C}|x |^{-\frac{2s(\gamma-2)}{\delta-1}}\bar{u}\rq{}(x) &\mbox{if } x\geq \bar{x},
\end{cases}
\end{split}
\end{equation}
up to relabeling~$\widetilde{C}>0$.

Possibly taking~$\bar{x}$ large (and recalling the assumption in~\eqref{riar}), we suppose that
\begin{equation}\label{ccu7878}
 \frac{\bar{x}^{-
\frac{2s(\beta-\alpha+1)}{\beta-1}}( \beta-1)}{2s(\beta-\alpha+1)}+
\frac{\bar{x}^{-\frac{2s(\delta -\gamma+1) }{\delta-1}}(\delta-1)}{2s(\delta -\gamma+1)} \leq \frac{\widetilde{C}}{4\Lambda},
\end{equation}
where~$\Lambda$ is the quantity appearing in~\eqref{newbound}.

Now, we consider~$\phi \in C^{\infty}(\R, (0,+\infty))$ such that
\begin{equation}\label{3wertyjxsdcvb0ffdj88}
\phi(x) := \begin{cases}
|x |^{-\left(1+\frac{2s(\beta-\alpha+1)}{\beta-1}\right)}  &\mbox{if } x \leq -\bar{x}, \\
|x |^{-\left(1+\frac{2s(\delta-\gamma+1)}{\delta-1}\right)}  &\mbox{if } x \geq \bar{x}
\end{cases}
\end{equation}
and
\begin{equation}\label{pprodo}
\int_{-\bar{x}}^{\bar{x}} \phi(x) \, dx \leq \frac{\widetilde{C}}{4\Lambda}.
\end{equation}

The aim is to employ Proposition~\ref{tonto} with~$\phi$ as in~\eqref{3wertyjxsdcvb0ffdj88}. 
Indeed, thanks to the assumption in~\eqref{riar}
we can choose
\begin{eqnarray*}&&\bar{C}:=1,\qquad \kappa:=\bar{x},\qquad \sigma:=1+\frac{2s(\beta-\alpha+1)}{\beta-1}\qquad {\mbox{and}}\qquad
\tau:=1+\frac{2s(\delta-\gamma+1)}{\delta-1}.\end{eqnarray*}
Also,
\begin{eqnarray*}
\lim_{x\to+\infty}x^3|\phi''(x)|=\lim_{x\to+\infty}x^3 x ^{-\left(3+\frac{2s(\delta-\gamma+1)}{\delta-1}\right)}=0
\end{eqnarray*}
and similarly
\begin{eqnarray*}
\lim_{x\to-\infty}|x|^3|\phi''(x)|=\lim_{x\to-\infty}x^3 |x |^{-\left(3+\frac{2s(\beta-\alpha+1)}{\beta-1}\right)}=0,
\end{eqnarray*}
which entail that the assumption in~\eqref{7f9c7w3n} is satisfied.

Thus, we are in the position of using Proposition~\ref{tonto}, from which we obtain that
$$
\lim_{|x | \to  +\infty} |x|^{1+2s} {L}_K\phi(x) \leq \Lambda \left( \frac{\bar{x}^{-
\frac{2s(\beta-\alpha+1)}{\beta-1}}( \beta-1)}{2s(\beta-\alpha+1)}+ \int_{-\bar{x}}^{\bar{x}} \phi(y) \, dy+
\frac{\bar{x}^{-\frac{2s(\delta -\gamma+1) }{\delta-1}}(\delta-1)}{2s(\delta -\gamma+1)}
\right).
$$
This, \eqref{ccu7878} and~\eqref{pprodo} lead to
\begin{equation*}
\lim_{|x | \to  +\infty} |x|^{1+2s} {L}_K\phi(x) \leq \Lambda \left( \frac{\widetilde{C}}{4\Lambda} + \frac{\widetilde{C}}{4\Lambda} \right) = \frac{\widetilde{C}}{2}.
\end{equation*}
As a consequence, there exists~$x_1\geq\bar{x}$ such that, for any~$|x|\geq x_1$,
\begin{equation}\label{1afdbf75gjviy87jjmnmkjo} L_K\phi(x)\leq \frac{\widetilde{C}}{|x|^{1+2s}}.\end{equation}

We now point out that, if~$x\le -\bar{x}$,
\begin{equation*}
\frac{\phi(x)}{|x|^{\frac{2s(\alpha-2)}{\beta-1}}}
=\frac{|x |^{-\left(1+\frac{2s(\beta-\alpha+1)}{\beta-1}\right)}}{|x|^{\frac{2s(\alpha-2)}{\beta-1}}}=\frac1{|x|^{1+2s}}
\end{equation*}
and similarly, if~$x\ge \bar{x}$,
\begin{equation*}
\frac{\phi(x)}{|x|^{\frac{2s(\gamma-2)}{\delta-1}}}
=\frac{|x |^{-\left(1+\frac{2s(\delta-\gamma+1)}{\delta-1}\right)}}{|x|^{\frac{2s(\gamma-2)}{\delta-1}}}=\frac1{|x|^{1+2s}}.
\end{equation*}
From these observations and~\eqref{1afdbf75gjviy87jjmnmkjo}, we deduce that
\begin{equation}\label{p48d4d}
	L_K \phi (x) \leq \begin{cases}
\widetilde{C} |x|^{-\frac{2s(\alpha-2)}{\beta-1}} \phi(x) &\mbox{if } x \leq -x_1,\\
\widetilde{C} |x|^{-\frac{2s(\gamma-2)}{\delta-1}} \phi(x) &\mbox{if } x \geq x_1.
\end{cases}
\end{equation}

Now, we set 
\begin{equation*}
\widehat{C}:=\max_{x\in[-x_1,x_1]}\bar{u}\rq{}(x) \left(  \min_{x \in [-x_1,x_1]}\phi(x) \right)^{-1}
\end{equation*}
and we notice that
\begin{equation}\label{546HJHDGILUIDyre65t}
\widehat{C} \phi(x)  - \bar{u}\rq{} (x) \geq 0\quad  {\mbox{for any }}|x | \leq x_1. 
\end{equation}

We claim that
\begin{equation}\label{039hgvce0}
\widehat{C} \phi(x)  \geq \bar{u}\rq{} (x) \quad\mbox{for any } x \in \R.
\end{equation}
In order to prove the claim, we define, for any~$b \in [0,+\infty)$, the function~${v_b:=\widehat{C} \phi +b - \bar{u}\rq{} }$.
Since~$\bar{u} \in \mathcal{X}$, we have that~$\bar{u}\rq{}$ is bounded. As a consequence, the strict positivity of~$\phi$ implies that~$v_b>0$ in~$\R$ for any~$b \geq \Vert \bar{u}\rq{}\Vert_{L^{\infty}(\R)}$.

Now, if~$v_b > 0$ in~$\R$ for any~$b \in [0,+\infty)$, the claim in~\eqref{039hgvce0} plainly follows
by taking~$b=0$. Hence, from now on, we suppose that there exists~$b_0\in(0,\Vert \bar{u}\rq{}\Vert_{L^{\infty}(\R)})$
such that~$v_b>0$ for all~$b\in(b_0,+\infty)$ and~$v_{b_0}(z)=0$ at some point~$z\in\R$.

By the definition of~$b_0$, there exist points~$x_k$ such that
\begin{equation}\label{bvn8iw97t4837hfwsallihsa}
v_{b_0}(x_k)< \frac1{2^{k}}.\end{equation} Without loss of generality, we may suppose that~$x_k\geq 0$ (otherwise,
in what follows, we use the information coming from the decay at~$-\infty$).

Furthermore, the sequence~$x_k$ is bounded from above since, if not, we would have
\begin{equation*}
b_0 = \limsup_{k \to +\infty} \left( \widehat{C}\phi(x_k) +b_0 -\bar{u}\rq{}(x_k)  \right) = \lim_{k \to +\infty} v_{b_0}(x_k) = 0,
\end{equation*}
which is a contradiction. 

Moreover, exploiting~\eqref{546HJHDGILUIDyre65t}, we obtain that,
when~$ |x | \leq x_1$ and~$k > -\log_{2}b_0$,
\begin{equation*}
\frac1{2^{k}} < b_0 \leq \widehat{C}\phi(x) - \bar{u}\rq{} (x) +b_0 = v_{b_0}(x).
\end{equation*}
As a consequence, in light of~\eqref{bvn8iw97t4837hfwsallihsa}, we have that~$x_k\in(x_1,+\infty)$ for any~$k$ sufficiently large.

Gathering these pieces of information,
we conclude that there exists~$x_{\infty} \in (x_1,+\infty)$ such that~$x_k \to x_{\infty}$ as~$k\to+\infty$, up to a subsequence. 
The continuity of~$v_{b_0}$ and~\eqref{bvn8iw97t4837hfwsallihsa} give that~$v_{b_0}({x_{\infty}})=0$.

As a result, since~$x_{\infty}\geq x_1\geq \bar{x}$, we can exploit~\eqref{90377djoi} and~\eqref{p48d4d} to compute
\begin{equation*}
\begin{split}
L_K v_{b_0} ({x_{\infty}}) &= \widehat{C}  L_K \phi ({x_{\infty}})-  L_K \bar{u}\rq{}({x_{\infty}})\\
&\leq \widetilde{C} | {x_{\infty}} |^{-\frac{2s(\gamma-2)}{\delta-1}} \big(\widehat{C} \phi ({x_{\infty}}) - \bar{u}\rq{}({x_{\infty}})\big)\\
&= - b_0 \widetilde{C}| {x_{\infty}}|^{-\frac{2s(\gamma-2)}{\delta-1}}\\&<0.
\end{split}
\end{equation*}
On the other hand, we have that
\begin{equation*}
L_K v_{b_0} ({x_{\infty}}) = \int_{\R} (v_{b_0}(y)-v_{b_0}({x_{\infty}})) K({x_{\infty}}-y) \, dy = \int_{\R} v_{b_0}(y) K({x_{\infty}}-y) \, dy \geq 0.
\end{equation*}
We thereby obtain the desired contradiction, which completes the proof of~\eqref{039hgvce0}.

Formulas~\eqref{3wertyjxsdcvb0ffdj88} and~\eqref{039hgvce0} yield  the estimates in~\eqref{eq:asymp-derivata}.

\section{Towards the optimality statements}\label{3sstzac}

In this section, we construct a function~$\widetilde{u}$ that will play a central role in the proof of Theorem~\ref{optimal}
in Section~\ref{proptimal}. Due to its stepwise nature, the construction is somewhat intricate and requires careful treatment. For this reason, we have organized the content into three subsections.

Section~\ref{notas} introduces the preliminary notations and definitions that will be used throughout the rest of the article.
In Section~\ref{constru}, we deal with the actual construction of the function~$\widetilde{u}$. Finally, Section~\ref{n5m7b38} is devoted to establishing the main properties of~$\widetilde{u}$. We point out that some of these properties will be essential in the proof of Theorem~\ref{optimal}, while others are used in Appendix~\ref{sancheck} for the ``sanity check'' on the function~$\widetilde{u}$.

\subsection{Notations}\label{notas}
In the following, we introduce some notations and definitions that will be used in Sections~\ref{constru} and~\ref{n5m7b38}, as well as throughout the rest of the paper.

Let~$s \in (0,1)$, $\alpha >\beta\geq 2$ and~$\gamma > \delta \geq 2$ such that~$\alpha<\beta+1$ and~$\gamma<\delta+1$. We set
\begin{equation}\label{bcuewoiyr8439t098765412345asdfg}
A:=\frac{2s}{\delta-1}, \qquad  B:=\frac{2s}{\gamma-1}, \qquad  D:=\frac{2s}{\beta-1} \qquad\mbox{and}\qquad  E:=\frac{2s}{\alpha-1}.\end{equation}
We observe that
\begin{equation}\label{bcuwq12345678poiuytre0987}
\frac{A}{B} \ln\left(\frac{A}{B}\right) \leq 2 \ln2 \qquad\mbox{and}\qquad \frac{D}{E} \ln\left(\frac{D}{E}\right) \leq 2 \ln2.
\end{equation}

Also, we consider~$\eta\in C^{\infty}([0,1], [0,+\infty))$ such that~$\eta = 1$ in~$[0,1/4]$, $\eta=0$ in~$[3/4,1]$ and~$\eta' \in(-4,0)$ in~$(1/4,3/4)$. Moreover, we set
\begin{equation}\label{notations98}
\bar{\eta}:= \max_{\substack{x \in [0,1] \\ i \in \{0,1,2,3\}}}\big\{ |\eta^{(i)}(x) |\big\} \qquad\mbox{and}\qquad \eta_0:= 
\min\left\{ \eta\left(\frac12\right);\, 1-\eta\left(\frac12\right)\right\},
\end{equation} where~$\eta^{(i)}$ denotes the~$i^{th}$-derivative of~$\eta$ (with the implicit understanding that~$\eta^{(0)}:=\eta$).

We stress that~$ \eta_0\leq\eta(\frac12)\leq \bar\eta$, and therefore
\begin{equation}\label{bvrikebgifuegwouw5636598jbvkdskj}
\frac{\bar\eta}{\eta_0}\geq 1.
\end{equation}

In addition, for any~$x \in [0,1]$ we define the functions
\begin{equation}\label{defetatilde09876}
i(x) :=x^3(1-x)^3 \qquad\mbox{and}\qquad \widetilde{\eta}(x):= \frac{\displaystyle \int_x^1 i(t)\, dt}{{\mathcal{B}}(4,4)},
\end{equation} where~${\mathcal{B}}(\cdot,\, \cdot)$ denotes the Euler Beta function.

We point out that
$$ {\mathcal{B}}(4,4)=\int_0^1 i(t)\, dt,$$
therefore~$\widetilde\eta$ enjoys the properties: $\widetilde{\eta}(0)=1$, $\widetilde{\eta}(1)=0$ and~$\widetilde{\eta}$ is strictly decreasing in~$(0,1)$.

Furthermore, we set
\[ a_0:= \max \big\{4,  e^{\frac1{B}}\big\} \]
and we define recursively in~$k \in \N$ the sequences
\begin{equation*}
b_k := a_k +1, \qquad
c_k := b_k^{e^{ \frac{128\bar{\eta}}{\eta_0}}},\qquad
d_k := 2c_k \qquad\mbox{and}\qquad
a_{k+1} := 2d_k.
\end{equation*}
We observe that, thanks also to~\eqref{bvrikebgifuegwouw5636598jbvkdskj},
\begin{equation}\label{bvrikebgifuegwouw5636598jbvkdskj2}
b_k<2b_k<\frac{c_k}2<c_k<d_k<a_{k+1}<b_{k+1}.
\end{equation}

 \medskip

For any~$x\in[0,1]$, we consider the functions
\begin{equation}\label{funcome}
w_1(x):= B \widetilde{\eta}(x)- (x+1)\widetilde{\eta}'(x) \qquad\mbox{and}\qquad w_2(x):= E \widetilde{\eta}(x)- (x+1)\widetilde{\eta}'(x).
\end{equation}
Moreover, we define
\begin{equation}\label{semil}
\bar{x}_1:=\frac{3(4-B)+\sqrt{B^2-8B+88}}{2(7-B)}-1 \qquad\mbox{and}\qquad \bar{x}_2:=\frac{3(4-E)+\sqrt{E^2-8E+88}}{2(7-E)}-1.
\end{equation}
We point out that both~$\bar{x}_1$ and~$\bar{x}_2$ belong to~$(0,1)$ (see Lemma~\ref{insicilia} for a proof of this fact).

The choice of the points~$\bar{x}_1$ and~$\bar{x}_2$ seems mysterious at this stage, but it becomes clear if one looks at the first
derivative of~$w_1$ and~$w_2$: as a matter of fact, $\bar{x}_1$ and~$\bar{x}_2$
are the only critical points of~$w_1$ and~$w_2$ in the interval~$(0,1)$ (see formula~\eqref{0987654qwertyuilkjhgfds}),
and this will entail some monotonicity properties for the functions~$w_1$
and~$w_2$ (see Lemma~\ref{i03y3}).
\medskip

Now we set, for any~$k \in\N$,
\begin{equation}\label{bvoiewyf83097u038012678}
C_2:= 2 \max\left\{ \frac1{B}, \left(1-\frac{B}{w_1(\bar{x}_1)}\right)^{-1} \right\}, \qquad C_{1,k}:= C_2- \frac{BC_2-((\bar{x}_1+1)c_k)^{-B(\gamma-\delta)}}{w_1(\bar{x}_1)},
\end{equation}
and
\begin{equation}\label{bvoiewyf83097u0380126782}
C_4:=2 \max\left\{ \frac1{E}, \left(1-\frac{E}{w_2(\bar{x}_2)}\right)^{-1} \right\}, \qquad C_{3,k}:= C_4- \frac{EC_4-((\bar{x}_2+1)c_k)^{-E(\alpha-\beta)}}{w_2(\bar{x}_2)}.
\end{equation}

We provide some properties of the functions~$w_1$ and~$w_2$ and of the constants~$C_2$, $C_{1,k}$, $C_4$ and~$C_{3,k}$ in Appendix~\ref{somboun}.

For any~$k \in \N$, we consider the quantities
\begin{equation}\label{0u9y9}
\zeta := \ln\left( \frac{\ln c_k }{\ln b_k }\right) \left(\ln\left(\frac{A}{B}\right)\right)^{-1} \qquad\mbox{and}\qquad
\xi := \ln\left( \frac{\ln c_k}{\ln b_k }\right) \left(\ln\left(\frac{D}{E}\right)\right)^{-1}.
\end{equation}
{F}rom the inequalities in~\eqref{bcuwq12345678poiuytre0987} we deduce that
\begin{equation}\label{2938de}
\zeta = \frac{128\bar{\eta}}{\eta_0}\ln\left(\frac{A}{B}\right)^{-1} \geq  \frac{32A\bar{\eta}}{B\eta_0}  \qquad
{\mbox{and}}\qquad  \xi = \frac{128\bar{\eta}}{\eta_0}\ln\left(\frac{D}{E}\right)^{-1} \geq  \frac{32D\bar{\eta}}{E\eta_0}.
\end{equation}

We also define on the interval~$ [0,1]$ the functions
\begin{equation}\label{bvcig3433}
\phi_k(x):= A\left( \frac{\ln b_k}{\ln(x(c_k-b_k)+b_k)} \right)^{\frac1{\zeta}} \qquad\mbox{and}\qquad
\psi_k(x):= D\left( \frac{\ln b_k }{\ln(x(c_k-b_k)+b_k)} \right)^{\frac1{\xi}}.
\end{equation}

We notice that~$\phi_k(0)=A$ and
\begin{eqnarray*}
&&\phi_k(1)=A\left( \frac{\ln b_k}{\ln c_k} \right)^{\frac1{\zeta}}=A\left( \frac{\ln b_k}{\ln c_k} \right)^{  \frac{\ln\left(\frac{A}{B}\right)}{\ln\left( \frac{\ln c_k }{\ln b_k }\right)}  }= A\exp\left(  \frac{\ln\left(\frac{A}{B}\right)}{\ln\left( \frac{\ln c_k }{\ln b_k }\right)} \ln\left( \frac{\ln b_k}{\ln c_k} \right)\right)
=B.
\end{eqnarray*}
Moreover, $\phi_k$ is decreasing, whence
\begin{equation}\label{werecall078567654}\phi_k \in [B,A].\end{equation}

Similarly,
we have that~$\psi_k(0)= D$, $\psi_k(1)= E$, $\psi_k$ is decreasing and so~$\psi_k \in [E,D]$. 

We stress that, exploiting the ODE result in Lemma~\ref{khf869}, applied here with~$a := b_k$, $b := c_k$, $\mu := \zeta$, and~$f_0 := A$ for~$\phi_k$, and also with~$a := b_k$, $b := c_k$, $\mu := \xi$, and~$f_0 := D$ for~$\psi_k$,
we find that, for any~$x \in [b_k,c_k]$,
\begin{equation}\label{dyugffucdt-9-}
\phi_k\left(\frac{x-b_k}{c_k-b_k}\right) = - \frac{\zeta}{c_k-b_k} \phi'_k \left(\frac{x-b_k}{c_k-b_k}\right) x \ln x.
\end{equation}
and
\begin{equation}\label{234de43232}
\psi_k\left(\frac{x-b_k}{c_k-b_k}\right) = - \frac{\xi}{c_k-b_k} \psi'_k \left(\frac{x-c_k}{c_k-b_k}\right) x \ln x.
\end{equation}

Furthermore, we point out that
\begin{equation}\label{mnbvcxz123456789poiuytrew}
\phi_k\left(\frac{x-b_k}{c_k-b_k}\right) = A \left(\frac{\ln b_k}{\ln x}\right)^{\frac1{\zeta}}= B\left(\frac{\ln c_k}{\ln x}\right)^{\frac1{\zeta}} \quad\mbox{for any } x \in [b_k,c_k] ,
\end{equation}
and
$$ 
\psi_k\left( \frac{|x|-b_k}{c_k-b_k}\right)=  D \left(\frac{\ln b_k}{\ln|x|}\right)^{\frac1{\xi}}= E\left(\frac{\ln c_k}{\ln|x|}\right)^{\frac1{\xi}}  \quad\mbox{for any } x \in [-c_k,-b_k].
$$

Finally, we define the function
\begin{equation}\label{defofu}
u_k(x):=
\begin{cases}
1- C_{1,k} x^{-\phi_k\left(\frac{x-b_k}{c_k-b_k}\right)} &\mbox{if } x \in [b_k,c_k],\\
 -1+C_{3,k}|x|^{-\psi_k\left( \frac{|x|-b_k}{c_k-b_k}\right)}&\mbox{if } x \in [-c_k,-b_k].
\end{cases}
\end{equation}
Some useful bounds on the quantities introduced in this section and the main properties of the function~$u_k$
are discussed in Appendices~\ref{somboun} and~\ref{ukprop}, respectively.
In particular, formulas~\eqref{dyugffucdt-9-} and~\eqref{234de43232}
are pivotal in establishing the monotonicity properties of~$u_k$. 

\subsection{The function~$\widetilde{u}$}\label{constru}
In this section we provide the definition of the function~$\widetilde{u}$, that will serve as a main ingredient to prove Theorem~\ref{optimal}. 
We will use the notation introduced in~Section~\ref{notas}. 

For any~$k \in \N$, we set
\begin{equation*}
\widetilde{u}_k(x):=
\begin{cases}\displaystyle
\eta\left(\frac{x-b_k}{b_k}\right)C_{1,k}\left( x^{-\phi_k\left( \frac{x-b_k}{c_k-b_k}\right)} - x^{-A} \right) + u_k(x) &\mbox{if } x \in [b_k,2b_k],\\
u_k(x)  &\mbox{if } x \in [2b_k,c_k/2],\\
\displaystyle
\eta\left(\frac{2x-c_k}{c_k}\right)C_{1,k} \left( x^{-B}-x^{-\phi_k\left( \frac{x-b_k}{c_k-b_k}\right)}\right) +1 -C_{1,k} x^{-B} &\mbox{if } x \in [c_k/2, c_k],\\
\displaystyle
1- C_2x^{-B} + (C_2-C_{1,k})\widetilde{\eta}\left(\frac{x-c_k}{c_k}\right) x^{-B} &\mbox{if } x \in [c_k, d_k],\\
\displaystyle 1- C_{1,k+1} x^{-A}-   \eta \left( \frac{x-d_k}{d_k} \right) \left(C_2 x^{-B}- C_{1,k+1}x^{-A}\right) &\mbox{if } x \in  [d_k, a_{k+1}].
\end{cases}
\end{equation*}
We stress that we have made use of the inequalities in~\eqref{bvrikebgifuegwouw5636598jbvkdskj2}.

The function~$\widetilde{u}_k$ transitions from~$1 -  C_{1,k} x^{-A}$ to~$1 - C_{1,k}x^{-B}$ over the interval~$[b_k, c_k]$ via the intermediate function~$u_k$, given in~\eqref{defofu}. It then proceeds to join~$1 - C_{1,k}x^{-B}$ with~$1 - C_2x^{-B}$ on~$[c_k, d_k]$, and finally interpolates between~$1 - C_2x^{-B}$ and~$1 -  C_{1,k+1}x^{-A}$ on~$[d_k, a_{k+1}]$. All these transitions are smooth thanks to the use of the cutoff functions~$\eta$ and~$\widetilde{\eta}$. In particular, recalling the definition of~$\widetilde{\eta}$ in~\eqref{defetatilde09876}, we obtain that the resulting function is of class~$C^{3,1}$.

In the same spirit, we define~$\widetilde{u}_k$ for negative values of~$x$ as follows:
\begin{equation*}
\widetilde{u}_k(x):=
\begin{cases}\displaystyle
-1+C_{3,k+1} |x|^{-D}-\eta \left( \frac{|x|-d_k}{d_k} \right) \left(C_{3,k+1}|x|^{-D}-C_4 |x|^{-E}\right) &\mbox{if } x \in  [-a_{k+1},-d_k ],\\
\displaystyle
-1+C_4 |x|^{-E} - (C_4-C_{3,k})\widetilde{\eta}\left( \frac{|x|-c_k}{c_k}\right) |x|^{-E}&\mbox{if } x \in [-d_k, -c_k],\\ 
\displaystyle
\eta\left(\frac{2|x|-c_k}{c_k}\right) C_{3,k}\left( |x|^{-\psi_k\left( \frac{|x|-b_k}{c_k-b_k}\right)}-|x|^{-E}\right) -1 + C_{3,k} |x|^{-E} &\mbox{if } x \in [-c_k, -c_k/2],\\
u_k(x)  &\mbox{if } x \in [-c_k/2,-2b_k],\\
\displaystyle\eta\left(\frac{|x|-b_k}{b_k}\right)C_{3,k}\left(  |x|^{-D} - |x|^{-\psi_k\left( \frac{|x|-b_k}{c_k-b_k}\right)}\right) + u_k(x) &\mbox{if } x \in [-2b_k,-b_k].
\end{cases}
\end{equation*}
Hence, $\widetilde{u}_k$ goes from~$1 -C_{3,k+1} |x|^{-D}$ to~$1 - C_{4}|x|^{-E}$ on the interval~$[-a_{k+1},-d_k]$, next it connects~$1 - C_{4}|x|^{-E}$ with~$1 - C_{3,k}|x|^{-E}$ across~$[-d_k,-c_k]$. Finally, it brings the function back from~$1 - C_{3,k} |x|^{-E}$ to~$1 - C_{3,k} |x|^{-D}$ over~$[-c_k,-b_k]$, using~$u_k$ to handle the transition. Moreover, also these transitions are~$C^{3,1}$ thanks to the cutoff functions~$\eta$ and~$\widetilde{\eta}$.

Having introduced all the required notation, we now present the final expression of our function~$\widetilde{u}: \R \to [-1,1]$. Namely, 
\begin{equation}\label{defur382tgjekf6pto37654}
{\mbox{$\widetilde{u}$ is a~$C^{3,1} (\R)$ function such that~$\widetilde{u}'$ is strictly positive in~$(-a_0,a_0)$}}\end{equation} and, for all~$k\in\N$,
\begin{equation}\label{riga927}
\widetilde{u}(x)=
\begin{cases}
\widetilde{u}_k(x) &\mbox{if } x \in [-a_{k+1},-b_k], \\
-1+C_{3,k}|x|^{-D} &\mbox{if } x \in [-b_k,-a_k],\\
1- C_{1,k}x^{-A} &\mbox{if } x \in [a_k,b_k],\\
\widetilde{u}_k(x)  &\mbox{if } x \in [b_k, a_{k+1}].
\end{cases}
\end{equation}

\subsection{Properties of~$\widetilde{u}$}\label{n5m7b38}

This section studies the main properties of the function~$\widetilde{u}$ introduced in formula~\eqref{riga927}.
More precisely, Propositions~\ref{6tg4} and~\ref{03uhe3} below establish bounds for~$\widetilde{u}$ and its first derivatives and  Proposition~\ref{4o0p9z} provides upper and lower estimates for~$L_K \widetilde{u}'$ and~$L_K \widetilde{u}''$  for large values of~$x$.


\begin{prop}\label{6tg4}
There exist two positive constants~$\widetilde{C}$ and~$\widehat{C}$, depending on~$A$ and~$B$, such that, for any~$x \in [b_k,c_k]$, we have
\begin{enumerate}[label=(\roman*)]
  \setcounter{enumi}{0} 
\item $\widetilde{C} x^{-\phi_k\left( \frac{x-b_k}{c_k-b_k} \right)}\leq 1- \widetilde{u}(x) \leq  \widehat{C}  x^{-\phi_k\left( \frac{x-b_k}{c_k-b_k} \right)}$,
\item $\widetilde{u}'(x) \leq \widehat{C} \min \left\{ x^{-1-A(\delta-\gamma+1)}, x^{-1-2s+(\gamma-2)\phi_k\left( \frac{x-b_k}{c_k-b_k}\right)} \right\}$, 
\item$\widetilde{u}'(x) \geq \widetilde{C} \max \left\{ x^{-1-B(\gamma-\delta+1)}, \,x^{-1-2s+(\delta-2)\phi_k\left( \frac{x-b_k}{c_k-b_k}\right)} \right\}$,
\end{enumerate}
for any~$x \in [c_k,d_k]$, we have
\begin{enumerate}[label=(\roman*)]
  \setcounter{enumi}{3} 
\item $\widetilde{C} x^{-B} \leq 1- \widetilde{u}(x) \leq  \widehat{C} x^{-B}$,
\item$ \widetilde{u}'(x) \geq \widetilde{C} x^{-1-B(\gamma-\delta+1)}$,
\end{enumerate}
and, for any~$x \in [d_k,a_{k+1}]$, we have
\begin{enumerate}[label=(\roman*)]
  \setcounter{enumi}{5} 
\item $\widetilde{C} x^{ -A-1} \leq \widetilde{u}'(x) \leq \widehat{C} x^{-B-1}$.
\end{enumerate}
\end{prop}

\begin{prop}\label{03uhe3}
There exist two positive constants~$\widetilde{C}$ and~$\widehat{C}$, depending at most from~$D$ and~$E$, such that, for any~$x \in [-c_k,-b_k]$, we have
\begin{enumerate}[label=(\roman*)]
  \setcounter{enumi}{0} 
\item $ \widetilde{C} |x|^{-\psi_k\left( \frac{|x|-b_k}{c_k-b_k} \right)}\leq  1+ \widetilde{u}(x) \leq\widehat{C}|x|^{-\psi_k\left( \frac{|x|-b_k}{c_k-b_k} \right)}$ ,
\item $\widetilde{u}'(x) \leq \widehat{C} \min \left\{ |x|^{-1-D(\beta-\alpha+1)}, |x|^{-1-2s+(\alpha-2)\psi_k\left( \frac{|x|-b_k}{c_k-b_k}\right)} \right\}$,
\item$\widetilde{u}'(x) \geq \widetilde{C} \max \left\{ |x|^{-1-E(\alpha-\beta+1)}, |x|^{-1-2s+(\beta-2)\psi_k\left( \frac{|x|-b_k}{c_k-b_k}\right)} \right\}$,
\end{enumerate}
for any~$x \in [-d_k,-c_k]$, we have
\begin{enumerate}[label=(\roman*)]
  \setcounter{enumi}{3} 
\item $\widehat{C}|x|^{-E} \leq 1+ \widetilde{u}(x) \leq \widetilde{C} |x|^{-E}$,
\item $ \widetilde{u}'(x) \geq \widetilde{C} |x|^{-1-E(\alpha-\beta+1)}$,
\end{enumerate}
and, for any~$x \in [-a_{k+1},-d_k]$, we have
\begin{enumerate}[label=(\roman*)]
  \setcounter{enumi}{5} 
\item $\widetilde{C} |x|^{ -D-1} \leq \widetilde{u}'(x) \leq \widehat{C} |x|^{-E-1}$.
\end{enumerate}
\end{prop}

For the sake of completeness, in Corollary~\ref{strcap} we specify the exact form of the~$\min$ and~$\max$ functions appearing in~\textit{(ii)} and~\textit{(iii)} of Propositions~\ref{6tg4} and~\ref{03uhe3}, according to the value of~$x$.

We now present the proofs of Propositions~\ref{6tg4} and~\ref{03uhe3}.
We will provide full details for Proposition~\ref{6tg4} and the necessary modifications for Proposition~\ref{03uhe3}.
Also, being quite long, the proof of Proposition~\ref{6tg4} is divided into six separate subproofs.

We recall that, throughout the rest of this section, the setting introduced in Section~\ref{notas} is used.

\begin{proof}[Proof of point~\textit{(i)} of Proposition~\ref{6tg4}]
In the interval~$[2b_k,c_k/2]$ it holds that
\[ 1 - \widetilde{u}(x) = 1 - \widetilde{u}_k(x) =
1-u_k(x)= C_{1,k} x^{-\phi_k \left( \frac{x-b_k}{c_k-b_k} \right)}.\]
Also, by Lemma~\ref{drao} we know that~$C_{1,k}\in (2,C_2)$, from
which the desired estimate in~\textit{(i)} follows.

Accordingly,
from now on we only consider the cases~$x \in [b_k,2b_k]$ and~$x \in [c_k/2,c_k]$.

We claim that
\begin{equation}\label{n8bv7v6}
\mbox{the estimates in~\textit{(i)} hold true for any} \ x \in [b_k,2b_k].
\end{equation}
To show this, 
we recall the definition of~$\widetilde{u}$ and we find that
\begin{equation}\label{viuytrewjkewhf8765}\begin{split}&
\widetilde{u}(x)=\widetilde{u}_k(x) =
\eta\left(\frac{x-b_k}{b_k}\right)C_{1,k}\left( x^{-\phi_k\left( \frac{x-b_k}{c_k-b_k}\right)} - x^{-A} \right) + u_k(x) \\&\qquad=
\eta\left(\frac{x-b_k}{b_k}\right)C_{1,k}\left( x^{-\phi_k\left( \frac{x-b_k}{c_k-b_k}\right)} - x^{-A} \right) +
1- C_{1,k}x^{-\phi_k\left(\frac{x-b_k}{c_k-b_k}\right)}.\end{split}
\end{equation}
We now use~\eqref{werecall078567654} to see that~$ x^{-\phi_k\left( \frac{x-b_k}{c_k-b_k}\right)}\ge x^{-A}$ for all~$x\in[b_k,2b_k]$.
Moreover, we know that~$C_{1,k}\in (2,C_2)$, thanks to Lemma~\ref{drao}. In this way, we gather from~\eqref{viuytrewjkewhf8765} that
\begin{equation}\label{ytfjksabgkfe458732trgfout3wqiuger}
\widetilde{u}(x) 
\ge  1- C_{1,k}x^{-\phi_k\left(\frac{x-b_k}{c_k-b_k}\right)}
\geq 1- C_2x^{-\phi_k\left(\frac{x-b_k}{c_k-b_k}\right)}.
\end{equation} 

Furthermore,
\begin{eqnarray*}
\widetilde{u}(x)\leq C_{1,k}\left(
x^{-\phi_k\left( \frac{x-b_k}{c_k-b_k}\right)} - x^{-A}\right) +
1-C_{1,k} x^{-\phi_k\left(\frac{x-b_k}{c_k-b_k}\right)}=1-C_{1,k}x^{-A} \leq1 -2 x^{-A} .
\end{eqnarray*}
Thus, exploiting~\eqref{be983f} and the fact that~$\zeta \geq 2A$,
\begin{equation*}
1-\widetilde{u}(x)\geq 2 x^{-\phi_k\left( \frac{x-b_k}{c_k-b_k}\right)} x^{\phi_k\left( \frac{x-b_k}{c_k-b_k}\right)-A} \geq 2  x^{-\phi_k\left( \frac{x-b_k}{c_k-b_k}\right)}  e^{-\frac{2A\ln2}{\zeta}}\geq   x^{-\phi_k\left( \frac{x-b_k}{c_k-b_k}\right)}.
\end{equation*}
From this and~\eqref{ytfjksabgkfe458732trgfout3wqiuger} we obtain~\eqref{n8bv7v6}.

Similarly, we check that
\begin{equation}\label{3v4b5n}
\mbox{the estimates in~\textit{(i)} hold true for any} \ x \in [c_k/2,c_k].
\end{equation}
Indeed, by the definition of~$\widetilde{u}$,
\begin{equation}\label{56789vdsf098765}
\widetilde{u} (x) =\widetilde{u}_k(x)
=\eta\left(\frac{2x-c_k}{c_k}\right) C_{1,k} \left( x^{-B}-x^{-\phi_k\left( \frac{x-b_k}{c_k-b_k}\right)}\right) +1 -C_{1,k} x^{-B}.\end{equation}
Since~$\phi_k\geq B$ (recall~\eqref{werecall078567654}), we infer that~$ x^{-B}\geq x^{-\phi_k\left( \frac{x-b_k}{c_k-b_k}\right)}$ for all~$x\in[c_k/2,c_k]$.
Therefore, by the fact that~$C_{1,k}\geq 2$ (see Lemma~\ref{drao}) we deduce that
\begin{eqnarray*}
\widetilde{u} (x)\leq C_{1,k}\left( x^{-B}-x^{-\phi_k\left( \frac{x-b_k}{c_k-b_k}\right)}\right)+1 -C_{1,k} x^{-B}=1- C_{1,k}x^{-\phi_k\left( \frac{x-b_k}{c_k-b_k}\right)} \leq 1- 2x^{-\phi_k\left( \frac{x-b_k}{c_k-b_k}\right)}.
\end{eqnarray*}

Moreover, form~\eqref{56789vdsf098765} we also infer that
\begin{eqnarray*} \widetilde{u} (x)\geq1 -C_{1,k} x^{-B}.
\end{eqnarray*}
As a result,
by~\eqref{wb5e983f}, and recalling that~$C_{1,k}\leq C_2$ (see Lemma~\ref{drao}) and~$\zeta \geq 2B$,
\begin{equation*}
1-\widetilde{u}(x) \leq C_{1,k} x^{-\phi_k\left( \frac{x-b_k}{c_k-b_k}\right)}x^{\phi_k\left( \frac{x-b_k}{c_k-b_k}\right)-B}\leq C_2 x^{-\phi_k\left( \frac{x-b_k}{c_k-b_k}\right)} e^{\frac{2B\ln2}{\zeta}}\leq 2 C_2 x^{-\phi_k\left( \frac{x-b_k}{c_k-b_k}\right)}.
\end{equation*}
{F}rom these considerations we conclude that~\eqref{3v4b5n} holds.
\end{proof}

\begin{proof}[Proof of point~\textit{(ii)} of Proposition~\ref{6tg4}]
As~$\widetilde{u} = u_k$ in~$[2b_k, c_k/2]$, in this range
the estimate in~\textit{(ii)}  is a consequence of~\eqref{0gbvctfdshgdcxui}.

If~$ x \in[b_k,2b_k]$, by the definition of~$\widetilde{u}$ and recalling the equation for~$\phi_k$ in~\eqref{dyugffucdt-9-}, we have that
\begin{equation}\label{cn8439yth3efvb-743892}\begin{split}&
\widetilde{u}'(x)=\widetilde{u}_k'(x)=
C_{1,k}\frac{\eta'\left(\frac{x-b_k}{b_k}\right)}{b_k}\left( x^{-\phi_k\left( \frac{x-b_k}{c_k-b_k}\right)} - x^{-A} \right) +C_{1,k} \eta\left(\frac{x-b_k}{b_k}\right) A x^{-1-A}
\\&\quad-C_{1,k}\eta\left(\frac{x-b_k}{b_k}\right) 
x^{-\phi_k\left( \frac{x-b_k}{c_k-b_k}\right)}\left(
\frac{\phi'_k\left( \frac{x-b_k}{c_k-b_k}\right)\ln x}{c_k-b_k}
+\frac{\phi_k\left( \frac{x-b_k}{c_k-b_k}\right)}{x}
\right)  
+ u_k'(x)\\&=
C_{1,k}\frac{\eta'\left(\frac{x-b_k}{b_k}\right)}{b_k}\left( x^{-\phi_k\left( \frac{x-b_k}{c_k-b_k}\right)} - x^{-A} \right) +C_{1,k}\eta\left(\frac{x-b_k}{b_k}\right) A x^{-1-A} 
\\&\quad+C_{1,k}\eta\left(\frac{x-b_k}{b_k}\right) 
x^{-\phi_k\left( \frac{x-b_k}{c_k-b_k}\right)}\phi_k\left( \frac{x-b_k}{c_k-b_k}\right)\frac{1-\zeta}{\zeta x} 
+ u_k'(x).\end{split}
\end{equation} 
We now recall that~$\phi_k\le A$, thanks to~\eqref{werecall078567654}, that~$\eta'<0$ and that~$\zeta>1$ 
and we conclude that
$$ \widetilde{u}'(x)\leq C_{1,k} A x^{-1-A} + u_k'(x).$$
Thus, using~\eqref{vio7t} and~\eqref{be983f},
\begin{equation}\label{0987654qwertyusdfghjkl}\begin{split}& \widetilde{u}'(x)\leq C_{1,k} A x^{-1-A} + 2C_2x^{-1-\phi_k\left(\frac{x-b_k}{c_k-b_k}\right)}
\leq 2 C_2 x^{-1-A}\left(1+x^{A-\phi_k\left( \frac{x-b_k}{c_k-b_k}\right)}  \right) 
\\&\qquad  \leq 2C_2 x^{-1-A}\left(1+e^{\frac{2A\ln2}{\zeta}}\right)\leq 6C_2 x^{-1-A}
.\end{split}\end{equation}

We observe that, by~\eqref{syhvyuf4},
$$ \max\left\{ A (\delta-\gamma+1), \, 2s - (\gamma-2) \phi_k(x) \right\}\leq B\leq A,$$
which, together with~\eqref{0987654qwertyusdfghjkl}, gives the desired estimate in~\textit{(ii)} for~$ x \in[b_k,2b_k]$.

Hence, we are left to show that point~\textit{(ii)} holds true for any~$x \in[c_k/2,c_k]$. In this case,
recalling the equation for~$\phi_k$ in~\eqref{dyugffucdt-9-},
\begin{equation}\label{cpopia7573ihbdh-876543}\begin{split}&
\widetilde{u}'(x)=\widetilde{u}_k'(x)=
\frac{2C_{1,k}\eta'\left(\frac{2x-c_k}{c_k}\right)}{c_k}\left( x^{-B}- x^{-\phi_k\left( \frac{x-b_k}{c_k-b_k}\right)}  \right) -\eta\left(\frac{2x-c_k}{c_k}\right)C_{1,k}B x^{-1-B} 
\\&\quad
+\eta\left(\frac{2x-c_k}{c_k}\right)C_{1,k}x^{-\phi_k\left( \frac{x-b_k}{c_k-b_k}\right)}
\left(
\frac{\phi'_k\left( \frac{x-b_k}{c_k-b_k}\right)\ln x}{c_k-b_k}
+\frac{\phi_k\left( \frac{x-b_k}{c_k-b_k}\right)}{x}
\right)
+C_{1,k}B x^{-1-B}\\&=\frac{2C_{1,k}\eta'\left(\frac{2x-c_k}{c_k}\right)}{c_k}\left( x^{-B}- x^{-\phi_k\left( \frac{x-b_k}{c_k-b_k}\right)}  \right) -\eta\left(\frac{2x-c_k}{c_k}\right)C_{1,k}B x^{-1-B} 
\\&\quad+\eta\left(\frac{2x-c_k}{c_k}\right) C_{1,k}
x^{-\phi_k\left( \frac{x-b_k}{c_k-b_k}\right)}\phi_k\left( \frac{x-b_k}{c_k-b_k}\right)\frac{\zeta-1}{\zeta x} 
+C_{1,k}B x^{-1-B}.
\end{split}\end{equation}
Therefore, since~$B\leq \phi_k \leq A$ (see~\eqref{werecall078567654}) and~$C_2\geq C_{1,k}\geq2$ (see Lemma~\ref{drao}),
$$ \widetilde{u}'(x)\leq C_2Ax^{-1-\phi_k\left( \frac{x-b_k}{c_k-b_k}\right)}
+C_{1,k}B x^{-1-B}\leq 4C_{2} x^{-1-B}.$$
This and~\eqref{syhvyuf4} provides the desired estimate in~\textit{(ii)} for any~$x \in[c_k/2,c_k]$ and completes the proof.\end{proof}

\begin{proof}[Proof of point~\textit{(iii)} of Proposition~\ref{6tg4}]
In~$[2b_k,c_k/2]$ it holds that~$\widetilde{u}=\widetilde{u}_k= u_k$, and therefore the estimate in~\textit{(iii)} follows from~\eqref{04yf7y5}.
Hence it remains to prove~\textit{(iii)} in the intervals~$[b_k, 2b_k]$ and~$[c_k/2, c_k]$. 

We observe that, exploiting the fact that~$\zeta\geq 2A\geq 2B$,
\begin{equation}\label{3;k'}
e^{\frac{2A\ln2}{\zeta}} -1 
=\int_{0}^{\frac{2A\ln2}{\zeta}} e^t\,dt
\leq e^{\frac{2A\ln2}{\zeta}}\frac{2A\ln2}{\zeta}
\leq e^{\ln2}\frac{2A\ln2}{\zeta}
 \leq \frac{4A}{\zeta}
\end{equation}
and similarly
\begin{equation}\label{n''e}
e^{\frac{2B\ln2}{\zeta}} -1  \leq \frac{4B}{\zeta}.
\end{equation}

Now, we check that
\begin{equation}\label{3eghs5}
\mbox{the estimate in~\textit{(iii)} holds true for any } x \in \left[b_k,\frac{3b_k}2\right].
\end{equation}
With this aim, we stress that~$\eta' \leq 0$ implies that
\begin{equation}\label{recalling9769kdjs09865}
\eta\left(\frac{x-b_k}{b_k}\right)\geq \eta\left(\frac12\right)\geq \eta_0 \quad\mbox{for any} \ x \in \left[b_k,\frac{3b_k}2\right],
\end{equation}
where~$\eta_0$ is given in~\eqref{notations98}.

We also recall~\eqref{cn8439yth3efvb-743892}, according to which
\begin{equation}\label{bvncxk932759rfhoiewaut90436-}\begin{split}
\widetilde{u}'(x)&=
\frac{\eta'\left(\frac{x-b_k}{b_k}\right)}{b_k}C_{1,k}\left( x^{-\phi_k\left( \frac{x-b_k}{c_k-b_k}\right)} - x^{-A} \right)  +\eta\left(\frac{x-b_k}{b_k}\right) C_{1,k}A x^{-1-A} 
\\&\qquad+\eta\left(\frac{x-b_k}{b_k}\right) C_{1,k}
x^{-\phi_k\left( \frac{x-b_k}{c_k-b_k}\right)}\phi_k\left( \frac{x-b_k}{c_k-b_k}\right)\frac{1-\zeta}{\zeta x}
+ u_k'(x).\end{split}
\end{equation}
Now, by~\eqref{cs542652},
\begin{equation}\label{45738fvxsperamf4u3876y980t-87654}
\eta\left(\frac{x-b_k}{b_k}\right) C_{1,k}
x^{-\phi_k\left( \frac{x-b_k}{c_k-b_k}\right)}\phi_k\left( \frac{x-b_k}{c_k-b_k}\right)\frac{1-\zeta}{\zeta x}
+ u_k'(x)=u'_k(x)
\left(
1-\eta\left(\frac{x-b_k}{b_k}\right) \right)\geq0.
\end{equation}
Using this information into~\eqref{bvncxk932759rfhoiewaut90436-},
and recalling~\eqref{recalling9769kdjs09865} and the definition of~$\bar\eta$ in~\eqref{notations98}, we conclude that
\begin{eqnarray*}
C^{-1}_{1,k}\,\widetilde{u}'(x) &\geq& \eta_0 A x^{-1-A} - \frac{\bar{\eta}x^{-A}}{b_k}  \left(x^{A-\phi_k\left( \frac{x-b_k}{c_k-b_k}\right)}-1 \right)
.\end{eqnarray*}
Hence, exploiting~\eqref{be983f} and~\eqref{3;k'}, we gather
\begin{equation*}
C^{-1}_{1,k}\,\widetilde{u}'(x)
\geq \eta_0 Ax^{-1-A} - \frac{3\bar{\eta}x^{-1-A}}{2}  \left(e^{\frac{2A\ln2}{\zeta}} -1 \right) \geq
A \left(\eta_0-\frac{6\bar{\eta}}{\zeta}\right)x^{-1-A}.
\end{equation*}
Since~$\zeta \geq (16 \bar{\eta})/\eta_0$ (recall~\eqref{2938de}) and~$C_{1,k} \geq 2$ (see Lemma~\ref{drao}), we find that 
\begin{equation}\label{umiss}
\widetilde{u}'(x) \geq  A\eta_0 x^{-1-A},
\end{equation}
and~\eqref{3eghs5} follows from this and~\eqref{9dggfye}.

Now we prove that
\begin{equation}\label{eg5}
\mbox{the estimate in~\textit{(iii)} holds true for any } x \in \left[\frac{3b_k}2,2b_k\right].
\end{equation}
To do this, we point out that
\begin{equation*}
1- \eta\left(\frac{x-b_k}{b_k}\right)\geq 1-\eta\left(\frac12\right)\geq \eta_0 \quad\mbox{for any} \ x \in \left[\frac{3b_k}2,2b_k\right].
\end{equation*}
As a result, we use~\eqref{bvncxk932759rfhoiewaut90436-} and~\eqref{45738fvxsperamf4u3876y980t-87654} to find that
\begin{equation*}
\widetilde{u}'(x)\geq
\eta_0 u_k'(x)- \frac{\bar{\eta}x^{-A}}{b_k} C_{1,k} \left(x^{A-\phi_k\left( \frac{x-b_k}{c_k-b_k}\right)}-1 \right).
\end{equation*}
Thus, thanks to~\eqref{cs542652}, \eqref{be983f} and~\eqref{3;k'},
\begin{equation}\label{umis}
\widetilde{u}'(x)\geq \eta_0 u'_k(x) - \frac{\bar{\eta}x^{-A}}{b_k} C_{1,k} \left(e^{\frac{2A\ln2}{\zeta}} -1 \right)
\geq C_{1,k} \left(\frac{B\eta_0}2 -\frac{8A\bar{\eta}}{\zeta}\right)x^{-1-A}.
\end{equation}
Using~\eqref{2938de},~\eqref{9dggfye} and the fact that~$C_{1,k}\geq 2$ (see Lemma~\ref{drao}) we thereby complete the proof of~\eqref{eg5}.

We now check that
\begin{equation}\label{09yrtr}
\mbox{the estimate in~\textit{(iii)} holds true for any } x \in \left[\frac{c_k}{2}, \, \frac{3c_k}4\right].
\end{equation}
To do this, we stress that
\begin{equation}\label{mnbvcpoiuyt765}
\eta\left(\frac{2x-c_k}{c_k}\right)\geq \eta\left(\frac12\right)\geq \eta_0 \quad\mbox{for any} \ x \in \left[\frac{c_k}2, \frac{3c_k}4\right].
\end{equation}
Also, from~\eqref{cpopia7573ihbdh-876543} 
we know that
\begin{equation}\label{cbjy832y523jdshk}\begin{split}
\widetilde{u}'(x)&=\frac{2C_{1,k}\eta'\left(\frac{2x-c_k}{c_k}\right)}{c_k}\left( x^{-B}- x^{-\phi_k\left( \frac{x-b_k}{c_k-b_k}\right)}  \right)
+\left(1-\eta\left(\frac{2x-c_k}{c_k}\right)\right)C_{1,k}B x^{-1-B}
\\&\qquad+\eta\left(\frac{2x-c_k}{c_k}\right) C_{1,k}
x^{-\phi_k\left( \frac{x-b_k}{c_k-b_k}\right)}\phi_k\left( \frac{x-b_k}{c_k-b_k}\right)\frac{\zeta-1}{\zeta x} 
\\
&\geq
\frac{2C_{1,k}\eta'\left(\frac{2x-c_k}{c_k}\right)}{c_k}\left( x^{-B}- x^{-\phi_k\left( \frac{x-b_k}{c_k-b_k}\right)}  \right)
\\&\qquad
+\eta\left(\frac{2x-c_k}{c_k}\right) C_{1,k}
x^{-1-\phi_k\left( \frac{x-b_k}{c_k-b_k}\right)}\phi_k\left( \frac{x-b_k}{c_k-b_k}\right)\frac{\zeta-1}{\zeta} \\&\geq
\frac{2C_{1,k}\eta'\left(\frac{2x-c_k}{c_k}\right)}{c_k}\left( x^{-B}- x^{-\phi_k\left( \frac{x-b_k}{c_k-b_k}\right)}  \right)
\\&\qquad
+\eta_0 B C_{1,k}
x^{-1-\phi_k\left( \frac{x-b_k}{c_k-b_k}\right)}\frac{\zeta-1}{\zeta}
.
\end{split}\end{equation}

We now observe that, by~\eqref{2938de}, \eqref{werecall078567654} and~\eqref{mnbvcpoiuyt765},
\begin{equation}\label{vnjsdlty84tyuhfgdsjk}
\eta\left(\frac{2x-c_k}{c_k}\right)
\phi_k\left( \frac{x-b_k}{c_k-b_k}\right)\frac{\zeta-1}{\zeta}
\ge \eta_0 B\left(1-\frac{\eta_0B }{32 A\bar\eta}\right)\ge\frac{\eta_0B}2.
\end{equation}

Moreover, exploiting~\eqref{wb5e983f} and~\eqref{n''e}, we find that
\begin{equation}\label{cdjsty843fhejkawhfalie}\begin{split}
& x^{-B} -x^{-\phi_k\left( \frac{x-b_k}{c_k-b_k}\right)}
=\left(
x^{\phi_k\left( \frac{x-b_k}{c_k-b_k}\right)-B}- 1 \right)x^{-\phi_k\left( \frac{x-b_k}{c_k-b_k}\right)}
\\&\qquad \le
\left( e^{\frac{2B\ln2}{\zeta}} -1 \right)x^{-\phi_k\left( \frac{x-b_k}{c_k-b_k}\right)}
\le \frac{4B}{\zeta}x^{-\phi_k\left( \frac{x-b_k}{c_k-b_k}\right)}.
\end{split}\end{equation}
Since~$\eta'\le0$, plugging this information
and~\eqref{vnjsdlty84tyuhfgdsjk} into~\eqref{cbjy832y523jdshk}
we obtain that
\begin{equation*}\begin{split}
\widetilde{u}'(x) &\geq \frac{\eta_0B}2C_{1,k}
x^{-1-\phi_k\left( \frac{x-b_k}{c_k-b_k}\right)}
- \frac{8B\bar{\eta}}{\zeta c_k}C_{1,k}x^{-\phi_k\left( \frac{x-b_k}{c_k-b_k}\right)}
\\&\geq
C_{1,k}\frac{\eta_0 B}2 x^{-1-\phi_k\left( \frac{x-b_k}{c_k-b_k}\right)} - \frac{6B\bar{\eta}}{\zeta}C_{1,k}x^{-1-\phi_k\left( \frac{x-b_k}{c_k-b_k}\right)}\\
& = C_{1,k} B \left(\frac{\eta_0 }2- \frac{6\bar{\eta}}{\zeta}\right)
x^{-1-\phi_k\left( \frac{x-b_k}{c_k-b_k}\right)}
.\end{split}\end{equation*}
Now, we recall that~$\zeta \geq (32 \bar{\eta})/\eta_0$ (see~\eqref{2938de}) and~$C_{1,k}\geq 2$ (see Lemma~\ref{drao}) and we conclude that
$$ \widetilde{u}'(x)\ge 
\frac{C_{1,k} B\eta_0}4 x^{-1-\phi_k\left( \frac{x-b_k}{c_k-b_k}\right)}
\ge \frac{ B\eta_0}2 x^{-1-\phi_k\left( \frac{x-b_k}{c_k-b_k}\right)}.$$
This, together with~\eqref{werecall078567654} and~\eqref{9dggfye},
entails~\eqref{09yrtr}.

In a similar manner, we establish that
\begin{equation}\label{0pppqf4}
\mbox{the estimate in~\textit{(iii)} holds true for any } x \in \left[\frac{3c_k}4, \,c_k\right].
\end{equation}
Indeed, from~\eqref{cpopia7573ihbdh-876543} we know that
\begin{equation*}\begin{split}
\widetilde{u}'(x)&=\frac{2C_{1,k}\eta'\left(\frac{2x-c_k}{c_k}\right)}{c_k}\left( x^{-B}- x^{-\phi_k\left( \frac{x-b_k}{c_k-b_k}\right)}  \right)
+\left(1-\eta\left(\frac{2x-c_k}{c_k}\right)\right)C_{1,k}B x^{-1-B}
\\&\qquad+\eta\left(\frac{2x-c_k}{c_k}\right) C_{1,k}
x^{-\phi_k\left( \frac{x-b_k}{c_k-b_k}\right)}\phi_k\left( \frac{x-b_k}{c_k-b_k}\right)\frac{\zeta-1}{\zeta x} \\
&\geq \frac{2C_{1,k}\eta'\left(\frac{2x-c_k}{c_k}\right)}{c_k}\left( x^{-B}- x^{-\phi_k\left( \frac{x-b_k}{c_k-b_k}\right)}  \right)
+\left(1-\eta\left(\frac{2x-c_k}{c_k}\right)\right)C_{1,k}B x^{-1-B}\\
&\geq \frac{2C_{1,k}\eta'\left(\frac{2x-c_k}{c_k}\right)}{c_k}\left( x^{-B}- x^{-\phi_k\left( \frac{x-b_k}{c_k-b_k}\right)}  \right)
+\eta_0 C_{1,k}B x^{-1-B}.
\end{split}\end{equation*}
We exploit the computation in~\eqref{cdjsty843fhejkawhfalie} and find that
\begin{equation*}\begin{split}
\widetilde{u}'(x)&\geq
\eta_0C_{1,k} B x^{-1-B}
- \frac{8B\bar{\eta}}{\zeta c_k}C_{1,k}x^{-\phi_k\left( \frac{x-b_k}{c_k-b_k}\right)}\\&\geq
\eta_0 C_{1,k}B x^{-1-B} - \frac{8B\bar{\eta}}{\zeta }C_{1,k}x^{-1-\phi_k\left( \frac{x-b_k}{c_k-b_k}\right)}\\
&\geq
C_{1,k}B \left(\eta_0 - \frac{8\bar{\eta}}{\zeta }\right)x^{-1-\phi_k\left( \frac{x-b_k}{c_k-b_k}\right)}\\
&\geq
\frac{C_{1,k}B}4 x^{-1-\phi_k\left( \frac{x-b_k}{c_k-b_k}\right)}\\
&\geq
\frac{B}2 x^{-1-\phi_k\left( \frac{x-b_k}{c_k-b_k}\right)}.
\end{split}\end{equation*}
{F}rom this, \eqref{werecall078567654} and~\eqref{9dggfye} we infer~\eqref{0pppqf4} and complete the proof.
\end{proof}

\begin{proof}[Proof of point~\textit{(iv)} of Proposition~\ref{6tg4}]
It follows by construction that in~$[c_k,d_k]$
\begin{equation*}
1-C_2x^{-B} \leq \widetilde{u}(x) \leq 1 - C_{1,k}x^{-B}.
\end{equation*}
Thus, the fact that~$C_{1,k}\geq 2$ (see Lemma~\ref{drao}) gives the desired result.\end{proof}

\begin{proof}[Proof of point~\textit{(v)} of Proposition~\ref{6tg4}] \label{x3}
{W}e recall the definition of~$w_1$ in~\eqref{funcome} and compute
\begin{equation}\label{pjwh8888}
\begin{split}
\widetilde{u}'(x)&= B C_2x^{-B-1} + \frac{(C_2-C_{1,k})\widetilde{\eta}'\left(\frac{x-c_k}{c_k}\right)}{c_k}x^{-B} - B (C_2-C_{1,k})\widetilde{\eta}\left(\frac{x-c_k}{c_k}\right)  x^{-B-1}\\&=
B C_2x^{-B-1} - (C_2-C_{1,k}) x^{-B-1}\left(B \widetilde{\eta}\left(\frac{x-c_k}{c_k}\right)  -\frac{\widetilde{\eta}'\left(\frac{x-c_k}{c_k}\right) x}{c_k} \right)\\&=
 x^{-B-1} \left( BC_2-(C_2-C_{1,k}) w_1\left(\frac{x-c_k}{c_k}\right) \right).
\end{split}
\end{equation}

We observe that
\begin{equation*}
0\le \frac{x-c_k}{c_k}\le \frac{d_k-c_k}{c_k}=1,
\end{equation*} and therefore, by~\eqref{jeoft1} in Lemma~\ref{i03y3},
\begin{equation}\label{thisnad05382}
w_1\left(\frac{x-c_k}{c_k}\right)\le w(\bar{x}_1).
\end{equation}

In addition, by the definition of~$C_{1,k}$ in~\eqref{bvoiewyf83097u038012678} and recalling Lemma~\ref{drao}, we see that
\begin{equation*}
0<C_2-C_{1,k}=  \frac{BC_2-((\bar{x}_1+1)c_k)^{-B(\gamma-\delta)}}{w_1(\bar{x}_1)}.
\end{equation*}
Using this and~\eqref{thisnad05382}, we conclude that
\begin{equation*}
\begin{split}
BC_2-(C_2-C_{1,k}) w_1\left(\frac{x-c_k}{c_k}\right) &=BC_2- \left(BC_2-((\bar{x}_1+1)c_k )^{-B(\gamma-\delta)}\right)\frac{w_1\left( \frac{x-c_k}{c_k} \right)}{w_1(\bar{x}_1)} \\
&\geq ((\bar{x}_1+1)c_k)^{-B(\gamma-\delta)} \\
&\geq 2^{-B(\gamma-\delta)} x^{-B(\gamma-\delta)}.
\end{split}
\end{equation*}
Hence, from this and~\eqref{pjwh8888} the estimate in point~\textit{(v)} follows.
\end{proof}

\begin{proof}[Proof of point~\textit{(vi)} of Proposition~\ref{6tg4}]
Recalling the definition of~$\widetilde{u}$, we see that
\begin{eqnarray*}\widetilde{u}'(x) &=&
A C_{1,k+1}x^{-A-1} - \frac{\eta'\left(\frac{x-d_k}{d_k}\right)}{d_k}\left(C_2x^{-B}-C_{1,{k+1}}x^{-A}  \right)\\&&\qquad +
\eta\left(\frac{x-d_k}{d_k}\right)\left(BC_2x^{-B-1}-A C_{1,k+1}x^{-A-1}\right).
\end{eqnarray*}

By the facts that~$2\leq C_{1,k+1}\leq C_2$ (see Lemma~\ref{drao}) and~$B<A\leq2$ (recall~\eqref{bcuewoiyr8439t098765412345asdfg}),
we see that
\begin{equation*}
C_2x^{-B}-C_{1,k+1}x^{-A} = C_2x^{-B}\left( 1- \frac{C_{1,k+1}}{C_2}x^{B-A}\right)>0.
\end{equation*}
Therefore, using also that~$\eta'<0$,
\begin{equation*}
\begin{split}
\widetilde{u}'(x) &\leq
A C_{2}x^{-B-1} +\frac{C_2\bar{\eta}}{d_k}x^{-B} +B C_2x^{-B-1}\\&\leq
A C_{2}x^{-B-1} + 2C_2\bar{\eta}x^{-B-1} +B C_2x^{-B-1}\\&\leq
2C_2(2+\bar{\eta})x^{-B-1},
\end{split}
\end{equation*} from which the upper bound in~\textit{(vi)} follows.

Moreover, we can write
\begin{equation*}
\begin{split}
\widetilde{u}'(x) &= AC_{1,k+1} \left(1-\eta\left(\frac{x-d_k}{d_k}\right)\right)x^{-A-1} + \frac{\left|\eta'\left(\frac{x-d_k}{d_k}\right)\right|}{d_k}\left(C_2x^{-B}-C_{1,{k+1}}x^{-A}  \right) \\&\qquad+BC_2\eta\left(\frac{x-d_k}{d_k}\right)x^{-B-1}\\&\geq
 AC_{1,k+1} \left(1-\eta\left(\frac{x-d_k}{d_k}\right)\right)x^{-A-1} +BC_2\eta\left(\frac{x-d_k}{d_k}\right)x^{-B-1}\\&\geq
 2B \left(1-\eta\left(\frac{x-d_k}{d_k}\right)\right)x^{-A-1} +2B\eta\left(\frac{x-d_k}{d_k}\right) x^{-A-1}\\&=
 2B x^{-A-1},
\end{split}
\end{equation*}
which yields also the desired lower bound.
\end{proof}

\begin{proof}[Proof of Proposition~\ref{03uhe3}]
The proof of Proposition~\ref{03uhe3} is analogous to that of Proposition~\ref{6tg4} and can be obtained by following the same reasoning with the appropriate modifications.

The interested reader may retrace the same steps, with the following substitutions: use ~\eqref{gf6e} in place of~\eqref{9dggfye}, \eqref{f4453} instead of~\eqref{syhvyuf4}, \eqref{0gdcxui} for~\eqref{0gbvctfdshgdcxui}, \eqref{05y6647yf} in place of~\eqref{04yf7y5}, \eqref{2vio7t} instead of~\eqref{vio7t}, \eqref{vejet} for~\eqref{jeoft1}, \eqref{eb59fw83} in place of~\eqref{wb5e983f} and ~\eqref{8fbe93} instead of~\eqref{be983f}.
\end{proof}

We now collect some properties of~$\widetilde{u}$ that will be used
in the proof of Theorem~\ref{optimal}. To this aim, we recall the setting introduced in~\eqref{riga927} and
the definition of~$\widetilde{u}$ in~\eqref{riga927}.

\begin{prop}\label{qusicon}
The function~$\widetilde{u}$ belongs to~$\mathcal{X}$.

Also, $\widetilde{u} \in C^{1+2s+\theta}(\R)$ for any~$\theta \in (0,1)$ and~$\widetilde{u}'>0$ in~$\R$.

In addition, 
\begin{equation}\label{astes}
\textit{ $\widetilde{u}$ satisfies the asymptotic estimates in~\eqref{asymp_decay}, \eqref{398r7gree}, \eqref{asymp_decay_lowbound} and~\eqref{eq:asymp-derivata}.}
\end{equation}

Also, there exist two positive constants~$\widetilde{C}$ and~$\widehat{C}$ such that for any~$x\leq -a_0$
\begin{equation}\label{pplklk}
\widetilde{C} \, (1 + \widetilde{u}(x))^{\alpha - 2} \leq \frac{|x|^{-1-2s}}{\widetilde{u}'(x)} \leq \widehat{C} \, (1 + \widetilde{u}(x))^{\beta - 2} 
\end{equation}
and for any~$x\geq a_0$
\begin{equation}\label{pplkplkk}
\widetilde{C} (1-\widetilde{u}(x))^{\gamma-2}\leq \frac{|x|^{-1-2s}}{\widetilde{u}'(x)}\leq \widehat{C} (1-\widetilde{u}(x))^{\delta-2}.
\end{equation}

Furthermore, for any~$x\geq a_0$ there exists~$\epsilon>0$ such that
\begin{equation}\label{qzapdff}
|\widetilde{u}''(x)| \leq \widehat{C} \min \left\{ |x|^{-2-\epsilon}, \widetilde{u'}(x)|x|^{-1+ B(\gamma-\delta)} \right\}.
\end{equation}
Similarly, for any~$x\leq- a_0$ there exists~$\epsilon>0$ such that
\begin{equation}\label{qzapdff2}
|\widetilde{u}''(x)|\leq \widehat{C} \min \left\{ |x|^{-2-\epsilon}, \widetilde{u'}(x)|x|^{-1+ E(\alpha-\beta)} \right\}.
\end{equation}
\end{prop}

\begin{proof}
The function~$\widetilde{u}$ belongs to~$\mathcal{X}$ by construction. 
Also, $\widetilde{u} \in C^{3,1}(\R)$ (see Section~\ref{constru}), so that in particular~$\widetilde{u} \in C^{1+2s+\theta}(\R)$ for any~$\theta \in (0,1)$.

Next, we verify the monotonicity property. We observe that~$\widetilde{u}'>0$ on~$[-a_0,a_0]$ and~$[-b_k,-a_k] \cup [a_k,b_k]$ by construction. The positivity of~$\widetilde{u}'$ on the whole~$\R$ then follows from points~\textit{(iii)}, \textit{(v)}, and~\textit{(vi)} of Propositions~\ref{6tg4} and~\ref{03uhe3}.

We now focus
our attention on the asymptotic estimates in~\eqref{astes}.

Since~$\phi_k \in [B,A]$ and~$\psi_k \in [E,D]$, we can apply points~\textit{(i)} and~\textit{(iv)} of Propositions~\ref{6tg4} and~\ref{03uhe3} to deduce~\eqref{asymp_decay} and~\eqref{asymp_decay_lowbound} for~$|x| \in [b_k,d_k]$. Outside this interval, the same bounds follow directly from the construction of~$\widetilde{u}$.

Moreover, Lemma~\ref{ctmosd7r} yields~\eqref{398r7gree} and~\eqref{eq:asymp-derivata} for~$|x| \in [a_k,b_k]$. These estimates extend to~$|x| \in [b_k,c_k] \cup [d_k,a_{k+1}]$ thanks to Lemma~\ref{ctmosd7r}, together with points~\textit{(ii)}, \textit{(iii)}, and~\textit{(vi)} of Propositions~\ref{6tg4} and~\ref{03uhe3}.

Finally, in the intervals~$[-d_k,-c_k] \cup [c_k,d_k]$, the upper bound in~\eqref{eq:asymp-derivata} follows directly from the construction, while the lower bound in~\eqref{398r7gree} is ensured by point~\textit{(v)} of Propositions~\ref{6tg4} and~\ref{03uhe3}.

This completes the proof~\eqref{astes}.  

Hence, we now establish the estimates in~\eqref{pplklk}, \eqref{pplkplkk}, \eqref{qzapdff} and~\eqref{qzapdff2}.
We point out that
the proof of~\eqref{pplkplkk} is analogous to that of~\eqref{pplklk}
and the one of~\eqref{qzapdff2} is analogous to that of~\eqref{qzapdff},
therefore we will only show~\eqref{pplklk} and~\eqref{qzapdff}.

We claim that
\begin{equation}\label{p003jjj}
\mbox{the inequalities in~\eqref{pplklk} and~\eqref{qzapdff} hold for any} \ x\in [a_k,b_k].
\end{equation}
To prove this, we recall that~$\widetilde{u}(x) = 1-C_{1,k}x^{-A}$ for~$x\in[a_k,b_k]$, hence~\eqref{qzapdff} follows in this interval noticing that~$\widetilde{u}'(x)= AC_{1,k} x^{-A-1}$ and that
\[ |\widetilde{u}''(x)|= A (A+1) C_{1,k} x^{-A-2} = (A+1) \widetilde{u}'(x) x^{-1}.\]

Moreover, noticing that
\begin{equation}\label{noylel}
A-2s = \frac{2s}{\delta-1}-2s = -\frac{2s}{\delta-1}(\delta-2)= -A(\delta-2),
\end{equation}
we infer that
\[
 \frac{x^{-1-2s}}{\widetilde{u}'(x)} =  (C_{1,k}A)^{-1} x^{A-2s} = (C_{1,k}A)^{-1} x^{ -A(\delta-2)}=   C^{1-\delta}_{1,k}A^{-1}(1-\widetilde{u}(x))^{\delta-2}.
\]
Since~$\delta<\gamma$, this entails~\eqref{p003jjj}, as desired.

Next, we verify that
\begin{equation}\label{p1gzrjj}
\mbox{the inequalities in~\eqref{pplklk} and~\eqref{qzapdff} hold for any} \ x\in [b_k,c_k].
\end{equation}
To this end, we rely on Proposition~\ref{6tg4}\textit{(i)}-\textit{(ii)}-\textit{(iii)}, from which we obtain that in~$[b_k,c_k]$, up to renaming~$\widehat{C}$ and~$\widetilde{C}$,
\[ \frac{x^{-1-2s}}{\widetilde{u}'(x)} \leq   \frac{\widehat{C}x^{-1-2s}}{x^{-1-2s+(\delta-2)\phi_k\left(\frac{x-b_k}{c_k-b_k}\right)}}=\widehat{C}  x^{-(\delta-2)\phi_k\left(\frac{x-b_k}{c_k-b_k}\right)}\leq \widehat{C} (1-\widetilde{u}(x))^{\delta-2} \]
and
\[   \frac{x^{-1-2s}}{\widetilde{u}'(x)}\geq\frac{ \widetilde{C}x^{-1-2s}}{x^{-1-2s+(\gamma-2)\phi_k\left(\frac{x-b_k}{c_k-b_k}\right)}} = \widetilde{C}x^{-(\gamma-2)\phi_k\left(\frac{x-b_k}{c_k-b_k}\right)} \geq \widetilde{C}(1-\widetilde{u}(x))^{\gamma-2} .\]
This yields~\eqref{pplklk} in~$[b_k,c_k]$. 

We now focus on establishing~\eqref{qzapdff} in~$[b_k,c_k]$.
To do this, we note that when~$x \in [2b_k, \frac{c_k}2]$, we have that~$\widetilde{u}= u_k$. Thus, we can use~\eqref{09876541qaz3edc6yhn9oil}, the fact that~$\phi_k\in[B,A]$ and~$\zeta\geq 1$ (see~\eqref{poiuytre09876543nhbvc}) to deduce that
\begin{equation}\label{heh}
\frac{|\widetilde{u}''(x)|}{\widetilde{u}'(x)} =\frac{|u_k''(x)|}{u_k'(x)} =\frac{1}{ x}  \left( 1+ \frac1{\zeta\ln x}+  \frac{\zeta-1}{\zeta} \phi_k\left(\frac{x-b_k}{c_k-b_k}\right)  \right) \leq \frac{ \widehat{C}}x.
\end{equation}
Also, by~\eqref{o30hf} and the fact that~$\phi_k \geq B$ we know that~$|\widetilde{u}''(x)| \leq C x^{-B-2}$, for some positive~$C$.

When instead~$x \in [b_k,2b_k]$, we rely on~\eqref{2938de}, on the lower estimates for~$\widetilde{u}'$ given in~\eqref{umiss} and~\eqref{umis}, on the estimate for~$\zeta$ in~\eqref{2938de} and on the fact that~$C_{1,k}\geq 2$ (see Lemma~\ref{drao}), to obtain for some~$C>0$ that
\begin{equation}\label{ghoo}
\widetilde{u}'(x) \geq C x^{-A-1}.
\end{equation}
Moreover, by the definition of~$\widetilde{u}$ given in Section~\ref{constru}, in this interval we have
\[
\widetilde{u}(x) = \widetilde{u}_k(x)= u_k(x)+  \eta\left(\frac{x-b_k}{b_k}\right)C_{1,k} \left(
x^{-\phi_k\left( \frac{x-b_k}{c_k-b_k}\right)}-x^{-A}\right),
\]
so that differentiating~$\widetilde{u}$ twice, we obtain
\[
\begin{aligned}
\widetilde{u}''(x) &= 
\eta''\left(\frac{x-b_k}{b_k}\right)\frac{C_{1,k}}{b_k^2}\left(x^{-\phi_k\left( \frac{x-b_k}{c_k-b_k}\right)}-  x^{-A}\right)  
+ \eta' \left(\frac{x-b_k}{b_k}\right)\frac{2}{b_k}\left( C_{1,k}A x^{-A-1}- u_k'(x) \right) \\
&\qquad
+ \left(1-\eta\left(\frac{x-b_k}{b_k}\right)\right)\,u_k''(x)
- \eta\left(\frac{x-b_k}{b_k}\right)C_{1,k}A(A+1)x^{-A-2}.
\end{aligned}\]
Hence, the estimates on~$u_k'$ and~$u_k''$ in Lemma~\ref{vaab4n}, the boundedness of~$C_{1,k}$ (see Lemma~\ref{drao}), the regularity of~$\eta$,~\eqref{be983f}, and the facts that~$b_k \geq \frac x2$ and~$\phi_k\leq A$ give
\begin{equation*}
 |\widetilde{u}''(x)| \leq C x^{-2-\phi_k\left(\frac{x-b_k}{c_k-b_k}\right)} =  C  x^{-2-A}x^{A-\phi_k\left(\frac{x-b_k}{c_k-b_k}\right)} \leq C  e^{\frac{2A\ln2}{\zeta}} x^{-A-2} ,
\end{equation*}
for some positive constant~$C$.

Combining this with~\eqref{ghoo}, we deduce that in~$[b_k,2b_k]$  
\begin{equation}\label{hehh}
|\widetilde{u}''(x)| \leq C \min \left\{x^{-A-2}, \widetilde{u}'(x) x^{-1}\right\}.
\end{equation}
Moreover, an analogous argument shows that in~$[\frac{c_k}2,c_k]$ it holds
\begin{equation*}
|\widetilde{u}''(x)| \leq C \min \left\{x^{-B-2}, \widetilde{u}'(x) x^{-1}\right\}.
\end{equation*}
This fact, together with~\eqref{heh} and~\eqref{hehh} leads to~\eqref{qzapdff} in~$[b_k,c_k]$, thus completing the proof of~\eqref{p1gzrjj}.

We next verify that
\begin{equation}\label{shg351erjj}
\mbox{the inequalities in~\eqref{pplklk} and~\eqref{qzapdff} hold for any} \ x\in [c_k,d_k].
\end{equation}

To do this, we first recall~\textit{(iv)} in Proposition~\ref{6tg4}, stating that 
\[ 1-u(x) \in  (\widetilde{C} x^{-B}, \widehat{C} x^{-B}) \quad\mbox{for any}\quad x \in [c_k,d_k],  \]
and observe that
\[B(\gamma-\delta+1)-2s = \frac{2s}{\gamma-1}(\gamma-\delta+1)-2s=\frac{2s}{\gamma-1}\left((\gamma-1)-(\delta-2)\right)-2s = - B(\delta-2) .\]

Hence, by relying on~\textit{(v)} in Proposition~\ref{6tg4}, we see that
\begin{equation*}
\frac{x^{-1-2s}}{\widetilde{u}'(x)}\leq  \widehat{C} \frac{x^{-1-2s}}{x^{-1-B(\gamma-\delta+1)}} =  \widehat{C} x^{B(\gamma-\delta+1)-2s}= \widehat{C} x^{-B(\delta-2)}\leq  \widehat{C} (1-\widetilde{u}(x))^{ \delta-2}.
\end{equation*}
Moreover, exploiting the fact that~$w_1$ is nonnegative (see Lemma~\ref{i03y3}), that~$C_2-C_{1,k}>0$ (see Lemma~\ref{drao}), and that~$B\leq 2$, we recall the computation of~$\widetilde{u}'$ in~\eqref{pjwh8888} and compute
\begin{equation*}
\begin{split}
\widetilde{u}'(x)&=
 x^{-B-1} \left( BC_2-(C_2-C_{1,k}) w_1\left(\frac{x-c_k}{c_k}\right) \right) \leq BC_2x^{-B-1} \leq  2C_2x^{-B-1}.
\end{split}
\end{equation*}

Therefore, using the fact that
\begin{equation}\label{08tgf}
B-2s=\frac{2s}{\gamma-1}-2s =2s\left(\frac1{\gamma-1} -1\right) = \frac{2s}{\gamma-1}(2-\gamma) = -B(\gamma-2),
\end{equation}
we obtain
\begin{equation*}
\frac{x^{-1-2s}}{\widetilde{u}'(x)} \geq \widetilde{C}\frac{x^{-1-2s}}{x^{-B-1}} = \widetilde{C}  x^{B-2s}=  \widetilde{C}  x^{-B(\gamma-2)} \geq \widetilde{C}(1-\widetilde{u}(x))^{\gamma-2},
\end{equation*}
and thus establish~\eqref{pplklk} in~$[c_k,d_k]$.

In addition, we recall the definition of~$\widetilde{u}$ in~$[c_k,d_k]$ given in Section~\ref{constru}, namely
\[ \widetilde{u}(x)=1- C_2x^{-B} + (C_2-C_{1,k})\widetilde{\eta}\left(\frac{x-c_k}{c_k}\right) x^{-B},\] 
which implies
\begin{eqnarray*}
&&\widetilde{u}''(x)
= -C_2 B(B+1) x^{-B-2} \\
&&\quad + (C_2 - C_{1,k})
\left(
\widetilde{\eta}'' \left(\frac{x-c_k}{c_k}\right)\frac{x^{-B}}{c_k^2}
- \widetilde{\eta}'\left(\frac{x-c_k}{c_k}\right)\frac{2 Bx^{-B-1}}{c_k}
+ \widetilde{\eta} \left(\frac{x-c_k}{c_k}\right)B(B+1)x^{-B-2}
\right).
\end{eqnarray*}
Hence, recalling that~$c_k\geq \frac x 2$ and the regularity of~$\widetilde{\eta}$, we deduce that
\begin{equation*}
|\widetilde{u}''(x)| \leq C x^{-B-2},
\end{equation*}
for some positive constant~$C$.

Moreover, by~\textit{(v)} in Proposition~\ref{6tg4} we see that~$\widetilde{u}'(x) \geq C x^{-1-B(\gamma-\delta+1)}$. Hence, it holds in~$[c_k,d_k]$
\begin{equation*}
\frac{|\widetilde{u}''(x)|}{\widetilde{u}'(x)} \leq C \frac{x^{-B-2}}{x^{-1-B(\gamma-\delta+1)}} = C x^{-1+B(\gamma-\delta)},
\end{equation*}
which yields~\eqref{shg351erjj}, as desired.

We now verify that
\begin{equation}\label{293elmdjj}
\mbox{the inequalities in~\eqref{pplklk} and~\eqref{qzapdff} hold for any} \ x\in [d_k,a_{k+1}].
\end{equation}
To check this claim, we notice that in~$[d_k,a_{k+1}]$, it holds by construction that
$$
\widetilde{C} x^{-A} \leq 1-\widetilde{u}(x)\leq \widehat{C}x^{-B}.
$$
Hence, by the fact that~$A-2s= -A(\delta-2)$ (see~\eqref{noylel}), that~$B-2s=-B(\gamma-2)$ (see~\eqref{08tgf}), and~\textit{(vi)} in Proposition~\ref{6tg4}, we obtain
\[\frac{x^{-1-2s}}{\widetilde{u}'(x)} \leq  \widehat{C} \frac{x^{-1-2s}}{x^{-A-1}}=\widehat{C}  x^{-A(\delta-2)}\leq  \widehat{C} (1-\widetilde{u}(x))^{\delta-2} \]
and
\[ \frac{x^{-1-2s}}{\widetilde{u}'(x)}\leq\widetilde{C} \frac{x^{-1-2s}}{x^{-B-1}}= \widetilde{C} x^{-B(\gamma-2)}\geq  \widetilde{C} (1-\widetilde{u}(x))^{\gamma-2}. \]
Hence, \eqref{pplklk} holds in~$[d_k,a_{k+1}]$.

In addition, we recall the definition of~$\widetilde{u}$ in this interval, given in Section~\ref{constru} as
\[
\widetilde{u}(x) =  1- C_{1,k+1} x^{-A}-   \eta \left( \frac{x-d_k}{d_k} \right) \left(C_2 x^{-B}- C_{1,k+1}x^{-A}\right) .
\]
Hence, we have that
\[
\begin{aligned}
\widetilde{u}''(x)
&= \frac{\eta''\big(\frac{x-d_k}{d_k}\big)}{d_k^2}\,\big(C_{1,k+1}x^{-A}-C_2 x^{-B}\big) \\
&\quad + 2\,\frac{\eta' \big(\tfrac{x-d_k}{d_k}\big)}{d_k}\,\big(C_2 B\,x^{-B-1}-C_{1,k+1}A\,x^{-A-1}\big) \\
&\quad -\left(1-\eta \left(\tfrac{x-d_k}{d_k}\right)\right)\,C_{1,k+1}A(A+1)\,x^{-A-2}
      -\eta\left(\frac{x-d_k}{d_k}\right)\,C_2B(B+1)\,x^{-B-2}.
\end{aligned}
\]

In particular, recalling the boundedness of~$C_{1,k}$ established in Lemma~\ref{drao}, that~$B\leq A$, and the fact that~$d_k \geq  \frac x 2$, it follows that
\begin{equation*}
|\widetilde{u}''(x)| \leq C x^{-B-2}.
\end{equation*}

Moreover, by Proposition~\ref{6tg4}~\textit{(vi)} we infer
\begin{equation*}
\frac{|\widetilde{u}''(x)|}{\widetilde{u}'(x)} \leq C\frac{x^{-B-2}}{x^{-A-1} }= C x^{A-B-1}.
\end{equation*}
This, together with the fact that
\begin{equation*}
A-B= \frac{2s}{\delta-1}-\frac{2s}{\gamma-1}= \frac{2s(\gamma-\delta)}{(\delta-1)(\gamma-1)}= \frac{B(\gamma-\delta)}{\delta-1}\leq B(\gamma-\delta),
\end{equation*}
shows~\eqref{qzapdff} in~$[d_k,a_{k+1}]$, yields~\eqref{293elmdjj}, and completes the proof.
\end{proof}

\begin{rem}
{\rm 
Proposition~\ref{qusicon} can be seen as ``sanity check'' for the function~$\widetilde{u}$. Namely, we verify here that~$\widetilde{u}$ displays the same qualitative features—both in terms of smoothness and asymptotic decay—as those required by Theorems~\ref{vogpian} and~\ref{optimal}. This ensures that our counterexample remains in line with the general theory and does not contradict the existing results.
}
\end{rem}

The next result provides both upper and lower bounds for the operator~$L_K\widetilde{u}'$, showing that it asymptotically behaves like~$|x|^{-1-2s}$.

\begin{prop}\label{4o0p9z}
Let~$K$ satisfy~\eqref{krn_symm} and~\eqref{newbound}. Then, there exist~$R>0$ and two positive constants~$C_1$ and~$C_2$, depending on ~$A$, $B$, $D$, $E$, $\lambda$ and~$\Lambda$, such that
\begin{equation*}
C_1|x|^{-1-2s}\leq L_K \widetilde{u}'(x) \leq C_2|x|^{-1-2s} \quad\mbox{for any } |x| \geq R.
\end{equation*}
\end{prop}

\begin{proof}
The desired estimates in Proposition~\ref{4o0p9z}
will be a consequence of Propositions~\ref{tonto} and~\ref{tontobis},
used here with~$\phi := \widetilde{u}'$ and~$\kappa := 2a_0$.

Indeed, by~\eqref{defur382tgjekf6pto37654}
we know that~$\widetilde{u}' \in C^{2,1}(\mathbb{R})$ and~$\widetilde{u}' > 0$ in~$(-a_0, a_0)$.
Moreover,
from Proposition~\ref{qusicon} we have that~$\widetilde{u}'$ satisfies the lower and upper bounds given, respectively, in~\eqref{398r7gree} and~\eqref{eq:asymp-derivata}. Hence  there exist~$ \epsilon_1$, $\epsilon_2 > 0$ such that, for any~$|x| \geq a_0$,
\begin{equation*}
C_2 |x|^{-1-\epsilon_1} \leq \widetilde{u}'(x) \leq C_1 |x|^{-1-\epsilon_2} ,
\end{equation*}
for some positive constants~$C_1$ and~$C_2$.
This says that the assumption in~\eqref{aggiuntaforse} is fulfilled.

To check that the assumption in~\eqref{7f9c7w3n} is satisfied,
we claim that there exists
constant~$C>0$, depending at most on~$A$, $B$, $D$ and~$E$, such that
\begin{equation}\label{eqref}
|\widetilde{u}'''(x)| \leq C |x|^{-B-3} \quad\mbox{for any }  x\geq a_0.
\end{equation}
To check this, we observe that, by the definition of~$\widetilde{u}$, for~$x\in[a_k,b_k]$,
$$ \widetilde{u}'''(x)=C_{1,k}A(A+1)(A+2) x^{-A-3},
$$ which implies~\eqref{eqref}.

Moreover, for~$x\in[2b_k, c_k/2]$ we use~\eqref{kofd} to conclude that
$$ |\widetilde{u}'''(x)|\leq C x^{-3-\phi_k\left(\frac{x-b_k}{c_k-b_k}\right)}\le C x^{-3-B},
$$
which entails~\eqref{eqref}.

Furthermore, for~$x\in[b_k,2b_k]$, we have that
\begin{eqnarray*}
\widetilde{u}(x)&=& \eta\left(\frac{x-b_k}{b_k}\right)C_{1,k}\left( x^{-\phi_k\left( \frac{x-b_k}{c_k-b_k}\right)} - x^{-A} \right) + u_k(x)\\
&=& \eta\left(\frac{x-b_k}{b_k}\right)\left(
C_{1,k}\left( x^{-\phi_k\left( \frac{x-b_k}{c_k-b_k}\right)} - x^{-A}+u_k(x)\right) \right) + \left(1-\eta\left(\frac{x-b_k}{b_k}\right)\right)
u_k(x)\\
&=&\eta\left(\frac{x-b_k}{b_k}\right)\left(  1- C_{1,k}x^{-A}\right) + \left(1-\eta\left(\frac{x-b_k}{b_k}\right)\right)
u_k(x).
\end{eqnarray*}
The aim is now to apply 
Lemma~\ref{7} to~$\widetilde{u}$ with~$a:=b_k$, $b:=2b_k$,
$f_1 := 1 - C_{1,k}x^{-A}$, $f_2 := u_k$ and~$\theta := \eta$.

To do this, we observe that,
by the definition of~$u_k$ and the facts that~$\phi_k \in [A,B]$ (see~\eqref{werecall078567654})
and~$C_{1,k}\leq C_2$ (see Lemma~\ref{drao}),
\begin{eqnarray*}&&|f_1(x)-f_2(x)|=
\big|1 - C_{1,k}x^{-A} - u_k(x)\big|=C_{1,k} \Big|x^{-\phi_k\left( \frac{x-b_k}{c_k-b_k}\right)}-x^{-A}\Big| \leq 2C_2 x^{-B}.
\end{eqnarray*}
Also, recalling 
the estimates in~\eqref{vio7t}, \eqref{o30hf} and~\eqref{kofd},
for all~$i =1,2,3$ we see that
\begin{eqnarray*}
\max\big\{ |f_2^{(i)}(x)|, |f_1^{(i)}(x)|\big\}  =
\max\left\{ \left|\prod_{j=1}^i (A+i-1)x^{-A-i}\right|, |u_k^{(i)}(x)|\right\} 
\leq \overline{C}x^{-B-i}. 
\end{eqnarray*}
Accordingly, the assumptions in~\eqref{932ygud} and~\eqref{667} are satisfied, and we obtain~\eqref{eqref} in~$[b_k,2b_k]$ as a consequence of Lemma~\ref{7}.

A similar argument can be performed in the interval~$[c_k/2, c_k]$
applying Lemma~\ref{7} to~$\widetilde{u}$
with~$a:=c_k$, $b:=2c_k$, $f_1 := 1 - C_{1,k} x^{-B}$, $f_2 := u_k$
and~$\theta := \eta$, using the fact that
$$|f_1(x)-f_2(x)|=\big| u_k(x) - 1 + C_{1,k} x^{-B}\big|= C_{1,k}\Big|x^{-B}-x^{-\phi_k\left( \frac{x-b_k}{c_k-b_k}\right)} \Big| \leq 2C_2 x^{-B}$$
together with the estimates in~\eqref{vio7t}, \eqref{o30hf} and~\eqref{kofd}.


An application of Lemma~\ref{7} to the intervals~$[c_k, d_k]$
(with~$a:=c_k$, $b:=2c_k$, $f_1 := 1 - C_{1,k} x^{-B}$, $f_2 := 1 - C_2 x^{-B}$ and~$\theta := \widetilde{\eta}$)
and~$[d_k, a_{k+1}]$ (with~$a:=d_k$, $b:=2d+k$,
$f_1 := 1 - C_2 x^{-B}$, $f_2 := 1 - C_{1,k+1} x^{-A}$ and~$\theta := \eta$) shows that~\eqref{eqref} also holds on these intervals.


The proof of~\eqref{eqref} is thereby complete.

In an analogous manner, one can show that
\begin{equation}\label{eqrefBIS}
|\widetilde{u}'''(x)| \leq C |x|^{-E-3} \quad\mbox{for any }  x\leq -a_0.
\end{equation}

Now, as a consequence of~\eqref{eqref} and~\eqref{eqrefBIS} we have that
\begin{eqnarray*}
&& x^3 \Vert \widetilde{u}'''(x)\Vert_{L^{\infty}(\frac{x}2,\frac{3x}2)} \leq C x^{-B} \quad\mbox{for all}\quad x\geq 2a_0
\\ {\mbox{and }} &&
|x|^3 \Vert \widetilde{u}'''(x)\Vert_{L^{\infty}(\frac{3x}2,\frac{x}2)} \leq C |x|^{-E} \quad\mbox{for all}\quad x\leq - 2a_0. \end{eqnarray*}
{F}rom these estimates we deduce that~\eqref{7f9c7w3n} holds true
with~$\phi=\widetilde{u}'$.

Gathering all the pieces of information, we conclude that
the assumptions of
Propositions~\ref{tonto} and~\ref{tontobis} are satisfied by taking~$\phi = \widetilde{u}'$ and~$\kappa := 2a_0$, and the desired result
thus follows.
\end{proof}

\section{Proof of Theorem~\ref{optimal}}\label{proptimal}

In this section, we provide the proof of Theorem~\ref{optimal}, which establishes the optimality of the estimates given in~\eqref{asymp_decay}, \eqref{398r7gree}, and~\eqref{asymp_decay_lowbound}. 
In light of Remark~\ref{23bve8670g7},
we will focus on the case~$\alpha \neq \beta$ and~$\gamma \neq \delta$.

The argument is structured as follows. To begin with, we consider a potential~$V$ suitably defined in terms of the function~$\widetilde{u}$ constructed in Section~\ref{constru}, in such a way that the Allen--Cahn equation in~\eqref{all_ca_frl} is satisfied. Then, Propositions~\ref{edoi4} and~\ref{pdio2} will confirm that the potential~$V$ satisfies assumptions~\eqref{pot_reg}, \eqref{pot_zero}
and~\eqref{pot_deg}. This ensures that our construction meets the required structural conditions. Finally, in Proposition~\ref{304waq} we will prove the optimality of the decay estimates by showing that they are sharp for the function~$\widetilde{u}$. That is, $\widetilde{u}$ attains the predicted decay rates precisely, thereby demonstrating that the bounds cannot be improved.
\medskip

We now dive into the technical details of this construction.
Let~$K:\R\to [0,+\infty]$ be the kernel of the fractional Laplacian of order~$s\in(0,1)$ in dimension~$n=1$, namely
\begin{equation*}
K(x):= |x|^{-1-2s}.
\end{equation*}
Accordingly, and in order to remain consistent with the notation introduced in~\eqref{bvtg}, we will henceforth denote the associated operator by
\begin{equation*}
L_s \widetilde{u}(x) := {\rm PV}_x\int_{\R} \frac{\widetilde{u}(y)-\widetilde{u}(x)}{|x-y|^{1+2s}}\,dy.
\end{equation*}
In addition, we stress that this kernel satisfies~\eqref{krn_symm}, \eqref{newbound} and \eqref{nuovissima} (see~\cite[Example~A.1]{DPDV}).


Given the function~$\widetilde{u}$ defined in Section~\ref{constru} 
(see in particular formula~\eqref{riga927}),
we define the functions
\begin{equation*}
\begin{split}
g(t) &:= L_s\widetilde{u}(t)  \quad {\mbox{for all }} t \in \mathbb{R}\\
{\mbox{and }}\quad h(r) &:= g\left( \widetilde{u}^{-1}(r)\right)  \quad {\mbox{for all }} r \in (-1,1)
\end{split}
\end{equation*}
and the potential
\begin{equation}\label{eilmprot}
V(r) := \int_{-1}^r h(\rho) \, d\rho \quad {\mbox{for all }} r \in (-1,1).
\end{equation}
In this way, $\widetilde{u}$ satisfies
\begin{equation}\label{ehcivol}
L_s\widetilde{u}(x) = V'(\widetilde{u}(x)) \quad\mbox{for all }x\in \R.
\end{equation}

We now check that the potential~$V$ given in~\eqref{eilmprot}
satisfies the desired assumptions.

\begin{prop}\label{edoi4}
It holds that $V \in C^{2,1}([-1,1])$.
\end{prop}

\begin{proof}
By~\eqref{defur382tgjekf6pto37654} we know that~$\widetilde{u} \in C^{3,1}(\R)$. Moreover, from Proposition~\ref{qusicon} we know that there exists~$\epsilon>0$ such that
\begin{equation*}
\widetilde{u}^{(\ell)}(x) \in \begin{cases}
(0,C  |x|^{-\epsilon-\ell}) &\text{for } x\leq -a_0 \text{ and } \ell =1,2, \\
(0, C  |x|^{-\epsilon-\ell} ) &\text{for } x\geq a_0 \text{ and } \ell =1,\\
( - C  |x|^{-\epsilon-\ell},0 ) &\text{for } x\geq a_0 \text{ and }\ell =2,
\end{cases}
\end{equation*}
for some positive~ constant~$C$.

Furthermore, using Lemma~\ref{70} together with the estimates in~\eqref{nier} and~\eqref{0h9gg}, we deduce that
\begin{eqnarray*}
&& |x|^3 \Vert \widetilde{u}^{(4)}\Vert_{L^{\infty}\left( \frac{3x}2,\,\frac x 2\right)} =O(|x|^{-\epsilon-1})\quad {\mbox{as }} x\to-\infty
\\
{\mbox{and }} &&
|x|^3 \Vert \widetilde{u}^{(4)}\Vert_{L^{\infty}\left(\frac x 2,\, \frac{3x}2\right)} =O(|x|^{-\epsilon-1}) \quad {\mbox{as }} x\to+\infty .
\end{eqnarray*}

Also, the properties of~$\widetilde{u}$ in~\eqref{defur382tgjekf6pto37654} give
\[\Vert \widetilde{u}' \Vert_{L^1(\R)} =\int_\R \widetilde{u}'(x)\, dx= \lim_{x\to+\infty} \widetilde{u}(x) - \lim_{x\to-\infty}\widetilde{u}(x)=2.\]
Hence, we can apply~\cite[Proposition~3.2]{DDV}, with the choices~$i:=2$,~$\kappa:=a_0$,~$\alpha:=\epsilon$, and~$\beta:=\epsilon$ and we obtain
\begin{equation}\label{leb}
\lim_{x \to \pm \infty} |x|^{2+2s} L_s \widetilde{u}''(x) 
=  \mp\frac{\Vert \widetilde{u}' \Vert_{L^1(\R)} \Gamma(2+2s)}{\Gamma(1+2s) } 
= \mp  2(1+2s) .
\end{equation}

In addition, the regularity of~$\widetilde{u}$ allows us to apply Proposition~\ref{3o93uh44}, which yields that
\[
L_s\widetilde{u}'= (L_s\widetilde{u})' 
\qquad\text{and}\qquad 
L_s\widetilde{u}'' = (L_s\widetilde{u})''.
\] 
Thus, differentiating the equation in~\eqref{ehcivol} with respect to~$x$, we obtain that
\begin{eqnarray}
V''(\widetilde{u}(x)) &=&\frac{L_s\widetilde{u}'(x)}{\widetilde{u}'(x)}\nonumber
\\ {\mbox{and }}\qquad
\label{ki0}
V'''(\widetilde{u}(x)) &=& \frac{L_s \widetilde{u}''(x)}{(\widetilde{u}'(x))^2} - \frac{L_s \widetilde{u}'(x)\,\widetilde{u}''(x)}{(\widetilde{u}'(x))^3}.
\end{eqnarray}

Thanks to Proposition~\ref{qusicon}, we know that
\begin{equation}\label{bvcnxqwioyr72348t67892345820}
\widetilde{u}'>0\qquad {\mbox{and}}\qquad
\lim_{x\to\pm\infty}\widetilde{u}(x)=\pm1.\end{equation}
Moreover, by Proposition~\ref{29ueh38} and the regularity of~$\widetilde{u}$, we know that~$L_s \widetilde{u}'$ and~$L_s \widetilde{u}''$ are continuous functions in~$\R$.
Therefore, we deduce from these considerations and~\eqref{ki0} that
\begin{equation}\label{093}
V''' \text{ is continuous in } (-1,1).
\end{equation}

Next, we prove that
\begin{equation}\label{09}
\lim_{\rho \to 1^-} V'''(\rho) = 0.
\end{equation}
To show this, 
we recall the definitions of~$A$ and~$B$ in~\eqref{bcuewoiyr8439t098765412345asdfg}
and we use Proposition~\ref{qusicon} to see that,
for sufficiently large~$x$,
\begin{equation}\label{236dfr75943gfakjvjhdkjmhsaFFGJPI:OI}
1-\widetilde{u}(x) \geq C x^{-A} \qquad
{\mbox{and}}\qquad \widetilde{u}'(x) \geq C x^{-1-B(\gamma-\delta+1)}
\end{equation}
(recall also formula~\eqref{9dggfye}
in Lemma~\ref{ctmosd7r} to check that~$A\le B(\gamma-\delta+1)$ and obtain the estimate above for~$\widetilde{u}'$).

Moreover, recalling the conditions on~$\gamma$ and~$\delta$ in~\eqref{edkis}, we compute
\begin{equation}\label{erpep}
\begin{split}&
B(\gamma-\delta+1)-s = \frac{2s(\gamma-\delta+1)}{\gamma-1} -s 
= s\left( \frac{2(\gamma-\delta+1)-\gamma+1}{\gamma-1}\right) \\&\qquad
= s\left( \frac{\gamma-2\delta+3}{\gamma-1}\right) 
= s\left( \frac{(\gamma-\delta)-(\delta-3)}{\gamma-1}\right)  \leq 0,
\end{split}
\end{equation}
 and
\begin{equation}\label{erpepp}
\begin{split}
&B(2+3(\gamma-\delta))-2s= \frac{2s}{\gamma-1}\left(2+3(\gamma-\delta)\right)-2s\\ &\qquad= 
-2s\left(1- \frac{2+3(\gamma-\delta)}{\gamma-1} \right)
=-2s\left(\frac{3\delta-2\gamma-3}{\gamma-1}
\right)\\&\qquad=-2s\left(\frac{2(\delta-\gamma)+(\delta-3)}{\gamma-1}\right)
\leq 0.
\end{split}
\end{equation}

Hence, up to renaming~$C$, from~\eqref{236dfr75943gfakjvjhdkjmhsaFFGJPI:OI}
and, respectively, \eqref{erpep} and~\eqref{erpepp}, we obtain that
\begin{equation}\label{cnsadmrui34y7}
\frac{x^{-2-2s}}{(\widetilde{u}'(x))^2} \leq  Cx^{2B(\gamma-\delta+1)-2s} 
= C x^{-A\frac{2s-2B(\gamma-\delta+1)}{A}} 
\leq C(1-\widetilde{u}(x) )^{\frac{2s-2B(\gamma-\delta+1)}{A}}
\end{equation}
and
\begin{equation}\label{cnsadmrui34y7ee}
x^{-2s+2B+3B(\gamma-\delta)} = x^{B(2+3(\gamma-\delta))-2s} = x^{-A\frac{2s-B(2+3(\gamma-\delta))}{A}}\leq C (1-\widetilde{u}(x))^{\frac{2s-B(2+3(\gamma-\delta))}{A}}.
\end{equation}


In light of the estimates in~\eqref{qzapdff} in Proposition~\ref{qusicon},
we also recall that for~$x\geq a_0$ we have
\begin{equation*} 
|\widetilde{u}''(x)| \leq \widehat{C}  \widetilde{u}'(x)|x|^{-1+ B(\gamma-\delta)} .
\end{equation*}
Accordingly, we deduce from
Proposition~\ref{4o0p9z}, the limit in~\eqref{leb} and the estimates in~\eqref{cnsadmrui34y7} and~\eqref{cnsadmrui34y7ee} that,
for~$x$ sufficiently large,
\[
\begin{split}
|V'''(\widetilde{u}(x))| &\leq\displaystyle \frac{|L_s \widetilde{u}'' (x)|}{(\widetilde{u}'(x))^2} 
+ \frac{|L_s \widetilde{u}'(x)|\, |\widetilde{u}''(x)|}{(\widetilde{u}'(x))^3}  \\&\leq
C\left( \frac{  x^{-2-2s}}{(\widetilde{u}'(x))^2} + \frac{ x^{-1-2s}x^{-1+B(\gamma-\delta)}}{x^{-2-2B(\gamma-\delta+1)}}\right)\\ &=
C\left( \frac{  x^{-2-2s}}{(\widetilde{u}'(x))^2} +x^{-2s+2B +3B(\gamma-\delta)}\right)\\ &\leq
C\left((1-\widetilde{u}(x) )^{\frac{2s-2B(\gamma-\delta+1)}{A}}+ (1-\widetilde{u}(x))^{\frac{2s-B(2+3(\gamma-\delta))}{A}} \right).
\end{split}
\]
{F}rom this and~\eqref{bvcnxqwioyr72348t67892345820}
we obtain~\eqref{09}.

An analogous argument shows that
\begin{equation}\label{0w}
\lim_{\rho \to -1^+} V'''(\rho) = 0.
\end{equation}
Combining~\eqref{093}, \eqref{09} and~\eqref{0w}, we conclude that~$V''$ is Lipschitz continuous in~$[-1,1]$, which completes the proof.
\end{proof}

\begin{prop}\label{pdio2}

The potential~$V$ satisfies~$V(-1)=V(1)=0$, \eqref{pot_zero} and~\eqref{pot_deg}.
\end{prop}

\begin{proof}
We recall that, by construction, the potential~$V$ satisfies~\eqref{ehcivol}, namely
\[ L_s\widetilde{u}  = V'(\widetilde{u} ) \quad\mbox{in } \R.\]
Moreover, from~\eqref{defur382tgjekf6pto37654} and~\eqref{riga927} we observe that~$\widetilde{u}$ is a layer solution, meaning that it satisfies
\[ \widetilde{u'}>0 \qquad\mbox{and}\qquad \lim_{x\to\pm\infty}\widetilde{u}(x)= \pm1. \]
Furthermore, Proposition~\ref{edoi4} ensures that~$V' \in C^{1,1}(\R)$. 

Then, we are in position to apply~\cite[Theorem~2.2]{CS14} and obtain that
\[V>V(-1)=V(1)=0 \quad\mbox{in} \ (-1,1).\]
As a consequence, \eqref{pot_zero} holds and we now focus on showing~\eqref{pot_deg}.

For this, we differentiate~\eqref{ehcivol} and we obtain that, for any~$x \in \R$,
\begin{equation}\label{bcnxiw98u43785941234567}
V''(\widetilde{u}(x)) = \frac{L_s \widetilde{u}'(x)}{\widetilde{u}'(x)}.
\end{equation}
Now, by Proposition~\ref{4o0p9z} we know that
\begin{equation*}
C_1|x|^{-1-2s}\leq L_s \widetilde{u}'(x) \leq C_2|x|^{-1-2s} \quad\mbox{for any } |x| \geq R,
\end{equation*}
for some positive constants~$C_1$ and~$C_2$.

Also, let~$\bar{k} \in \N$ be such that~$a_{\bar{k}} \geq R$. Then, relying on~\eqref{pplklk},~\eqref{pplkplkk} and~\eqref{bcnxiw98u43785941234567} we obtain 
\[\widetilde{C} \, (1 + \widetilde{u}(x))^{\alpha - 2} \leq V''(\widetilde{u}(x)) \leq \widehat{C} \, (1 + \widetilde{u}(x))^{\beta - 2} \quad\mbox{for any}\quad x\leq -a_{\bar{k}}, \]
and
\[\widetilde{C} \, (1 - \widetilde{u}(x))^{\gamma - 2} \leq V''(\widetilde{u}(x)) \leq \widehat{C} \, (1 - \widetilde{u}(x))^{\delta - 2}\quad\mbox{for any}\quad x\geq a_{\bar{k}}.\]
Therefore, setting
\[
\mu := \max\big\{ \widetilde{u}(a_{\bar{k}}),\, | \widetilde{u}(-a_{\bar{k}}) | \big\}
\]
completes the proof.
\end{proof}

\begin{prop}\label{304waq}
For any~$k \in \N$, it holds
\begin{equation}\label{nomric}
\widetilde{u}(a_k)= 1-C_{1,k}a_k^{-A}\quad\mbox{and}\quad \widetilde{u}(d_k)= 1-C_2d_k^{-B},
\end{equation}
and
\begin{equation}\label{nomrh}
\widetilde{u}(-a_k)= -1+C_{3,k}a_k^{-D}\quad\mbox{and}\quad \widetilde{u}(-d_k)= -1+C_4d_k^{-E}.
\end{equation}
Moreover,
\begin{equation}\label{egdelae}
\widetilde{u}'((\bar{x}_1+1)c_k)=\big((\bar{x}_1+1)c_k\big)^{-1-B(\gamma-\delta+1)}
\end{equation}
and
\begin{equation}\label{e4elae}
\widetilde{u}'(-(\bar{x}_2+1)c_k)=\big((\bar{x}_2+1)c_k\big)^{-1-E(\alpha-\beta+1)}.
\end{equation}
\end{prop}

\begin{proof}
Equations~\eqref{nomric} and~\eqref{nomrh} follow from the definition of~$\widetilde{u}$.

We recall the computations in~\eqref{pjwh8888}, according to which, for all~$x\in(c_k,d_k)$,
\begin{equation*}\label{sx1s4}
\widetilde{u}'(x) = x^{-B-1} \left( BC_2-(C_2-C_{1,k}) w_1\left(\frac{x-c_k}{c_k}\right) \right).
\end{equation*}
We use the fact that~$\bar{x}_1\in(0,1)$, thanks to Lemma~\ref{insicilia}, to infer that~$(\bar{x}_1+1)c_k\in(c_k, 2c_k)=(c_k,d_k)$.
Accordingly, we can write
$$ \widetilde{u}'((\bar{x}_1+1)c_k) = \big((\bar{x}_1+1)c_k\big)^{-B-1} \left( BC_2- (C_2-C_{1,k}) w_1(\bar{x}_1) \right).$$

We now observe that, by~\eqref{bvoiewyf83097u038012678},
\[ C_2-C_{1,k}=  \frac{BC_2-((\bar{x}_1+1)c_k)^{-B(\gamma-\delta)}}{w_1(\bar{x}_1)},\]
and therefore
\begin{equation*}
\widetilde{u}'((\bar{x}_1+1)c_k) = \big((\bar{x}_1+1)c_k\big)^{-1-B(\gamma-\delta+1)},
\end{equation*} which establishes~\eqref{egdelae}.

An analogous argument (considering~$w_2$ in place of~$w_1$ and~$\bar{x}_2$ instead of~$\bar{x}_1$) leads to~\eqref{e4elae} and completes the proof.
\end{proof}

\begin{proof}[Proof of Theorem~\ref{optimal}]
The existence of a potential~$V$ and a solution of~$L_s\widetilde{u}= V'(\widetilde{u})$ in~$\R$ is guaranteed by~\eqref{eilmprot} and~\eqref{ehcivol}.
{F}rom Proposition~\ref{qusicon} we also know that~$\widetilde{u}\in C^{3,1} (\R)$, $\widetilde{u}'>0$ in~$\R$ and
$$ \lim_{x\to\pm\infty}\widetilde{u}(x)=\pm1,$$
and therefore~$\widetilde{u}\in{\mathcal{X}}$.

The expressions in~\eqref{7843tgfbewb} are a consequence of Proposition~\ref{304waq}.
In particular, \eqref{nomric} and~\eqref{nomrh} take care of the optimality of the estimates in~\eqref{asymp_decay} and~\eqref{asymp_decay_lowbound},
while~\eqref{egdelae}
and~\eqref{e4elae} entail the optimality of the estimates in~\eqref{398r7gree}.
\end{proof}

\begin{appendix}

\section{Some useful bounds and auxiliary results to the construction of the function~$\widetilde{u}$}\label{somboun}

In this appendix we collect some useful lemmata that are used throughout the paper and
play a key technical role in various estimates.

We start with a simple observation that establish a relation between the exponents in
the estimates~\eqref{asymp2_vecchio} and~\eqref{eq:asymp-derivata}.

\begin{lemma}\label{lemmazzone2}
Let~$s \in (0,1)$, $\gamma\geq \delta\geq 2$ and~$\alpha\geq \beta\geq 2$. Then,
\begin{equation}\label{90u4edsibv5}
-1-2s \frac{\left( 1 -(\gamma-2)(\gamma-\delta)\right)}{\gamma-1} \geq  -1 - 2s \frac{(\delta-\gamma+1)}{\delta-1}.
\end{equation}
In particular, the equality  occurs only when either~$\gamma=\delta$ or~$\delta=2$.

Also, 
\begin{equation}\label{b9584}
-1-2s \frac{\left( 1 -(\alpha-2)(\alpha-\beta)\right)}{\alpha-1} \geq  -1 - 2s \frac{(\beta-\alpha+1)}{\beta-1},
\end{equation}
with the equality only when either~$\alpha=\beta$ or~$\beta=2$.
\end{lemma}

\begin{proof}We focus on the proof of~\eqref{90u4edsibv5}, since the one of~\eqref{b9584} is similar.

The case~$\gamma=\delta$ is trivial and equality holds in~\eqref{90u4edsibv5}. In the case~$\delta=2$, formula~\eqref{90u4edsibv5} is satisfied with equality as well, since
\begin{equation*}
\begin{split}
&-2s \frac{\left( 1 -(\gamma-2)(\gamma-\delta)\right)}{\gamma-1}= -\frac{2s}{\gamma-1}\left(1-(\gamma-2)^2\right) = -\frac{2s}{\gamma-1}
(\gamma-1)(3-\gamma)\\
&\qquad\qquad= -2s\left(3-\gamma\right) = - 2s \frac{(\delta-\gamma+1)}{\delta-1}.
\end{split}
\end{equation*}

We now suppose that~$\gamma>\delta> 2$ and we claim that
\begin{equation}\label{iuxste4}
(\gamma-1)(\delta-\gamma+1) > (\delta-1)\left( 1- (\gamma-\delta)(\gamma-2)  \right).
\end{equation}
In order to prove the claim, for any~$\delta > 2$ we define the function
\[ f_{\delta}(x) := (x-1)(\delta-x+1) -(\delta-1)\left( 1- (x-\delta)(x-2)  \right). \]
In particular, $f_{\delta}$ is a polynomial of second degree and can be expressed as
\[  f_{\delta}(x) = \left( \delta-2\right) \left( x^2-x(\delta+2)+2\delta \right). \]
The function~$f_{\delta}$ possesses two zeros, $x =2$ and~$x=\delta$.
Moreover, $f_\delta>0$ in~$(\delta,+\infty)$. 
Consequently, recalling that~$\gamma>\delta$, \eqref{iuxste4} follows.

{F}rom~\eqref{iuxste4} we obtain~\eqref{90u4edsibv5}, as desired.
\end{proof}

Our next goal here is to derive a number of bounds involving the quantities~$A$, $B$, $D$ and~$E$ defined in~\eqref{bcuewoiyr8439t098765412345asdfg}, as well as the functions~$\phi_k$ and~$\psi_k$ introduced in~\eqref{bvcig3433}.

The notation introduced in Section~\ref{notas} will be used throughout the rest of this section. 

\begin{lemma}\label{ctmosd7r}
For any~$x \in [0,1]$, 
\begin{equation}\label{9dggfye}
A \leq \min\left\{ B(\gamma - \delta+1), \, 2s - (\delta-2) \phi_k(x) \right\}
\end{equation}
and 
\begin{equation}\label{syhvyuf4}
B \geq \max\left\{ A (\delta-\gamma+1), \, 2s - (\gamma-2) \phi_k(x) \right\}.
\end{equation}

Also, for any~$x \in [0,1]$,
\begin{equation}\label{gf6e}
D \leq \min\left\{ E(\alpha - \beta+1), \, 2s - (\beta-2) \psi_k(x) \right\}
\end{equation}
and 
\begin{equation}\label{f4453}
E \geq \max\left\{ D (\beta-\alpha+1), \, 2s - (\alpha-2) \psi_k(x) \right\}.
\end{equation}
\end{lemma}

\begin{proof}
We will establish here the estimates in~\eqref{9dggfye} and~\eqref{syhvyuf4}, since the ones in~\eqref{gf6e} and~\eqref{f4453} are
proved analogously.

For this, we recall that~$B\leq \phi_k(x)\leq A$ for all~$x\in[0,1]$, thanks to~\eqref{werecall078567654}. Thus,
\begin{equation}\label{976548bvnct8urhvoewur16}
A = 2s - A(\delta-2) \leq 2s - (\delta-2) \phi_k(x).
\end{equation}
Also, we note that
\begin{equation*}
\begin{split}
&\frac{\gamma-\delta+1}{\gamma-1}- \frac1{\delta-1} = \frac{(\gamma-\delta+1)(\delta-1)- (\gamma-1)}{(\gamma-1)(\delta-1)}\\
&\qquad= \frac{\gamma\delta-\delta^2+2\delta-2\gamma}{(\gamma-1)(\delta-1)}= \frac{(\delta-2)(\gamma-\delta)}{(\gamma-1)(\delta-1)}\geq0.
\end{split}
\end{equation*}
Hence,
\begin{equation*}
 B(\gamma - \delta+1)-A = 2s \left( \frac{\gamma-\delta+1}{\gamma-1}- \frac1{\delta-1} \right)\geq0.
\end{equation*}
This and~\eqref{976548bvnct8urhvoewur16} yield~\eqref{9dggfye}.

To prove~\eqref{syhvyuf4} we argue similarly. Indeed,
\begin{equation}\label{bvncx64123456-0987654}
B = 2s - B(\gamma-2) \geq 2s -  (\gamma-2)\phi_k(x).
\end{equation}
Moreover, we have that
\begin{equation*}
\begin{split}&
\frac1{\gamma-1}-\frac{\delta-\gamma+1}{\delta-1} = \frac{\delta-1-(\gamma-1)(\delta-\gamma+1)}{(\gamma-1)(\delta-1)} \\
&\qquad= \frac{2\delta-2\gamma-\gamma\delta+\gamma^2}{(\gamma-1)(\delta-1)}= \frac{(\gamma-2)(\gamma-\delta)}{(\gamma-1)(\delta-1)}
\geq 0.
\end{split}
\end{equation*}
Therefore,
\begin{equation*}
 B -A (\delta-\gamma+1) = 2s \left( \frac1{\gamma-1}-\frac{\delta-\gamma+1}{\delta-1} \right)\geq0,
\end{equation*} which, together with~\eqref{bvncx64123456-0987654},
establishes~\eqref{syhvyuf4}.
\end{proof}

\begin{lemma}\label{pieb37698}
Let~$x \in[b_k,c_k]$ be such that
\begin{equation}\label{nvdtr63}
\ln x \geq \ln c_k \left( \frac{\gamma-1}{\delta-1}\right)^{-\zeta}.
\end{equation}
Then, 
\begin{equation}\label{nie75g8u}
-1-2s+(\gamma-2)\phi_k\left(\frac{x-b_k}{c_k-b_k}\right) \leq -1-A(\delta-\gamma+1).
\end{equation}

Also, let~$x \in[b_k,c_k]$ be such that
\begin{equation}\label{buft7873vu}
\ln x \leq \ln b_k \left( \frac{\delta-1}{\gamma-1}\right)^{-\zeta}.
\end{equation}
Then,
\begin{equation}\label{3d86rtu3y5h0}
-1-2s+(\delta-2)\phi_k\left(\frac{x-b_k}{c_k-b_k}\right) \geq -1-B(\gamma-\delta+1).
\end{equation}
\end{lemma}

\begin{proof}
We use~\eqref{mnbvcxz123456789poiuytrew} and~\eqref{nvdtr63} to see that
\begin{equation*}
\begin{split}&
B-2s+(\gamma-2)\phi_k\left( \frac{x-b_k}{c_k-b_k}\right) 
=-\frac{2s(\gamma-2)}{\gamma-1}+(\gamma-2)B\left(\frac{\ln c_k}{\ln x} \right)^{\frac1{\zeta}}\\&\qquad=
\frac{2s(\gamma-2)}{\gamma-1}\left(\left(\frac{\ln c_k }{\ln x} \right)^{\frac1{\zeta}}-1\right) 
\leq  \frac{2s(\gamma-2)(\gamma-\delta)}{(\gamma-1)(\delta-1)} = B-A(\delta-\gamma+1),
\end{split}
\end{equation*} which implies~\eqref{nie75g8u}.

The proof of~\eqref{3d86rtu3y5h0} is similar: we exploit~\eqref{mnbvcxz123456789poiuytrew} 
and~\eqref{buft7873vu} and we find that
\begin{equation*}
\begin{split}&
A-2s+(\delta-2)\phi_k\left(\frac{x-b_k}{c_k-b_k}\right) =-\frac{2s(\delta-2)}{\delta-1}+(\delta-2)A\left(\frac{\ln b_k}{\ln x}\right)^{\frac1{\zeta} }
\\&\qquad = -\frac{2s(\delta-2)}{\delta-1} \left( 1 - \left(\frac{\ln b_k}{\ln x}\right)^{\frac1{\zeta} }\right)
\geq -\frac{2s(\delta-2)(\gamma-\delta)}{(\delta-1)(\gamma-1)}
= A - B\left(\gamma-\delta+1\right).\qedhere
\end{split}
\end{equation*}
\end{proof}

An analogous argument leads to the following result, which may be viewed as the dual of Lemma~\ref{pieb37698}. We skip the proof, as it is essentially identical to the previous one.

\begin{lemma}\label{piave98}
Let~$x \in[b_k,c_k]$ be such that
\begin{equation}\label{63}
\ln x \geq \ln c_k \left( \frac{\alpha-1}{\beta-1}\right)^{-\xi}.
\end{equation}
Then,
\begin{equation}\label{n8u}
-1-2s+(\alpha-2)\psi_k\left(\frac{x-b_k}{c_k-b_k}\right) \leq -1-D(\beta-\alpha+1).
\end{equation}

Also, let~$x \in[b_k,c_k]$ be such that
\begin{equation}\label{bu}
\ln x  \leq \ln b_k  \left( \frac{\beta-1}{\alpha-1}\right)^{-\xi}.
\end{equation}
Then,
\begin{equation}\label{3y70}
-1-2s+(\beta-2)\psi_k\left(\frac{x-b_k}{c_k-b_k}\right) \geq -1-E(\alpha-\beta+1).
\end{equation}
\end{lemma}

The following inequalities follow from Lemmata~\ref{pieb37698} and~\ref{piave98}.

\begin{corol}\label{strcap} Let~$x\in[b_k,c_k]$.

If~$x$ satisfies~\eqref{nvdtr63}, then
\[ \min \left\{ x^{-1-A(\delta-\gamma+1)},\, x^{-1-2s+(\gamma-2)\phi_k\left( \frac{x-b_k}{c_k-b_k}\right)}\right\} = x^{-1-2s+(\gamma-2)\phi_k\left( \frac{x-b_k}{c_k-b_k}\right)}.\]

If~$x$ satisfies~\eqref{buft7873vu}, then
\[ \max\left\{x^{-1-B(\gamma-\delta+1)}, \, x^{-1-2s+(\delta-2)\phi_k\left( \frac{x-b_k}{c_k-b_k}\right)}\right\} = x^{-1-2s+(\delta-2)\phi_k\left( \frac{x-b_k}{c_k-b_k}\right)}.\]

If~$x$ satisfies~\eqref{63}, then
\[  \min \left\{ x^{-1-D(\beta-\alpha+1)},\, x^{-1-2s+(\alpha-2)\psi_k\left( \frac{x-b_k}{c_k-b_k}\right)}\right\} = x^{-1-2s+(\alpha-2)\psi_k\left( \frac{x-b_k}{c_k-b_k}\right)}. \]

If~$x$ satisfies~\eqref{bu}, then
\[ \max\left\{x^{-1-E(\alpha-\beta+1)}, \, x^{-1-2s+(\beta-2)\psi_k\left( \frac{x-b_k}{c_k-b_k}\right)}\right\} = x^{-1-2s+(\beta-2)\psi_k\left( \frac{x-b_k}{c_k-b_k}\right)}.\]
\end{corol}

\begin{lemma}\label{zyuxeri8595}
It holds
\begin{equation}\label{be983f}
x^{A- \phi_k\left( \frac{x-b_k}{c_k-b_k}\right)}  \leq e^{\frac{2A\ln2}{\zeta}} \quad\mbox{for any} \ x \in  [b_k,2b_k]
\end{equation}
and
\begin{equation}\label{wb5e983f}
x^{\phi_k\left( \frac{x-b_k}{c_k-b_k}\right)-B} \leq e^{\frac{2B\ln2}{\zeta}}\quad\mbox{for any} \ x \in [c_k/2,c_k].
\end{equation}

Similarly,
\begin{equation}\label{eb59fw83}
|x|^{\psi_k\left( \frac{|x|-b_k}{c_k-b_k}\right)-E} \leq e^{\frac{2E\ln2}{\xi}}\quad\mbox{for any} \ x \in [-c_k,-c_k/2 ]
\end{equation}
and
\begin{equation}\label{8fbe93}
|x|^{D- \psi_k\left( \frac{|x|-b_k}{c_k-b_k}\right)}  \leq e^{\frac{2D\ln2}{\xi}} \quad\mbox{for any} \ x \in  [-2b_k,-b_k].\qedhere
\end{equation}
\end{lemma}

\begin{proof}
Let~$\zeta$ be the quantity in~\eqref{0u9y9}, and let~$\alpha > 0$ and~$y \geq 2$. We apply the Mean Value Theorem to the function
\[
z \mapsto (1+z)^{\frac1\zeta} \quad \text{on} \quad I := \left[0, \frac{\alpha}{\ln y}\right].
\]
In this way, for some~$\ell \in I$,
\begin{equation}\label{46m58lPRE}
\left( 1+ \frac{\alpha}{\ln y} \right)^{\frac1\zeta} = 1 + \frac{\alpha(1+\ell)^{\frac1\zeta - 1}}{\zeta \ln y}.
\end{equation}
Also, recalling~\eqref{bvrikebgifuegwouw5636598jbvkdskj} and~\eqref{2938de}, we find that
\begin{equation}\label{poiuytre09876543nhbvc}
\zeta\ge\frac{32A\bar{\eta}}{B\eta_0}\ge\frac{32A}{B}=\frac{32(\gamma-1)}{\delta-1}>32.
\end{equation}
Therefore, we deduce from~\eqref{46m58lPRE} that
\begin{equation}\label{46m58l} \left( 1+ \frac{\alpha}{\ln y} \right)^{\frac1\zeta}\leq 1 + \frac{\alpha(1+\ell)^{\frac1{32} - 1}}{\zeta \ln y}
\leq 1 + \frac{\alpha}{\zeta \ln y}.\end{equation}

Now we take~$\alpha:=\ln2$ and~$y:=b_k$ and infer from~\eqref{mnbvcxz123456789poiuytrew} and~\eqref{46m58l} that, for any~$x \in [b_k,2b_k]$,
\begin{equation*}
\begin{split}
&A- \phi_k\left( \frac{x-b_k}{c_k-b_k}\right)  = A \left( 1 - \left(\frac{\ln b_k}{\ln x}\right)^{\frac1{\zeta}}\right)  \\
&\leq A \left( 1 - \left(\frac{\ln b_k}{\ln(2b_k)}\right)^{\frac1{\zeta}}\right)= A\left(\frac{\ln b_k}{\ln(2b_k)} \right)^{\frac1{\zeta}}\left(\left(\frac{\ln(2b_k)}{\ln b_k}\right)^{\frac1{\zeta}} -1\right)\\
&\leq A \left(\left( 1+ \frac{\ln2}{\ln b_k}\right)^{\frac1{\zeta}}  -1\right) \leq \frac{A\ln2}{\zeta \ln b_k}.
\end{split}
\end{equation*}
Consequently,
\begin{equation*}
x^{A- \phi_k\left( \frac{x-b_k}{c_k-b_k}\right)} \leq x^{\frac{A\ln2}{\zeta \ln b_k}} \leq x^{\frac{2A\ln2}{\zeta \ln(2b_k)}} \leq x^{\frac{2A\ln2}{\zeta \ln x}} = e^{\frac{2A\ln2}{\zeta}},
\end{equation*} which establishes~\eqref{be983f}.

Now, let~$x \in [c_k/2,c_k]$. In this case, we use~\eqref{46m58l} with~$\alpha:=2\ln2$ and~$y:=c_k$ and, recalling also~\eqref{mnbvcxz123456789poiuytrew}, we obtain that
\begin{equation*}
\begin{split}
&\phi_k\left( \frac{x-b_k}{c_k-b_k}\right)-B = B \left( \left( \frac{\ln c_k}{\ln x}\right)^{\frac1{\zeta}}-1 \right) \\
&\leq B \left( \left( \frac{\ln c_k}{\ln c_k-\ln2}\right)^{\frac1{\zeta}}-1 \right) =  B \left( \left( 1+\frac{\ln2}{\ln(c_k/2)}\right)^{\frac1{\zeta}}-1 \right)\\
&\leq B \left( \left( 1+\frac{2\ln2}{\ln c_k}\right)^{\frac1{\zeta}}-1 \right)\leq \frac{2B\ln2}{\zeta \ln c_k}.
\end{split}
\end{equation*}
As a consequence, observing that
\[ x^{\phi_k\left( \frac{x-b_k}{c_k-b_k}\right)-B} \leq x^{\frac{2B\ln2}{\zeta \ln c_k}} \leq   x^{\frac{2B\ln2}{\zeta \ln x}} =  e^{\frac{2B\ln2}{\zeta}}, \]
shows~\eqref{wb5e983f}. 

The proof of~\eqref{eb59fw83} and~\eqref{8fbe93} is analogous.
\end{proof}

We now focus on some auxiliary results involving the functions~$w_1$ and~$w_2$ defined in~\eqref{funcome}.

\begin{lemma}\label{insicilia} Let~$\bar{x}_1$ and~$\bar{x}_2$ be as in~\eqref{semil}.
Then, $\bar{x}_1 $, $\bar{x}_2 \in (0,1)$.
\end{lemma}

\begin{proof}
We prove only that~ $\bar{x}_1 \in (0,1)$, as the proof for~$\bar{x}_2$ is analogous.

First, observe that for any $x \in \mathbb{R}$
\[ x^{2} - 8x + 88 = (x - 4)^{2} + 72 > 0.\]
Therefore, the function
\[ f(x) := \frac{\sqrt{x^{2} - 8x + 88} - (x + 2)}{2(7 - x)} \]
is well defined on the interval~$[0,2]$.  

We show that
\begin{equation}\label{04j56}
f(x) > 0 \quad \text{for all } x \in [0,2].
\end{equation}
To prove this, note that for $x \in [0,2]$,
\[ x^{2} - 8x + 88 > x^{2} + 8 \geq x^{2} + 2x + 4 = (x + 2)^{2}. \]
Hence, the numerator of~$f(x)$ is positive for all~$x\in[0,2]$. Since the denominator is also positive in this interval, we infer~\eqref{04j56}.

Furthermore, we claim that
\begin{equation}\label{p34945}
f(x) < 1 \quad \text{for all } x \in [0,2].
\end{equation}
Indeed, for every $x \in [0,2]$,
\[ \sqrt{x^{2} - 8x + 88} < 10 \leq 16 - x = (x + 2) + 2(7 - x), \]
which entails~\eqref{p34945}.

Since~$B \in (0,2)$, from~\eqref{04j56} and~\eqref{p34945} we deduce that
$$  f(B) = \frac{\sqrt{B^{2} - 8B + 88} - (B + 2)}{2(7 - B)}\in(0,1).$$
The identity
\[ f(B) = \frac{3(4 - B) + \sqrt{B^{2} - 8B+ 88}}{2(7 - B)} - 1\]
thus leads to the desired result.
\end{proof}

\begin{lemma}\label{i03y3}
The functions~$w_1$ and~$w_2$ are nonnegative in~$[0,1]$.

Also,
\begin{equation}\label{jeoft1}
w_1(\bar{x}_1) > w_1(x) \quad\mbox{for any}\quad x \in [0,1] \setminus \{\bar{x}_1\}
\end{equation}
and
\begin{equation}\label{vejet}
w_2(\bar{x}_2) > w_2(x) \quad\mbox{for any}\quad x \in [0,1] \setminus \{\bar{x}_2\}.\qedhere
\end{equation}
\end{lemma}

\begin{proof}
We recall that~$\widetilde{\eta}'$ is nonpositive, thus the first statement follows by noticing that~$w_1$ and~$w_2$ are the sum of two positive contributions. 

We observe that~$\bar{x}_1$, $\bar{x}_2\in(0,1)$, by Lemma~\ref{insicilia} and we now prove~\eqref{jeoft1} (the proof of~\eqref{vejet} is analogous).
For this, we compute
\begin{equation}\label{0987654qwertyuilkjhgfds}
\begin{split}
w'_1(x) &= (B-1) \widetilde{\eta}'(x) -(x+1) \widetilde{\eta}''(x) \\
&=-\frac1{{\mathcal{B}}(4,4)}\left( (B-1)x^3(1-x)^3-3x^2(x+1)(1-x)^3+3x^3(x+1)(1-x)^2\right)
\\&= - \frac{x^2(1-x)^2}{{\mathcal{B}}(4,4)}\left((x+1)^2(7-B)+(x+1)(3B-12)+2(1-B)\right).
\end{split}
\end{equation}
As a result, recalling the definition of~$\bar{x}_1$ in~\eqref{semil}, we find
that the only root of~$w'_1$ in~$(0,1)$ is~$\bar{x}_1$. 

Furthermore, we see that~$w'_1>0$ in~$(0,\bar{x}_1)$ and~$w'_1<0$
in~$(\bar{x}_1,1)$. This and the continuity of~$w_1$ in~$[0,1]$
yield~\eqref{jeoft1}.
\end{proof}

\begin{lemma}\label{drao} 
The constants defined in~\eqref{bvoiewyf83097u038012678} and~\eqref{bvoiewyf83097u0380126782} satisfy
$$ C_2 >2,\qquad C_{1,k} \in (2, C_2),
\qquad C_4>2\qquad{\mbox{and}}\qquad C_{3,k} \in (2, C_4).$$
\end{lemma}

\begin{proof}
We observe that
\begin{equation*}
w_1(0)=B \widetilde{\eta}(0)- \widetilde{\eta}'(0)=B, \end{equation*}
and similarly~$w_2(0)=E$.
Therefore, Lemma~\ref{i03y3} ensures that~$B <w_1(\bar{x}_1)$ and~$E <w_2(\bar{x}_2)$. As a consequence, we obtain that~$C_2 >2$ and~$C_4>2$.

Also, by the definition of~$C_2$ in~\eqref{bvoiewyf83097u038012678}
we have that~$BC_2\ge2$, and thus, recalling that~$c_k>1$,
\begin{equation*}
BC_2 - ( (\bar{x}_1+1)c_k )^{-B(\gamma-\delta)}  \geq 2 - ( (\bar{x}_1+1)c_k )^{-B(\gamma-\delta)} > 0.
\end{equation*}
Moreover, since~$w_1(\bar{x}_1) > 0$ by Lemma~\ref{i03y3}, it follows that~$C_{1,k} < C_2$.  

By the definition of~$C_2$ in~\eqref{bvoiewyf83097u038012678}
we also see that~$C_2\ge2\left( 1 - \frac{B}{w_1(\bar{x}_1)} \right) ^{-1}$, and so 
\begin{equation*}
C_{1,k} > C_2 - \frac{BC_2}{w_1(\bar{x}_1)} = C_2 \left( 1 - \frac{B}{w_1(\bar{x}_1)} \right) \geq 2.
\end{equation*}
{F}rom these observations we conclude that~$C_{1,k} \in( 2,C_2)$.

A similar argument gives that~$C_{3,k} \in (2, C_4)$, and the proof is complete.
\end{proof}

\section{A pivotal ordinary differential equation}\label{parode}

In this section, we prove an existence and uniqueness result for an ODE that appears in Section~\ref{3sstzac}.

\begin{lemma}\label{khf869}
Let~$f_0\in\R$, $ \mu \in \R \setminus \{0\}$ and~$b> a>1$. Then, the ODE problem
\begin{equation}\label{poj09i}
\begin{cases}
f'(x)= -\displaystyle\frac{f(x)(b-a)}{\mu(x(b-a)+a) \ln\left( x(b-a)+a\right)}&\mbox{ for all } x \in (0,1),\\
f(0)=f_0
\end{cases}
\end{equation}
admits a unique solution given by
\begin{equation}\label{31lp0s36}
f(x)= \left( \frac{\ln a}{\ln( x(b-a)+a)}\right)^{\frac1{\mu}} f_0.
\end{equation}
\end{lemma}

\begin{proof}
We consider~$x \in (0,1]$ and integrate both sides of the equation in~\eqref{poj09i}, obtaining that
\begin{equation*}
\int_0^x \frac{f'(y)}{f(y)}\, dy= - \frac{b-a}{\mu} \int_0^x \frac{dy}{(y(b-a)+a) \ln( y(b-a)+a)}.
\end{equation*}
We change variables~$z:=f(y)$ in the left-hand side and~$z:=y(b-a)+a$ in the right-hand side. In this way, we get
\begin{equation*}
\int_{f_0}^{f(x)}\frac{dz}{z} = -\frac{1}{\mu} \int_{a}^{x(b-a)+a}\frac{dz}{z\ln z}.
\end{equation*}
Computing the integrals we thereby find that
\begin{equation*}
\ln\left(\frac{f(x)}{f_0}\right) = -\frac1{\mu}\left( \ln  (\ln(x(b-a)+a))
-\ln (\ln a)\right),
\end{equation*}
which entails~\eqref{31lp0s36}.

We are left to prove the uniqueness of the solution. For this, we consider the set
$$ D:= [0,1]\times [f_0-1, \, f_0+1]$$
and the function~${g:D  \to \R}$ defined as
\[ g(x,t) := \displaystyle\frac{t(b-a)}{\mu(x(b-a)+a) \ln( x(b-a)+a)} .\]
Since~$g$ is continuous in~$x$ and Lipschitz continuous in~$t$ with Lipschitz constant
\[ L:= \frac{b-a}{|\mu| a \ln a}, \]
the theory of ODEs yields the desired result.
\end{proof}

\section{Properties of~$u_k$ and its derivatives}\label{ukprop}
The aim of this section is to provide the main properties of the function~$u_k$, as defined in~\eqref{defofu}. We point out that the results presented here are used extensively in Section~\ref{n5m7b38}.

Specifically, Lemmata~\ref{y846tgffzasaa} and~\ref{j37863sss} characterize the first derivative of~$u_k$, while Lemma~\ref{vaab4n} also addresses its second, third and fourth derivatives.

We will use throughout this appendix the setting introduced in Section~\ref{notas}.

\begin{lemma}\label{y846tgffzasaa}
For any~$x \in [b_k,c_k]$, it holds that
\begin{equation}\label{cs542652}
u'_k(x)= C_{1,k} \frac{\zeta-1}{\zeta }\phi_k\left(\frac{x-b_k}{c_k-b_k}\right)x^{-1-\phi_k\left(\frac{x-b_k}{c_k-b_k}\right)}.
\end{equation}
Similarly, for any~$x\in[-c_k,-b_k]$, it holds that
\begin{equation}\label{c6tttge}
u'_k(x)= C_{3,k} \frac{\xi-1}{\xi }\psi_k\left(\frac{|x|-b_k}{c_k-b_k}\right)|x|^{-1-\psi_k\left(\frac{|x|-b_k}{c_k-b_k}\right)}.
\end{equation}

In particular, $u'_k>0$.
\end{lemma}

\begin{proof}
We observe that~$u_k(x)<1$ for any~$x \in [b_k,c_k]$, and accordingly, recalling that$${u_k(x) = 1- C_{1,k}x^{-\phi_k\left(\frac{x-b_k}{c_k-b_k}\right)}},$$ we have
\begin{equation*}
\ln(1-u_k(x))= \ln C_{1,k} -\phi_k\left(\frac{x-b_k}{c_k-b_k}\right) \ln x.
\end{equation*}
Thus, differentiating both sides and exploiting~\eqref{dyugffucdt-9-}, we infer that
\begin{equation*}
\begin{split}
\frac{u'_k(x)}{1-u_k(x)} &= \frac{\phi_k'\left(\frac{x-b_k}{c_k-b_k}\right)}{c_k-b_k}\ln x+ \frac1x \phi_k\left(\frac{x-b_k}{c_k-b_k}\right)\\
 &=- \frac{1}{\zeta x}\phi_k\left(\frac{x-b_k}{c_k-b_k}\right) + \frac1x \phi_k\left(\frac{x-b_k}{c_k-b_k}\right)\\
&= \frac{\zeta-1}{\zeta x}\phi_k\left(\frac{x-b_k}{c_k-b_k}\right),
\end{split}
\end{equation*} which entails~\eqref{cs542652}.

Likewise, differentiating~$u_k$ in~$[-c_k,-b_k]$ and recalling~\eqref{234de43232} lead to~\eqref{c6tttge}. 

Moreover, by~\eqref{poiuytre09876543nhbvc} we infer that~$\zeta > 1$ and Lemma~\ref{drao} proves that~$C_{1,k}>2$ for any~$k \in \N$. Therefore~$u'_k>0$ in~$[b_k,c_k]$, and a similar argument shows that~$u'_k>0$ in~$[-c_k,-b_k]$.
\end{proof}

\begin{lemma}\label{j37863sss}
For any~$x\in[b_k,c_k]$ it holds that
\begin{equation}\label{0gbvctfdshgdcxui}
u'_k(x)\leq 2C_2\min \left\{ x^{-1-A(\delta-\gamma+1)},\, x^{-1-2s+(\gamma-2)\phi_k\left( \frac{x-b_k}{c_k-b_k}\right)}\right\}
\end{equation}
and
\begin{equation}\label{04yf7y5}
u'_k(x) \geq B \max\left\{x^{-1-B(\gamma-\delta+1)}, \, x^{-1-2s+(\delta-2)\phi_k\left( \frac{x-b_k}{c_k-b_k}\right)}\right\}.
\end{equation}

Similarly, for any~$x \in [-c_k,-b_k]$ it holds that
\begin{equation}\label{0gdcxui}
u'_k(x)\leq 2C_4\min \left\{ |x|^{-1-D(\beta-\alpha+1)}, |x|^{-1-2s+(\alpha-2)\psi_k\left( \frac{|x|-b_k}{c_k-b_k}\right)}\right\}
\end{equation}
and
\begin{equation}\label{05y6647yf}
u'_k(x) \geq  E \max\left\{|x|^{-1-E(\alpha-\beta+1)}, \, |x|^{-1-2s+(\beta-2)\psi_k\left( \frac{|x|-b_k}{c_k-b_k}\right)}\right\}.
\end{equation}
\end{lemma}

\begin{proof}
{I}n the proof, we often rely on~$C_{1,k}\in (2,C_2)$ and on~$C_{3,k}\in (2,C_4)$, as shown in Lemma~\ref{drao}.

From~\eqref{cs542652} and~\eqref{werecall078567654} we infer that, for any~$x\in[b_k,c_k]$,
\[ u'_k(x) x^{1+ \phi_k\left(\frac{x-b_k}{c_k-b_k}\right)} 
= C_{1,k}\frac{\zeta-1}{\zeta }\phi_k\left(\frac{x-b_k}{c_k-b_k}\right)
\leq C_2 A.  \]
Also, using~\eqref{syhvyuf4} and again~\eqref{werecall078567654}, we see that
\begin{equation*}
\phi_k\left( \frac{x-b_k}{c_k-b_k}\right) - A(\delta-\gamma+1) \geq B- A(\delta-\gamma+1) \geq 0.
\end{equation*}
Hence, since~$b_k>1$,
\begin{equation}\label{be8768yg0hg}
u'_k(x) x^{1+ \phi_k\left(\frac{x-b_k}{c_k-b_k}\right)} \leq C_2 A\leq C_2  A x^{\phi_k\left( \frac{x-b_k}{c_k-b_k}\right) - A(\delta-\gamma+1)}
\leq  2C_2 x^{\phi_k\left( \frac{x-b_k}{c_k-b_k}\right) - A(\delta-\gamma+1)}.
\end{equation}

Furthermore, exploiting~\eqref{mnbvcxz123456789poiuytrew}, we find that
\begin{equation*}
(\gamma-1) \phi_k\left(\frac{x-b_k}{c_k-b_k}\right)-2s= (\gamma-1) B\left( \frac{\ln c_k}{\ln x} \right)^{\frac1{\zeta}}-2s = 2s \left(\left( \frac{\ln c_k}{\ln x} \right)^{\frac1{\zeta}} -1\right) \geq 0,
\end{equation*} which entails that
\begin{equation*}
u'_k(x) x^{1+ \phi_k\left(\frac{x-b_k}{c_k-b_k}\right)} \leq C_2A \leq 2C_2   x^{(\gamma-1)  \phi_k\left(\frac{x-b_k}{c_k-b_k}\right)-2s} .
\end{equation*}
This and~\eqref{be8768yg0hg} give~\eqref{0gbvctfdshgdcxui}.

We now show~\eqref{04yf7y5}. From~\eqref{9dggfye} and~\eqref{werecall078567654}, we find that
\[\phi_k\left( \frac{x-b_k}{c_k-b_k}\right)  - B(\gamma-\delta+1) \leq A-B(\gamma-\delta+1) \leq 0,\]
so that, recalling~\eqref{cs542652} and~\eqref{poiuytre09876543nhbvc},
\begin{equation}\label{n8rd35}
u'_k(x) x^{1+ \phi_k\left(\frac{x-b_k}{c_k-b_k}\right)}=C_{1,k}\frac{\zeta-1}{\zeta }\phi_k\left(\frac{x-b_k}{c_k-b_k}\right) \geq B\geq B x^{\phi_k\left( \frac{x-b_k}{c_k-b_k}\right)  - B(\gamma-\delta+1)}.
\end{equation}

In addition, observing that, by~\eqref{mnbvcxz123456789poiuytrew},
\begin{equation*}
(\delta-1)\phi_k\left(\frac{x-b_k}{c_k-b_k}\right) -2s =(\delta-1)A\left( \frac{\ln b_k}{\ln x} \right)^{\frac1{\zeta}} -2s =  2s \left(\left( \frac{\ln b_k}{\ln x}\right)^{\frac1{\zeta}}-1\right)\leq 0,
\end{equation*}
we see that
\begin{equation*}
u'_k(x) x^{1+ \phi_k\left(\frac{x-b_k}{c_k-b_k}\right)} \geq B \geq  B x^{(\delta-1)  \phi_k\left(\frac{x-b_k}{c_k-b_k}\right)-2s}.
\end{equation*}
Combining this with~\eqref{n8rd35} we obtain~\eqref{04yf7y5}, as desired.

The proof of~\eqref{0gdcxui} and~\eqref{05y6647yf} can be done
in an analogous manner, relying on~\eqref{c6tttge}, \eqref{gf6e} and~\eqref{f4453}.
\end{proof}

\begin{lemma}\label{vaab4n}
For any~$x\in[b_k,c_k]$ it holds that
\begin{eqnarray}
&&u'_k(x) \in \left(0 , 2 C_2x^{-1-\phi_k\left(\frac{x-b_k}{c_k-b_k}\right)}\right),\label{vio7t}\\
&&u''_k(x) \in \left(-8C_2x^{-2-\phi_k\left(\frac{x-b_k}{c_k-b_k}\right)}, 0\right),\label{o30hf}\\
&&u'''_k(x)\in \left(0, 46C_2x^{-3-\phi_k\left(\frac{x-b_k}{c_k-b_k}\right)}\right),\label{kofd}\\
{\mbox{and }}&&u''''_k(x)\in \left(-344C_2  x^{-4-\phi_k\left(\frac{x-b_k}{c_k-b_k}\right)},0\right).\label{nier}
\end{eqnarray}

Similarly, for any~$x\in[-c_k,-b_k]$ it holds that
\begin{eqnarray}
&&u'_k(x) \in \left(0 , 2C_4|x|^{-1-\psi_k\left(\frac{|x|-b_k}{c_k-b_k}\right)}\right),\label{2vio7t}\\
&&u''_k(x) \in\left(0, 8C_4|x|^{-2-\psi_k\left(\frac{|x|-b_k}{c_k-b_k}\right)}\right), \label{22o30hf}\\
&&u'''_k(x)\in \left(0, 46C_4|x|^{-3-\psi_k\left(\frac{|x|-b_k}{c_k-b_k}\right)}\right),\label{2kofd}\\
\mbox{and }&& u''''_k(x)\in \left(0, 344C_4|x|^{-4-\psi_k\left(\frac{|x|-b_k}{c_k-b_k}\right)}\right).\label{0h9gg}
\end{eqnarray}
\end{lemma}

\begin{proof} We will prove only the estimates in~$[b_k,c_k]$, since the ones in~$[-c_k, -b_k]$ are similar.

The estimate in~\eqref{vio7t} is an immediate consequence of~\eqref{cs542652}, recalling that~$\zeta>1$,~${\phi_k \leq A\leq 2}$ and that~$C_{1,k}\leq C_2$ (see Lemma~\ref{drao}).

Now, differentiating~\eqref{cs542652}, and recalling the equation for~$\phi_k$ in~\eqref{dyugffucdt-9-},
we obtain that
\begin{equation}\label{09876541qaz3edc6yhn9oil}\begin{split}
u_k''(x) &= -\frac{u_k'(x)}{ x}\left(1+\phi_k \left(\frac{x-b_k}{c_k-b_k}\right)+\frac{\phi_k'\left(\frac{x-b_k}{c_k-b_k}\right) x\ln x}{c_k-b_k}
-\frac{\phi_k'\left(\frac{x-b_k}{c_k-b_k}\right) x}{(c_k-b_k)\phi_k \left(\frac{x-b_k}{c_k-b_k}\right)}\right)\\&=
-\frac{u_k'(x)}{ x}\left(1+\phi_k \left(\frac{x-b_k}{c_k-b_k}\right)-\frac{\phi_k\left(\frac{x-b_k}{c_k-b_k}\right) }{\zeta}
+\frac{1}{\zeta\ln x}\right)\\&=-\frac{u_k'(x)}{ x}  \left( 1+ \frac1{\zeta\ln x}+  \frac{\zeta-1}{\zeta} \phi_k\left(\frac{x-b_k}{c_k-b_k}\right)  \right).
\end{split}\end{equation}
The estimate in~\eqref{o30hf} is therefore a consequence of~\eqref{vio7t}.

Now we differentiate~\eqref{09876541qaz3edc6yhn9oil} and we exploit the equation for~$\phi_k$ in~\eqref{dyugffucdt-9-}
to find that
\begin{eqnarray*}
u_k'''(x) &=&\frac{u_k'(x)-xu_k''(x)}{ x^2}  \left( 1+ \frac1{\zeta\ln x}+  \frac{\zeta-1}{\zeta} \phi_k\left(\frac{x-b_k}{c_k-b_k}\right)  \right)\\&&
-\frac{u_k'(x)}{ x}  \left(- \frac1{\zeta x\ln^2 x}+  \frac{\zeta-1}{\zeta(c_k-b_k)} \phi'_k\left(\frac{x-b_k}{c_k-b_k}\right)  \right)\\&=&
-\frac{u_k''(x)}{ x}  \left( 1+ \frac1{\zeta\ln x}+  \frac{\zeta-1}{\zeta} \phi_k\left(\frac{x-b_k}{c_k-b_k}\right)  \right)\\&&+
\frac{u_k'(x)}{ x^2}  \left( 1+ \frac1{\zeta\ln x}+  \frac{\zeta-1}{\zeta} \phi_k\left(\frac{x-b_k}{c_k-b_k}\right)  
+ \frac1{\zeta \ln^2 x}+ \frac{\zeta-1}{\zeta^2\ln x}  \phi_k\left(\frac{x-b_k}{c_k-b_k}\right)\right)
\\&\leq&
-\frac{4u''_k(x)}{x}+\frac{4 u_k'(x)}{x^2}.
\end{eqnarray*}
As a consequence, \eqref{vio7t} and~\eqref{o30hf} imply~\eqref{kofd}.

Finally, \eqref{nier} can be obtained by computing
\begin{eqnarray*}
u_k''''(x) &=&\left(\frac{u''_k(x)}{x^2} -\frac{u'''_k(x)}{x}\right) \left(1+ \frac1{\zeta\ln x}+  \frac{\zeta-1}{\zeta} \phi_k\left(\frac{x-b_k}{c_k-b_k}\right) \right) \\&&
- \frac{u_k''(x)}{x} \left(-\frac{1}{x\zeta\ln^2x} + \frac{\zeta-1}{\zeta (c_k-b_k)} \phi'_k\left(\frac{x-b_k}{c_k-b_k}\right) \right)\\&&
+\left(\frac{u''_k(x)}{x^2}- \frac{2u'_k(x)}{x^3}\right) \left(1+ \frac1{\zeta\ln x}+  \frac{\zeta-1}{\zeta} \phi_k\left(\frac{x-b_k}{c_k-b_k}\right)  
+ \frac1{\zeta \ln^2 x}+ \frac{\zeta-1}{\zeta^2\ln x}  \phi_k\left(\frac{x-b_k}{c_k-b_k}\right)\right)\\&&
+\frac{u'_k(x)}{x^2}\bigg(-\frac1{\zeta x \ln^2x}+ \frac{\zeta-1}{\zeta(c_k-b_k)} \phi'_k\left(\frac{x-b_k}{c_k-b_k}\right) -\frac{2}{\zeta  x \ln^3x} -\frac{\zeta-1}{\zeta^2x\ln^2x}  \phi_k\left(\frac{x-b_k}{c_k-b_k}\right)\\ &&  \qquad\qquad\qquad+\frac{\zeta-1}{\zeta^2\ln(x)(c_k-b_k)}\phi'_k\left(\frac{x-b_k}{c_k-b_k} \bigg)    \right)\\ &=&
\left(\frac{u''_k(x)}{x^2} -\frac{u'''_k(x)}{x}\right) \left(1+ \frac1{\zeta\ln x}+  \frac{\zeta-1}{\zeta} \phi_k\left(\frac{x-b_k}{c_k-b_k}\right) \right)\\&&
+ \frac{u_k''(x)}{x^2} \left(+\frac{1}{\zeta\ln^2x} +\frac{\zeta-1}{\zeta^2 \ln(x)} \phi_k\left(\frac{x-b_k}{c_k-b_k}\right) \right)\\&&
+\left(\frac{u''_k(x)}{x^2}- \frac{2u'_k(x)}{x^3}\right) \left(1+ \frac1{\zeta\ln x}+  \frac{\zeta-1}{\zeta} \phi_k\left(\frac{x-b_k}{c_k-b_k}\right)  
+ \frac1{\zeta \ln^2 x}+ \frac{\zeta-1}{\zeta^2\ln x}  \phi_k\left(\frac{x-b_k}{c_k-b_k}\right)\right)\\&&
-\frac{u'_k(x)}{x^3}\bigg(\frac1{\zeta  \ln^2x}+\frac{\zeta-1}{\zeta^2 \ln(x)} \phi_k\left(\frac{x-b_k}{c_k-b_k}\right)  +\frac{2}{\zeta   \ln^3x} +\frac{\zeta-1}{\zeta^2\ln^2x}  \phi_k\left(\frac{x-b_k}{c_k-b_k}\right)\\ && \qquad\qquad\qquad+\frac{\zeta-1}{\zeta^3  \ln^2 x }\phi_k\left(\frac{x-b_k}{c_k-b_k}\right) \bigg)\end{eqnarray*}
and using~\eqref{vio7t}, \eqref{o30hf} and~\eqref{kofd}.
\end{proof}

\section{Smooth joining of two functions}\label{resudiffn}

We present two technical results concerning the smooth joining of two functions. They are used in the proof of Proposition~\ref{4o0p9z}.

\begin{lemma}\label{7}
Let~$a$, $b\in (0,+\infty)$ such that~$b\ge 2a$ and let~$\epsilon>0$. Also, let~$\theta \in C^3(0,1)$.

Moreover, we consider two functions~$f_1$ and~$f_2$ such that, for some~$\overline{C}>0$ and for any~$x \in (a,b)$, 
\begin{equation}\label{932ygud}
|f_1(x)-f_2(x)|\leq \overline{C} |x|^{-\epsilon}  
\end{equation}
and
\begin{equation}\label{667}
 \max\big\{ |f_1^{(i)}(x)|, |f_2^{(i)}(x)|\big\}  \leq \overline{C}x^{-i-\epsilon} \quad \mbox{ for any}\quad  i =1,2,3.
\end{equation}
 
Then, the function
\begin{equation*}
h(x) := \theta\left(\frac{x-a}{b-a}\right) f_1(x) + \left( 1- \theta\left(
\frac{x-a}{b-a}\right)\right) f_2(x)
\end{equation*}
satisfies
\begin{equation}\label{mnbvcx-0uhtryrfd87654}
| h'''(x)| \leq C  \sup_{\substack{ x\in(0,1)\\ i=0,1,2,3}} |\theta^{(i)}(x) |    x^{- 3-\epsilon} \quad\mbox{for any } x \in(a,b),
\end{equation} 
for some~$C>0$.
\end{lemma}

\begin{proof}
For the sake of readability, we set
\[ \bar{\theta}:= \sup_{\substack{ x\in(0,1)\\ i=0,1,2,3}} |\theta^{(i)}(x) | .\]
We point out that if~$\bar{\theta}=+\infty$ then the claim in~\eqref{mnbvcx-0uhtryrfd87654} is obvious, hence from now on we assume that~$\bar{\theta}\in(0,+\infty)$.

We set
\begin{equation*}
\begin{split}
I(x)&:=\theta''' \left( \frac{x-a}{b-a} \right)\frac{f_1(x)-f_2(x)}{(b-a)^3},\\
II(x)&:=3  \theta'' \left( \frac{x-a}{b-a} \right) \frac{f'_1(x)-f'_2(x)}{(b-a)^2},\\
III(x)&:=3  \theta' \left( \frac{x-a}{b-a} \right)\frac{f''_1(x)-f''_2(x)}{b-a} \\ {\mbox{and }}\quad
IV(x)&:= \theta \left( \frac{x-a}{b-a} \right)\big(f'''_1(x)-f'''_2(x) \big)
\end{split}
\end{equation*} and we observe that
\begin{equation}\label{87654vbcnxqu9eo32yihtrf325} h'''(x) =I(x) + II(x) +III(x)+ IV(x) + f'''_2(x).
\end{equation}

Since~$b\ge 2a$, we have that
\begin{equation}\label{2}
2( b-a) \geq x \quad\mbox{for any } x \in(a,b).
\end{equation}
Exploiting~\eqref{932ygud} and~\eqref{2} we have that
\begin{equation*}
|I(x)| \leq   \frac{\overline{C} \bar{\theta} x^{-\epsilon}}{(b-a)^3}\leq C \bar{\theta} x^{-\epsilon-3}.
\end{equation*}
Similarly, using~\eqref{667} and~\eqref{2} we obtain
\begin{equation*}
|II(x)| +|III(x)|+ |IV(x)|+|f'''_2(x)| \leq C  \bar{\theta}  x^{-\epsilon-3}.
\end{equation*}
Gathering these observations and recalling~\eqref{87654vbcnxqu9eo32yihtrf325}, we obtain the desired estimate.
\end{proof}

\begin{lemma}\label{70}
Let~$a$, $b\in (0,+\infty)$ such that~$b\ge 2a$ and let~$\epsilon>0$. Also, let~$\theta \in C^4(0,1)$.

Moreover, we consider two functions~$f_1$ and~$f_2$ such that, for some~$\overline{C}>0$ and for any~$x \in (a,b)$, 
\begin{equation}\label{S932ygud}
|f_1(x)-f_2(x)|\leq \overline{C} |x|^{-\epsilon}  
\end{equation}
and
\begin{equation}\label{S667}
 \max\big\{ |f_1^{(i)}(x)|, |f_2^{(i)}(x)|\big\}  \leq \overline{C}x^{-i-\epsilon} \quad \mbox{ for any}\quad  i =1,2,3,4.
\end{equation}
 
Then, the function
\begin{equation*}
h(x) := \theta\left(\frac{x-a}{b-a}\right) f_1(x) + \left( 1- \theta\left(
\frac{x-a}{b-a}\right)\right) f_2(x)
\end{equation*}
satisfies
\begin{equation}\label{Smnbvcx-0uhtryrfd87654}
| h''''(x)| \leq C  \sup_{\substack{ x\in(0,1)\\ i=0,1,2,3,4}} |\theta^{(i)}(x) |    x^{- 4-\epsilon} \quad\mbox{for any } x \in(a,b),
\end{equation} 
for some~$C>0$.
\end{lemma}

\begin{proof}
For the sake of readability, we set
\[ \bar{\theta}:= \sup_{\substack{ x\in(0,1)\\ i=0,1,2,34}} |\theta^{(i)}(x) | .\]
We point out that if~$\bar{\theta}=+\infty$ then the claim in~\eqref{Smnbvcx-0uhtryrfd87654} is obvious, hence from now on we assume that~$\bar{\theta}\in(0,+\infty)$.

We set
\begin{equation*}
\begin{split}
I(x)&:=\theta'''' \left( \frac{x-a}{b-a} \right)\frac{f_1(x)-f_2(x)}{(b-a)^4},\\
II(x)&:=6 \theta''\left( \frac{x-a}{b-a} \right)\frac{f''_1(x)-f''_2(x)}{(b-a)^2} ,\quad\\
III(x)&:=4  \theta''' \left( \frac{x-a}{b-a} \right)\frac{f'_1(x)-f'_2(x)}{(b-a)^3} , \quad\\
IV(x)&:=4 \theta' \left( \frac{x-a}{b-a} \right)\frac{f'''_1(x)-f'''_2(x)}{b-a} \\ {\mbox{and }}\quad
V(x)&:= \theta \left( \frac{x-a}{b-a} \right)\big(f''''_1(x)-f''''_2(x) \big)
\end{split}
\end{equation*} and we observe that
\begin{equation}\label{S87654vbcnxqu9eo32yihtrf325} h''''(x) =I(x) + II(x) +III(x)+ IV(x)+V(x) + f''''_2(x).
\end{equation}

Since~$b\ge 2a$, we have that
\begin{equation}\label{S2}
2( b-a) \geq x \quad\mbox{for any } x \in(a,b).
\end{equation}
Exploiting~\eqref{S932ygud} and~\eqref{S2} we have that
\begin{equation*}
|I(x)| \leq   \frac{\overline{C} \bar{\theta} x^{-\epsilon}}{(b-a)^4}\leq C \bar{\theta} x^{-\epsilon-4}.
\end{equation*}
Similarly, using~\eqref{S667} and~\eqref{S2} we obtain
\begin{equation*}
|II(x)| +|III(x)|+ |IV(x)|+|V(x)|+|f''''_2(x)| \leq C  \bar{\theta}  x^{-\epsilon-4}.
\end{equation*}
Gathering these observations and recalling~\eqref{S87654vbcnxqu9eo32yihtrf325}, we obtain the desired estimate.
\end{proof}

\end{appendix}

\end{document}